\documentclass[11pt,reqno]{amsart}
 
\usepackage{enumerate}
\usepackage{enumitem}
\usepackage{dsfont}
\usepackage{amsxtra}
\usepackage{amsmath}
\usepackage{amscd}
\usepackage{amssymb}
\usepackage{amsfonts}
\usepackage[all]{xy}
\usepackage{mathrsfs}
\usepackage{amsthm} 
\usepackage{color}
\usepackage{upgreek}
\usepackage{latexsym}
\usepackage[english]{babel} 
\usepackage{hyperref}
\usepackage{stmaryrd}
\usepackage{nicefrac}
\usepackage{euscript}
\usepackage{mathabx}
\usepackage[foot]{amsaddr}
\usepackage{amscd}
\usepackage[sort]{cite}           
\usepackage{xspace}            
\usepackage{fancyhdr}          
\usepackage{amsfonts}
\usepackage{color}
\usepackage{upgreek}
\usepackage{verbatim}
\usepackage{enumerate}
\usepackage{latexsym}
\usepackage{graphicx}
\usepackage{tikz}
\usepackage{todonotes}
\usepackage{setspace}
\usepackage{hyperref}
\usepackage{stmaryrd}
\usepackage{nicefrac}
\usepackage{euscript}
\usepackage{wasysym}
\usepackage{url}
\usepackage{nicefrac}
\usepackage{a4wide}
\usepackage{amscd}
\usepackage[pdftex]{pict2e}
\usepackage{curve2e}
\xyoption{all}

\newcommand{\arxiv}[1]{\href{http://arxiv.org/abs/#1}{\tt
    arXiv:\nolinkurl{#1}}}

\theoremstyle{plain}
\newtheorem{theorem}{Theorem}[section]
\newtheorem{lemma}[theorem]{Lemma}
\newtheorem{lemma-definition}[theorem]{Lemma-Definition}
\newtheorem{definition-lemma}[theorem]{Definition-Lemma}

\newtheorem{proposition}[theorem]{Proposition}
\newtheorem{conjecture}[theorem]{Conjecture}
\newtheorem{corollary}[theorem]{Corollary}

\newtheorem*{theorem-A}{Theorem A}
\newtheorem*{theorem-B}{Theorem B}
\newtheorem*{theorem-C}{Theorem C}
\newtheorem*{theorem-D}{Theorem D}
\newtheorem*{conjecture-A}{Conjecture A}
\newtheorem*{conjecture-B}{Conjecture B}
\newtheorem*{conjecture-C}{Conjecture C}
\newtheorem*{conjecture-D}{Conjecture D}

\theoremstyle{definition}

\theoremstyle{remark}
\newtheorem{remark}[theorem]{Remark}

\numberwithin{equation}{section}

\def\D{\mathrm{D}}
\def\E{\mathrm{E}}

\def\K{\mathrm{K}}
\def\L{\mathrm{L}}

\def\O{\mathrm{O}}

\def\U{\mathrm{U}}

\def\X{\mathrm{X}}

\def\bbA{\mathbb{A}}

\def\bbC{\mathbb{C}}
\def\bbD{\mathbb{D}}

\def\bbN{\mathbb{N}}

\def\bbZ{\mathbb{Z}}

\def\frakg{\mathfrak{G}}

\def\frakL{\mathfrak{L}}
\def\frakM{\mathfrak{M}}

\def\frakP{\mathfrak{P}}

\def\frakR{\mathfrak{R}}
\def\frakS{\mathfrak{S}}

\def\calA{\mathcal{A}}

\def\calC{\mathcal{C}}

\def\calE{\mathcal{E}}
\def\calF{\mathcal{F}}

\def\calH{\mathcal{H}}

\def\calL{\mathcal{L}}

\def\calO{\mathcal{O}}
\def\calP{\mathcal{P}}

\def\calU{\mathcal{U}}
\def\calV{\mathcal{V}}
\def\calW{\mathcal{W}}

\def\calZ{\mathcal{Z}}

\def\frakg{\mathfrak{g}}

\def\frakp{\mathfrak{p}}

\def\bfc{\mathbf{c}}

\def\bfv{{\mathbf{v}}}
\def\bfw{\mathbf{w}}

\def\bfC{\mathbf{C}}
\def\bfD{\mathbf{D}}

\def\bfK{\mathbf{K}}

\def\bfO{\mathbf{O}}

\def\uL{{\underline L}}

\def\uN{{\underline N}}
\def\uW{{\underline W}}
\def\uY{{\underline Y}}
\def\uZ{{\underline Z}}

\def\ua{{\underline a}}

\def\ux{{\underline x}}

\def\r{\rangle}

\def\k{{\operatorname{k}\nolimits}}

\def\Isom{\operatorname{Isom}\nolimits}

\def\tor{\operatorname{tor}\nolimits}

\def\Rep{{\frakR}}
\def\Coh{\operatorname{Coh}\nolimits}
\def\DCoh{{\operatorname{DCoh}\nolimits}}

\def\Perf{{\operatorname{Perf}\nolimits}}

\def\sg{{\operatorname{sg}\nolimits}}

\def\top{{\operatorname{top}\nolimits}}
\def\alg{{\operatorname{alg}\nolimits}}

\def\top{{\operatorname{top}\nolimits}}
\def\top{{\operatorname{top}\nolimits}}

\def\rk{{\operatorname{rk}\nolimits}}

\def\op{{{\operatorname{op}\nolimits}}}
\def\Ad{{{\operatorname{Ad}\nolimits}}}
\def\-{{\operatorname{-}\!}}

\def\Im{\operatorname{Im}\nolimits}
\def\Ker{\operatorname{Ker}\nolimits}
\def\Hom{\operatorname{Hom}\nolimits}

\def\End{\operatorname{End}\nolimits}

\def\Res{\operatorname{Res}\nolimits}

\def\Ext{\operatorname{Ext}\nolimits}

\def\id{\operatorname{id}\nolimits}

\def\Tr{\operatorname{Tr}\nolimits}
\def\Gr{{\operatorname{Gr}\nolimits}}

\def\crit{\operatorname{crit}\nolimits}

\def\Db{\operatorname{D^b\!}\nolimits}
\def\D{{\operatorname{D}\nolimits}}
\def\b{{\operatorname{b}\nolimits}}

\def\boxplus{{\oplus}}

\def\ev{\operatorname{ev}\nolimits}

\def\Quot{{\operatorname{Quot}\nolimits}}
\def\an{{\operatorname{an}\nolimits}}

\def\Tr{\operatorname{Tr}\nolimits}

\def\can{\operatorname{can}\nolimits}

\def\nil{{\operatorname{nil}\nolimits}}
\def\ch{\operatorname{ch}\nolimits}

\def\colim{\qopname\relax m{colim}}

\def\GL{\operatorname{GL}\nolimits}

\def\sgn{{o}}

\numberwithin{itemcounter}{subsection}
\numberwithin{equation}{section}
\usepackage[toc,page]{appendix}

\usepackage{etoolbox}
\appto\appendix{\addtocontents{toc}{\protect\setcounter{tocdepth}{1}}}

\title[Non-symmetric quantum loop groups and K-theory]
{Non-symmetric quantum loop groups and K-theory}

\author{M. Varagnolo$^1$} 
\address{\scriptsize{$^1$~CY Cergy Paris Universit\'e,  95302 Cergy-Pontoise, France,
UMR8088 (CNRS).}}
\author{E. Vasserot$^2$} 
\address{\scriptsize{$^2$~Universit\'e Paris Cit\'e, 75013 Paris, France, UMR7586 (CNRS), 
Institut Universitaire de France (IUF).
}}

\begin{document}
\maketitle

\begin{abstract}
We realize the quantum loop groups and shifted quantum loop groups of arbitrary types, 
possibly non-symmetric, using critical  K-theory. This generalizes the Nakajima construction of symmetric 
quantum loop groups via quiver varieties to non-symmetric types. 
We also give a new geometric construction of some simple 
modules of both quantum loop groups and shifted quantum loop groups.
\end{abstract}

\renewcommand{\abstractname}{R\'esum\'e}
\begin{abstract} Nous r\'ealisons les groupes quantiques affines et les groupes quantiques affines d\'ecal\'es de type arbitraire, 
\'eventuellement non sym\'etrique, en utilisant la K-th\'eorie critique. Cette construction g\'en\'eralise la construction de Nakajima des groupes  
quantiques affines sym\'etriques via les vari\'et\'es de carquois. Nous donnons \'egalement une nouvelle 
construction g\'eom\'etrique de certains modules simples des groupes quantiques affines et des groupes quantiques affines d\'ecal\'es. \end{abstract}

\tableofcontents

\section{Introduction and notation}

\subsection{Introduction}
The quiver varieties $\frakM(W)$, the graded quiver varieties $\frakM^\bullet(W)$,
and their Steinberg varieties $\calZ(W)$ and $\calZ^\bullet(W)$
were introduced by Nakajima in \cite{N94, N98} and \cite{N00}.
The K-theory of the Steinberg varieties, equipped with a convolution product,
yields a family of algebras $K(\calZ^\bullet(W))$ which are closely related
to the quantum loop groups $\U_\zeta(L\frakg)$ of symmetric types.
This algebra is important for the finite dimensional modules of $\U_\zeta(L\frakg)$ and their $q$-characters, see \cite{N04,N11}.
Nakajima's geometric realization of $\U_\zeta(L\frakg)$ in $K(\calZ^\bullet(W))$
does not extend to quantum groups of non-symmetric (finite or Kac-Moody) types.
How to generalize \cite{N00} is an old question.
An alternative approach has been suggested recently in \cite{NW23} using Coulomb branches.
The Coulomb realization also yields some interpretation of finite dimensional simple modules, 
but it is of a different nature from the construction here.
In this paper  we introduce a new family of convolution algebras attached to quiver varieties
with potentials, called critical convolution algebras. 
Here the K-theory is replaced by the critical K-theory, which is the Grothendieck group 
of the derived factorization categories attached to LG-models considered in
\cite{BFK13, EP15, H17a, H17b}. 
The critical K-theory depends on the choice of some function (the trace of the potential).
It is  supported  on the critical set of this function. 
We prove that the critical convolution algebra attached to the Steinberg variety
$\widetilde\calZ^\bullet(W)$ of the graded triple quiver variety $\widetilde\frakM^\bullet(W)$
 yields a
geometric realization of all non-twisted quantum loop groups and
of all non-twisted negatively shifted quantum loop groups for some well chosen potentials.
More precisely we prove the following theorem.

\begin{theorem}\label{thm:A}
Let $\frakg$ be any simple complex Lie algebra.
\hfill
\begin{enumerate}[label=$\mathrm{(\alph*)}$,leftmargin=8mm]
\item
There is  an algebra homomorphism
$\U_\zeta(L\frakg)\to 
K\big(\widetilde\frakM^\bullet(W)^2,(\tilde f^\bullet_\gamma)^{(2)}\big)_{\widehat\calZ^\bullet(W)}$
and representations of $\U_\zeta(L\frakg)$ on 
$K(\widetilde\frakM^\bullet(W),\tilde f^\bullet_\gamma)_{\widetilde\frakL^\bullet(W)}$ and
$K(\widetilde\frakM^\bullet(W),\tilde f^\bullet_\gamma)$.
\item
There is an algebra homomorphism
$\U_\zeta^{-w}(L\frakg)\to 
K\big(\widehat\frakM^\bullet(W)^2,(\widehat f^\bullet_2)^{(2)}\big)_{\widehat\calZ^\bullet(W)}$
and representations of $\U_\zeta^{-w}(L\frakg)$ on 
$K(\widehat\frakM^\bullet(W),\widehat f^\bullet_2)_{\widehat\frakL^\bullet(W)}$ and
$K(\widehat\frakM^\bullet(W),\widehat f^\bullet_2)$.
\end{enumerate}
\end{theorem}

\noindent Here  $\tilde f^\bullet_\gamma$, $\widehat f^\bullet_2$ are potentials on $\widetilde\frakM^\bullet(W)$, 
 $\widehat\frakM^\bullet(W)$, and $\widehat\frakM^\bullet(W)$ is a simply framed version of  
$\widetilde\frakM^\bullet(W)$. The algebra  $\U_\zeta^{-w}(L\frakg)$
is the shifted quantum loop group. Note that the 0-shifted quantum loop group is $\U_\zeta(L\frakg)$ up to some central elements.
The potential $\tilde f^\bullet_\gamma$ is a deformation of a potential
$\tilde f^\bullet_1$ on $\widetilde\frakM^\bullet(W)$ depending on a deformation parameter $\gamma$.
Part (a) of the theorem can be viewed as an extension of the construction in
\cite{N00} because:
\begin{enumerate}[label=$\mathrm{(\alph*)}$,leftmargin=8mm,itemsep=1.5mm]
\item[-]
for symmetric type, there is an isomorphism
of $\frakM^\bullet(W)$ with the critical set $\crit(\tilde f^\bullet_1)$ 
of a function $\tilde f^\bullet_1:\widetilde\frakM^\bullet(W)\to\bbC$ which will be defined later in the text,
\item[-]
in Proposition \ref{prop:crit1}, a dimension reduction yields an algebra and a module isomorphisms 
\begin{align*}
K\big(\widetilde\frakM^\bullet(W)^2,(\tilde f^\bullet_1)^{(2)}\big)_{\widetilde\calZ^\bullet(W)}=
K\big(\calZ^\bullet(W)\big),\quad
K\big(\widetilde\frakM^\bullet(W),\tilde f^\bullet_1\big)
=K\big(\frakM^\bullet(W)\big).
\end{align*}
\end{enumerate}
For symmetric types, the function $\tilde f_1^\bullet$ comes from some
cubic potential, while 
non-symmetric types require potentials of higher degrees.
The possibility to use potentials of arbitrary degree is an important property of
critical convolution algebras which has no analogue for the Nakajima convolution algebras.
The potential we use for non-symmetric types appears already in the work of
Hernandez-Leclerc and Geiss-Leclerc-Schr\"oer \cite{HL16, GLS17} on cluster algebras,
or in the work of Yang-Zhao \cite{YZ22} on cohomological Hall algebras.
We expect the general theory of Nakajima in \cite{N00}
to generalize to all types using critical cohomology and K-theory.
We will come back to this elsewhere.
In particular in the non-shifted case, we have the following conjecture, see the text below for the notation.

\begin{conjecture} For each $w\in\bbN I^\bullet$ the $\U_\zeta(\L\frakg)$-modules
$$K\big(\widetilde\frakM^\bullet(W),\tilde f^\bullet_1\big)
,\quad
K\big(\widetilde\frakM^\bullet(W),\tilde f^\bullet_1\big)_{\widetilde\frakL^\bullet(W)}$$
are isomorphic to the costandard module and the standard module with $\ell$-highest weight $\Psi_w$.
\end{conjecture}

\noindent
For fundamental modules, this conjecture follows from Theorem \ref{thm:B} below.

Another important point is that the
Nakajima  construction permits to recover the classification of the simple finite dimensional 
modules of quantum loop groups, but it does not give a geometric construction of those. 
More precisely, the  K-theory of quiver varieties yields a geometric realization 
of the standard modules,
and the simple modules are the Jordan H\"older constituents of the standards.
Remarkably, varying the potentials, the critical  K-theory also gives a realization of the 
simple modules in several settings:
we realize both 
all Kirillov-Reshetikhin modules of usual quantum loop groups and the prefundamental modules of shifted quantum loop groups 
as the critical  K-theory of some LG-models attached to quivers.
This construction is new already for symmetric types.
It was partly motivated by the work of Liu in \cite{L20}, where some representations
of some shifted quantum loop groups are realized via the K-theory of quasi-maps spaces.
Liu's construction uses some limit procedure similar to the limit procedure of Hernandez-Jimbo in \cite{HJ12}.
In our setting given by critical  K-theory of triple quiver varieties, 
this limit procedure admits a natural interpretation.
More precisely, we prove the following.

\begin{theorem}\label{thm:B}
\hfill
\begin{enumerate}[label=$\mathrm{(\alph*)}$,leftmargin=8mm]
\item The Kirillov-Reshetikhin modules of the quantum loop group
are realized in the critical K-theory of graded triple quiver varieties (for a convenient choice of the parameter $\gamma$).
\item The negative prefundamental modules of the shifted quantum loop group 
are realized in the critical K-theory 
of graded triple quiver varieties.
\end{enumerate}
\end{theorem}

A further motivation comes from cluster theory.
Using cluster algebras, Hernandez-Leclerc give in \cite{HL16} 
a $q$-character formula for  
prefundamental and Kirillov-Reshetikhin representations in terms of Euler characteristic of quiver Grassmanians. 
Their character formula does not give any geometric realization of the (shifted) quantum loop group action.
It is surprising that our construction yields indeed a representation of the (shifted) quantum loop group 
in some critical K-theory groups supported on the same quiver Grassmanians.
The Kirillov-Reshetikhin modules are particular cases of reachable
modules for the cluster algebra structure on the Grothendieck ring of the quantum loop group considered
in \cite{KKOP20}. 
The Euler characteristic description of the $q$-characters extends to all reachable modules.
We expect that all reachable modules admit also 
a realization in critical K-theory.

Finally let us point out a link with  K-theoretic Hall algebras of a quiver with potential. 
These algebras were introduced by Padurariu in \cite{P21}.
It was proved there that Isik's Koszul duality (=dimensional reduction) implies
that the K-theoretic Hall algebras of triple quivers with some particular potential 
coincide with the K-theoretic Hall algebras of preprojective algebras considered in \cite{VV22}.
We define an algebra homomorphism from 
K-theoretic Hall algebras to K-theoretic critical convolution algebras  using Hecke correspondences.
As a consequence, the K-theoretic critical convolution algebras may be viewed as some doubles 
of the K-theoretic Hall algebras. 
These doubles are a better setting for representation theory than the K-theoretic Hall algebras,
as the examples below suggest.
Different doubles
of the same K-theoretic Hall algebras can be realized via different 
K-theoretic critical convolution algebras.
This is especially transparent  in the symmetric case, since:
\begin{enumerate}[label=$\mathrm{(\alph*)}$,leftmargin=8mm,itemsep=1.5mm]
\item[-]
in \cite{VV22} we proved that
twisted K-theoretic Hall algebras  of preprojective Dynkin quivers are isomorphic to 
the Drinfeld halves of quantum loop groups,
\item[-]
quantum loop groups and shifted quantum loop groups map to (different) K-theoretic critical convolution algebras by Theorems \ref{thm:notshifted} and \ref{thm:shifted}.
\end{enumerate}

The paper is organized into three parts: the first part is geometric and aims to define critical convolution 
algebras. The second part consists of computations in critical K-theory to check the relations of the quantum 
loop groups and shifted quantum loop groups.
The third part is devoted to the applications to representation theory. 
More precisely, the contents are as follows:
Section 2 serves as a reminder of critical K-theory, where critical convolution algebras are introduced.
Section 3 serves as a reminder of triple quiver varieties and their potentials.
We need two versions of those: the double framed triple quiver variety
$\widetilde\frakM^\bullet(W)$ will correspond to quantum loop groups,
the simply framed one $\widehat\frakM^\bullet(W)$ to the shifted quantum loop groups.
Section 4 relates critical convolution algebras to quantum loop groups and
shifted quantum loop groups. More precisely, for any simple Lie algebra $\frakg$ we define
the deformed  potential $\tilde f^\bullet_\gamma$ on $\widetilde\frakM^\bullet(W)$, 
and a potential $\widetilde f^\bullet_2$ on $\widehat\frakM^\bullet(W)$.
Then, we prove Theorem \ref{thm:A}  in Corollary \ref{cor:notshifted2} and Corollary \ref{cor:shifted2}.
Section 5 deals with geometric constructions of simple representations 
of the quantum loop groups and the shifted quantum loop groups mentioned above.
In particular we prove  the Theorem \ref{thm:B}   in Theorem \ref{thm:HL1} and Theorem \ref{thm:PF1}.
The proof  is based on the following facts:
\begin{enumerate}[label=$\mathrm{(\alph*)}$,leftmargin=8mm,itemsep=1.5mm]
\item[-]
In the non-shifted case,
the critical locus of the potential is identified in Proposition \ref{prop:HL1} with the quiver Grassmanian 
used by Hernandez-Leclerc in \cite{HL16} to relate
the $q$-characters of Kirillov-Reshetikhin modules with cluster algebras.
\item[-]
In the shifted case,
the critical locus of the potential is also identified in Proposition \ref{prop:HL3}
with a quiver Grassmanian which is
used in \cite{HL16}.
\item[-]
The critical K-theory yields a geometric realization 
of the limit procedure in \cite{HJ12}, see Theorem \ref{thm:limH}.
\end{enumerate}

\noindent 

Section 6 deals with K-theoretic Hall algebras of a quiver with potential and their 
relations with critical convolution algebras. The existence of an algebra homomorphism 
from the K-theoretic Hall algebra to the critical convolution algebra is used in the proof of Part (b) of Theorem 1.2.
In Corollary \ref{cor:relations57} we  also give an algebra homomorphism 
from the Drinfeld half of $\U_\zeta(L\frakg)$ 
to the K-theoretic Hall algebra which generalizes \cite{VV22} for non-symmetric types.
Cohomological Hall algebras with symmetrizers already appear in \cite{YZ22} in relation with localized 
equivariant Borel-Moore homology. 
We cannot use \cite{YZ22}  because we need non-localized equivariant K-theory and we do not know wether the 
equivariant critical K-theory is torsion free or not.
Appendix A is a reminder of basic facts on representations of quantum loop groups
and shifted quantum loop groups which are used throughout the paper. 
Appendix B deals with the $Q=A_1$ case. 
In Appendix C we introduce the critical cohomological convolution algebras, which are
cohomological counterpart of the critical K-theoretic convolution algebras considered so far. 
Most of our results in critical K-theory extend to critical cohomology.
In this setting, the analog of Theorem 1.3 yields a realization of Kirillov-Reshetikhin and prefundamental modules, 
for all types, in the homology of 
quiver Grassmanians with coefficients in some sheaf of vanishing cycles.
These quiver Grassmanians are described in Propositions \ref{prop:HL1} and \ref{prop:HL3}.
They are the same as the quiver Grassmanians used in \cite[thm.~4.8, rmk.~4.19]{HL16}.
This clarifies a remark in \cite[rem.~4.11]{HL16}.
See \cite{VV23} for more details.
Appendix D is a reminder on algebraic 
and topological critical K-theory.
Most of our results in critical K-theory extend also to topological critical K-theory.

Many of our results hold in a greater generality than the one we use.
For instance, we could allow the quantum parameter to be a root of unity or the Cartan matrix
to be a symmetrizable generalized Cartan matrix in Theorems
\ref{thm:notshifted} and \ref{thm:shifted}. To simplify the exposition
We restrict to the case which is the most used in representation theory: 
quantum loop groups with generic quantum parameter and their shifted analogues.

\subsection{Notation}
All schemes are assumed to be separated schemes, locally of finite type, over the field $\bbC$.
We may allow an infinite number of connected components, but each of them
is assumed to be of finite type.
By a point of a $\bbC$-scheme we will always mean a $\bbC$-point.
Given a scheme $X$ with an action of an affine group $G$, let
$\Db\Coh_G(X)$ be the bounded derived category of the category $\Coh_G(X)$ of $G$-equivariant coherent 
sheaves $X$ and let $\Perf_G(X)$ be the full subcategory of perfect complexes.
For each $G$-invariant closed subscheme $Z$ let $\Coh_G(X)_Z$
be the category of coherent sheaves with set-theoretic support in $Z$,
and let
$\Db\Coh_G(X)_Z$ be the full triangulated subcategory of $\Db\Coh_G(X)$ 
consisting of the complexes with cohomology set-theoretically supported on $Z$.
We say that a $G$-invariant morphism $\phi:Y\to X$ of $G$-schemes is of finite $G$-flat 
dimension if the pull-back functor
$L\phi^*:\D^-\Coh_G(X)\to\D^-\Coh_G(Y)$ 
takes $\Db\Coh_G(X)$ to $\Db\Coh_G(Y)$.

Let $K_0(\calC)$ be the complexified Grothendieck group of an Abelian or triangulated category $\calC$.
Let $R_G$ be the complexified Grothendieck ring of the group $G$,
and $F_G$ be the fraction field of $R_G$.
We abbreviate
$R=R_{\bbC^\times}=\bbC[q,q^{-1}]$ and $F=F_{\bbC^\times}=\bbC(q)$.
We also set
$K_G(X)=K_0(\Perf_G(X)),$
$K^G(X)=K^G(\Db\Coh_G(X))$ and
$K^G(X)_Z=K_0(\Db\Coh_G(X)_Z).$
Note that $K^G(X)_Z=K^G(Z)$. 
If $G=\{1\}$ we abbreviate $K(X)=K^G(X)$.
This notation might be confusing, however in this paper we will not use
$K_0(\Perf(X))$ in the non equivariant case.
We write
$$\Lambda_a(\calE)=\sum_{i\geqslant 0}a^i\Lambda^i(\calE)\in K_G(X)
 ,\quad
\calE\in K_G(X)
 ,\quad
a\in R_G^{\times}.$$

Given two schemes $X_1$, $X_2$ and functions $f_a:X_a\to\bbC$ with $a=1,2$, 
we define $f_1\oplus f_2:X_1\times X_2\to\bbC$ to be the function
$f_1\oplus f_2=f_1\otimes 1+1\otimes f_2.$ If $X_1=X_2=X$, and $f_1=f_2=f$ we abbreviate
$f^{\oplus 2}=f\oplus f$ and
$f^{(2)}=f\oplus(-f).$

\section{Critical convolution algebras}

This section is a reminder on critical K-theory. We mainly refer to 
\cite{BFK13} and \cite{EP15}. We will use the equivariant critical K-theory of non-affine schemes
relatively to the action of affine groups. Note that loc.~cit. considers only the non-equivariant case.
The results there generalize easily to the equivariant case.
We refer to \cite{H17b} for more details in the equivariant case. The goal of this section  is to introduce
critical convolution algebras in \S\ref{sec:Kcritalg}. The main result is Proposition  \ref{prop:critalg1}.

\subsection{Derived factorization categories and critical K-theory}\label{sec:DKcrit}
\subsubsection{Derived factorization categories}\label{sec:Dcrit}
A $G$-equivariant LG-model is a triple $(X,\chi,f)$ where
\hfill
\begin{enumerate}[label=$\mathrm{(\alph*)}$,leftmargin=8mm]
\item  $G$ is an affine group,
\item $X$ is a quasi-projective scheme with a $G$-equivariant ample line bundle,
\item $\chi:G\to\bbC^\times$ is a character of $G$ and 
$f:X\to\bbC$ is a $\chi$-semi-invariant function on $X$,
\item the critical set of $f$ is contained into its zero locus.
\end{enumerate}

A morphism of $G$-equivariant LG-models $\phi:(X_2,\chi,f_2)\to (X_1,\chi,f_1)$ is a $G$-invariant morphism
$\phi:X_2\to X_1$ such that $f_2=\phi^*f_1$.
We say that the $G$-equivariant LG-model $(X,\chi,f)$ is smooth if $X$ is smooth.
If $\chi=1$ we say that $(X,f)$ is a $G$-invariant LG-model, and if $G=\{1\}$ that $(X,f)$ is an LG-model.

Let $\Coh_G(X,f)$ be the dg-category of all $G$-equivariant coherent factorizations of $f$ on $X$,
see \cite{BFK14}.
An object of $\Coh_G(X,f)$ is a sequence
$$\xymatrix{\calE_1\ar[r]^-{\phi_1}&\calE_0\ar[r]^-{\phi_0}&\calE_1\otimes\chi}$$
where $\calE_0,\calE_1\in\Coh_G(X)$ and $\phi_0,\phi_1$ are $G$-invariant homomorphisms such that
$$\phi_0\circ\phi_1=f\cdot\id_{\calE_1}
,\quad
(\phi_1\otimes\chi)\circ\phi_0=f\cdot\id_{\calE_0}.$$
The homotopy category $H^0(\Coh_G(X,f))$ of $\Coh_G(X,f)$ is a triangulated category. 
The category of absolutely acyclic objects is
the thick subcategory  of $H^0(\Coh_G(X,f))$
generated by the totalization of the exact triangles.
Let $Acyclic$ denote this category.
The absolute derived factorization category is the Verdier quotient
\begin{align*}
\DCoh_G(X,f)=H^0(\Coh_G(X,f))\,/\,Acyclic.
\end{align*}
We abbreviate derived factorization category
for absolute derived factorization category.

Let $Z\subset X$ be a closed $G$-invariant subset. 
An object $(\calE_1,\calE_0,\phi_1,\phi_0)$ in $\Coh_G(X,f)$ is set-theoretically supported on $Z$ if its restriction to 
$X\setminus Z$ is 0, i.e., if the coherent sheaves $\calE_0$, $\calE_1$ are set-theoretically supported on $Z$.
Let  
$$\Coh_G(X,f)_\uZ$$
be the full dg-subcategory of all factorizations set-theoretically supported on $Z$.
Let $\DCoh_G(X,f)_\uZ$ be the Verdier quotient of the homotopy category $H^0(\Coh_G(X,f)_\uZ)$ of $\Coh_G(X,f)_\uZ$ 
by the thick subcategory of acyclic objects.
We have a fully faithful embedding
$$\DCoh_G(X,f)_\uZ\subset\DCoh_G(X,f),$$
see \cite[\S3.1]{EP15}, \cite[prop.~2.25]{H17b}.
We will say that an object in $\DCoh_G(X,f)_\uZ$ is set-theoretically supported on $\uZ$.
An object of $\DCoh_G(X,f)$ is category-theoretically supported on $Z$
if its restriction in $H^0(\Coh_G(X\setminus Z,f))$ is acyclic, i.e., 
its restriction in $\DCoh_G(X\setminus Z,f)$ is zero.
Let  
$$\DCoh_G(X,f)_Z$$
be the full triangulated subcategory of all factorizations category-theoretically supported on $Z$.
We have a fully faithful embedding
$$\DCoh_G(X,f)_\uZ\subset\DCoh_G(X,f)_Z.$$
Forgetting the support yields fully faithful triangulated functors
$$\DCoh_G(X,f)_\uZ\subset\DCoh_G(X,f)_Z\subset\DCoh_G(X,f).$$ 
The category $\DCoh_G(X,f)_Z$ is the thick envelope of $\DCoh_G(X,f)_\uZ$ in
$\DCoh_G(X,f)$. So $\DCoh_G(X,f)_\uZ$ is a dense full subcategory of
$\DCoh_G(X,f)_Z$ in the terminology of Thomason in \cite{T97}.
The restriction $$\DCoh_G(X,f)\to\DCoh_G(X\setminus Z,f)$$ 
is the Verdier localization by the triangulated subcategory $\DCoh_G(X,f)_Z$.
Hence, we have an exact 
sequence of triangulated categories
\begin{align}\label{loc1}
\DCoh_G(X,f)_Z\to \DCoh_G(X,f)\to \DCoh_G(X\setminus Z,f),
\end{align}
i.e., the thick category $\DCoh_G(X,f)_Z$ is the kernel of the restriction functor,
see \cite[prop.~2.26]{H17b}.
We deduce that for each $G$-invariant closed subset $F\subset X$ we also have
an exact sequence of triangulated categories
\begin{align}\label{loc2}
\DCoh_G(X,f)_{Z\cap F}\to \DCoh_G(X,f)_F\to \DCoh_G(X\setminus Z,f)_{(X\setminus Z)\cap F},
\end{align}
See \cite[\S 3.1]{EP15} for more details on supports.

Let $\phi:(X_2,\chi,f_2)\to (X_1,\chi,f_1)$ be a morphism of $G$-equivariant LG-models.
Let $Z_1,$ $Z_2$ be closed $G$-invariant subsets of $X_1,$ $X_2$.
We need the following functorial properties.

Assume that $\phi^{-1}(Z_1)\subset Z_2$.
Then, we have a pull-back dg-functor 
$$\phi^*:\Coh_G(X_1,f_1)_{\uZ_1}\to\Coh_G(X_2,f_2)_{\uZ_2}.$$
Assume further that $\phi$ has finite $G$-flat dimension.
Then, we have a triangulated functor
\begin{align*}
L\phi^*:\DCoh_G(X_1,f_1)_{\uZ_1}\to\DCoh_G(X_2,f_2)_{\uZ_2}.
\end{align*}

Assume that $\phi(Z_2)\subset Z_1$ and that the restriction $\phi|_{Z_2}$ of the map $\phi$
to the subset $Z_2$ is proper.
Then, we have a pushforward dg-functor 
$$\phi_*:\Coh_G(X_2,f_2)_{\uZ_2}\to\Coh_G(X_1,f_1)_{\uZ_1}$$
and a triangulated functor
\begin{align*}
R\phi_*:\DCoh_G(X_2,f_2)_{\uZ_2}\to\DCoh_G(X_1,f_1)_{\uZ_1}.
\end{align*}

The external tensor product yields a dg-functor
$$\boxtimes:\Coh_G(X_1,f_1)\otimes\Coh_G(X_2,f_2)\to\Coh_G(X_1\times X_2,f_1\boxplus f_2)$$
and a triangulated functor
$$\boxtimes:\DCoh_G(X,f_1)\otimes\DCoh_G(X,f_2)\to\DCoh_G(X_1\times X_2,f_1\oplus f_2).$$
Assume that $X_1=X_2=X$.
There is a dg-functor
$$\otimes:\Coh_G(X,f_1)\otimes\Coh_G(X,f_2)\to\Coh_G(X,f_1+f_2)$$
Assume further that $X$ is smooth.
Then we have a triangulated functor
$$\otimes^L:\DCoh_G(X,f_1)_{\uZ_1}\otimes\DCoh_G(X,f_2)_{\uZ_2}\to\DCoh_G(X,f_1+f_2)_{\underline{Z_1\cap Z_2}}.$$
See \cite[\S 3.5-6]{EP15} and \cite[\S 2.3-4]{H17b} for more details on functoriality and tensor products.
The functors $L\phi^*$, $R\phi_*$ and $\otimes$ are also compatible with categorical supports.
See \cite[lem.~6.4]{AK20} for details.

\begin{remark}\label{rem:base change 1} 
\hfill
\begin{enumerate}[label=$\mathrm{(\alph*)}$,leftmargin=8mm]
\item 
By \cite[lem.~ 5.5, prop.~5.6, 5.7]{AK20} and \cite[prop.~4.32, lem.~4.34]{H17a},
the derived pullback commutes with tensor products, and
the derived pushforward and pulback satisfy the projection formula
and the flat base change property.
\item
Since $\DCoh_G(X,f)_\uZ$ is a dense full subcategory of $\DCoh_G(X,f)_Z$,
a theorem of Thomason \cite{T97} yields an inclusion of Grothendieck groups
$$K_0(\DCoh_G(X,f)_\uZ)\subset K_0(\DCoh_G(X,f)_Z).$$
This inclusion may not be an isomorphism. 
We expect that, in the particular cases considered in this paper, both support conditions will coincide. However, we cannot prove this 
in general, primarily because we used set-theoretic support in the localization theorem in Proposition \ref{prop:TTK} below.
\end{enumerate}
\end{remark}

\subsubsection{The functor $\Upsilon$}\label{sec:Upsilon}
From now on we will always assume that $\chi=1$, as this suffices for our purpose.
Let $(X,f)$ be a smooth $G$-invariant LG-model.
Let $Y$ be the zero locus of $f$,
$i$ be the closed embedding $Y\subset X$,
and $Z\subset Y$ be a closed $G$-invariant subset.
The derived category of bounded complexes over an Abelian category coincides with its absolute
derived category. Hence, we have a triangulated functor
\begin{align*}\Upsilon:\Db\Coh_G(Y)_Z\to \DCoh_G(X,f)_\uZ\end{align*}
which takes a complex $(\calE^\bullet,d)$ to the factorization
\begin{align*}
\xymatrix{
{\displaystyle\bigoplus_{m\in\bbZ}i_*\calE^{2m-1}}\ar[r]^-{d}&
{\displaystyle\bigoplus_{m\in\bbZ}i_*\calE^{2m}}\ar[r]^-{d}&
{\displaystyle\bigoplus_{m\in\bbZ}i_*\calE^{2m-1}}}
\end{align*}
Given a closed $G$-invariant subset $Z\subset Y$, let 
$\Perf_G(Y)_Z$ be the full subcategory of $\Perf_G(Y)$ consisting of the perfect complexes
with cohomology sheaves set-theoretically supported in $Z$.
The equivariant triangulated category of singularities of $Y$ supported on $Z$ is the Verdier quotient 
\begin{align}\label{sing1}\DCoh_G^\sg(Y)_\uZ=\Db\Coh_G(Y)_Z\,/\,\Perf_G(Y)_Z.\end{align}
We abbreviate $\DCoh_G^\sg(Y)=\DCoh_G^\sg(Y)_\uY.$
By \cite[lem.~3.1]{EP15}, the forgetful functor 
$$\DCoh_G^\sg(Y)_\uZ\to\DCoh_G^\sg(Y)$$
is fully faithful,
and, by \cite[prop.~3.1]{EP15}, the functor $\Upsilon$ factorizes through an equivalence 
\begin{align}\label{sg}\Db\Coh^\sg_G(Y)_\uZ\simeq \DCoh_G(X,f)_\uZ.
\end{align}
Hence, the functor $\Upsilon$ can be identified with the composed functor
\begin{align}\label{upsilon2}
\Upsilon:\Db\Coh_G(Y)_Z\to\DCoh^\sg_G(Y)_\uZ\to\DCoh_G(X,f)_\uZ
\end{align}
where the first arrow is the localization functor in \eqref{sing1} and the second one is \eqref{sg}.

\begin{lemma}\label{lem:Upsilon}
Let $\phi:(X_2,f_2)\to (X_1,f_1)$ be a morphism of smooth $G$-invariant LG-models.
Let $Y_1=f_1^{-1}(0)$, $Y_2=f_2^{-1}(0)$ and
$Z_1\subset Y_1$, $Z_2\subset Y_2$ be closed $G$-invariant subsets.
Assume that $\phi$ is of finite $G$-flat dimension. Assume that the functions $f_1$, $f_2$ are locally
non-zero everywhere on $X_1$, $X_2$.
If $\phi^{-1}(Z_1)\subset Z_2$, then there is an isomorphism of functors
$$L\phi^*\circ\Upsilon=\Upsilon\circ L\phi^*.$$
If $\phi(Z_2)\subset Z_1$ and the map $\phi$ is proper, then there is an isomorphism of functors
$$R\phi_*\circ\Upsilon=\Upsilon\circ R\phi_*.$$
\end{lemma}

\begin{proof}
Since $\phi$ has finite $G$-flat dimension, so does the homomorphism $\phi|_{Y_2}:Y_2\to Y_1$.
Hence, the functors $L\phi^*$ and $R\phi_*$ on 
$\Db\Coh_G(Y_1)_{Z_1}$ and $\Db\Coh_G(Y_2)_{Z_2}$ are well-defined.
The functor $\Upsilon$ above is considered in \cite[\S 2.7]{EP15}, 
where it is proved to factorize to an equivalence of triangulated 
categories from the category of singularities of $Y$ to $\DCoh(X,f)$.
The isomorphism of functors $R\phi_*\circ\Upsilon=\Upsilon\circ R\phi_*$
is proved in \cite[\S 3.6]{EP15}. The isomorphism $L\phi^*\circ\Upsilon=\Upsilon\circ L\phi^*$
is obvious. Note that  $G=\{1\}$ in loc.~cit.,
but the case of a general group $G$ is proved in the same way.
The compatibility with supports is also obvious, see, e.g., \cite[\S 3.1]{EP15}.
\end{proof}

\begin{remark}
Lemma \ref{lem:Upsilon} identifies
the pull-back and pushforward functors of the equivariant triangulated category of singularities 
and the derived factorization category with set-theoretical supports,
under the equivalence \eqref{sg}.
\end{remark}

\subsubsection{Critical K-theory}\label{sec:critK}
Fix a $G$-invariant LG-model $(X,f)$.
Let $Y\subset X$ be the zero locus of $f$, 
$i$ be the closed embedding $Y\to X$, and
$Z\subset Y$ a closed $G$-invariant subset.
Set
\begin{align}\label{Kcrit}
K_G(X,f)_Z=K_0(\DCoh_G(X,f)_Z)
,\quad
K_G(X,f)_\uZ=K_0(\DCoh_G(X,f)_\uZ).
\end{align}
Compare \eqref{Ktop2}.
The density of the category $\DCoh_G(X,f)_\uZ$ in
$\DCoh_G(X,f)_Z$ and \cite[cor.~2.3]{T97} imply that
\begin{align}\label{density}
K_G(X,f)_\uZ\subset K_G(X,f)_Z.
\end{align}
Assume that $(X,f)$ is smooth.
The triangulated functors $\otimes^L$, $L\phi^*$ and $R\phi_*$ above yield maps
in critical K-theory. In particular, we have maps
\begin{align}\label{otimes}
\otimes:K_G(X)\otimes K_G(X,f)_\uZ\to K_G(X,f)_{\uZ}
,\quad
\otimes:K_G(X)\otimes K_G(X,f)_Z\to K_G(X,f)_{Z}
\end{align}
and maps
\begin{align}\label{Upsilon2}
\Upsilon:K^G(Y)\to K_G(X,f)
,\quad
\Upsilon:K^G(Z)\to K_G(X,f)_\uZ.
\end{align}

\begin{lemma}\label{lem:surjectivity} 
Assume that the function $f$ is locally
non-zero everywhere on $X$.
Then the maps 
$\Upsilon:K^G(Y)\to K_G(X,f)$
and
$\Upsilon:K^G(Z)\to K_G(X,f)_\uZ$
are surjective.
\end{lemma}

\begin{proof}
We must check that the first functor in \eqref{upsilon2}
yields a surjective morphism of Grothendieck groups.
By \eqref{sing1} this surjectivity follows from \cite[prop.~VIII.3.1]{SGA5}. 
\end{proof}

Echanging $Z$ and $F$ in  \eqref{loc2} yields the following localization exact sequence
\begin{align}\label{les}
K_G(X,f)_{Z\cap F}\to K_G(X,f)_Z\to K_G(U,f)_{Z\cap U}\to 0
\end{align}
Here $F$ is any $G$-invariant closed subset of $X$ and $U=X\setminus F$.
In particular, we have the excision property
\begin{align}\label{excision}
Z\subset U\Rightarrow K_G(X,f)_Z=K_G(U,f)_Z.
\end{align}
Further \cite[prop.~2.7]{O11} yields $\DCoh_G(U,f)=0$ if $F=\crit(f)$.
Using \eqref{loc2}, we deduce that 
\begin{align}\label{crit}
K_G(X,f)_Z=K_G(X,f)_{Z\cap\crit(f)}.
\end{align}

We will  need the deformed dimensional reduction from \cite[thm.~1.2]{H17b} which generalizes
the dimensional reduction in K-theory in \cite{I12}. 
More precisely, let $(X,f)$ be a smooth $G$-invariant LG-model.
Let $\pi:E\to X$ be a $G$-equivariant vector bundle on $X$ with a $G$-invariant section $s$ of the dual bundle $E^\vee$.
Let $Z\subset X$ be the zero scheme of $s$.
Assume that the restriction $f|_Z:Z\to\bbC$ is flat and that the section $s$ is regular, i.e., the codimension of $Z$ in $X$ is
equal to the rank of the vector bundle $E$.
Let $g:E\to\bbC$ be the function given by the pairing with the section $s$. Then, there is an isomorphism
\begin{align}\label{DDR}
K_G(Z, f|_Z)\simeq K_G(E,\pi^*f+g).
\end{align}

\subsection{Critical convolution algebras}\label{sec:Kcritalg}

Let $(X_1,f_1)$, $(X_2,f_2)$ and $(X_3,f_3)$ be smooth $G$-invariant LG-models.
Set $X_{123}=X_1\times X_2\times X_3$ and $X_{ab}=X_a\times X_b$.
Let $\pi_{ab}:X_{123}\to X_{ab}$ be the projection along the factor not named.
Set 
$$f_{ab}=f_a\boxplus(-f_b)
,\quad
Y_a=f_a^{-1}(0)
,\quad
Y_{ab}=f_{ab}^{-1}(0).$$
The group $G$ acts diagonally on $X_{123}$.
Let $Z_{ab}$ be a $G$-invariant closed subset of $Y_{ab}$.
Assume that the restriction of
$\pi_{13}$ to $\pi_{12}^{-1}(Z_{12})\cap\pi_{23}^{-1}(Z_{23})$
is proper and maps into  $Z_{13}$.
Assume also that the function $f_a$ is locally
non-zero everywhere on $X_a$.
There is a convolution functor 
\begin{align}\label{conv1}
\star:\Db\Coh_G(X_{12})_{Z_{12}}\otimes\Db\Coh_G(X_{23})_{Z_{23}}\to
\Db\Coh_G(X_{13})_{Z_{13}}
\end{align}
such that 
$$\calE\star\calF=R(\pi_{13})_*(L(\pi_{12})^*(\calE)\otimes^LL(\pi_{23})^*(\calF)).$$
Since
$(\pi_{12}\times\pi_{23})^*(f_{12}\oplus f_{23})=(\pi_{13})^*f_{13}$,
we define in a similar way a convolution functor of derived factorization categories 
\begin{align}\label{conv2}
\star:\DCoh_G(X_{12},f_{12})_{Z_{12}}\otimes\DCoh_G(X_{23},f_{23})_{Z_{23}}\to
\DCoh_G(X_{13},f_{13})_{Z_{13}}
\end{align}
such that 
$\calE\star\calF=R(\pi_{13})_*(L(\pi_{12})^*(\calE)\otimes^LL(\pi_{23})^*(\calF))$.
Using set-theoretical support we define in the same way a 
convolution functor
\begin{align}\label{conv3}
\star:\DCoh_G(X_{12},f_{12})_{\uZ_{12}}\otimes\DCoh_G(X_{23},f_{23})_{\uZ_{23}}\to
\DCoh_G(X_{13},f_{13})_{\uZ_{13}}
\end{align}

Now, we consider the following particular case:
the pair $(X,f)$ is a smooth $G$-invariant LG-model,
the map $\pi:X\to X_0$ is a proper $G$-map to an affine $G$-scheme,
the function $f$ is $f=f_0\circ\pi$ where $f_0$ is a $\chi$-semi-invariant function on $X_0$,
and $Y$, $Y_0$ are the zero loci of $f$ and $f_0$. We also choose a $G$-fixed point $x_0$ in $Y_0$.
We define 
$$Z=X\times_{X_0} X
 ,\quad
L=X\times_{X_0}\{x_0\}.
$$
Set $X_a=X$, $Z_{ab}=Z$, $f_a=f$ and $f_{ab}=f^{(2)}$ for each $a$, $b$.
Note that  $Z\subset Y_{ab}$.

First, recall the usual convolution algebra in K-theory, following \cite{CG}.
The convolution functor \eqref{conv1} yields a monoidal structure on the triangulated category
$\Db\Coh_G(X^2)_Z$, and a $\Db\Coh_G(X^2)_Z$-module structure
on the categories $\Db\Coh_G(X)_L$, $\Db\Coh_G(X)_Y$, $\Db\Coh_G(X)$.
This yields an associative $R_G$-algebra structure on
$$K^G(X^2)_Z=K^G(Z)$$ and $K^G(Z)$-representations in $K^G(L)$, $K^G(Y)$ and $K^G(X)$.

Now, we consider convolution algebras in critical K-theory.

\begin{proposition}\label{prop:critalg1}
\hfill
\begin{enumerate}[label=$\mathrm{(\alph*)}$,leftmargin=8mm]
\item 
$\DCoh_G(X^2,f^{(2)})_Z$
is a monoidal category.
\item
$\DCoh_G(X,f)_L$ and $\DCoh_G(X,f)$  are modules over  
$\DCoh_G(X^2,f^{(2)})_Z$.
\item
$K_G(X^2,f^{(2)})_Z$ is an $R_G$-algebra which acts on
$K_G(X,f)_L$ and $K_G(X,f)$.

\item
$\Upsilon:K^G(Z)\to K_G(X^2,f^{(2)})_Z$ is an algebra homomorphism.

\item
$\Upsilon:K^G(Y)\to K_G(X,f)$ and $\Upsilon:K^G(L)\to K_G(X,f)_L$ are module homomorphism.

\item
Parts $\operatorname{(a)}$-$\operatorname{(e)}$ holds also with categorical supports replaced by set-theoretical ones.
 
 \item
 $\Upsilon(K^G(Z))=K_G(X^2,f^{(2)})_\uZ$,
 $\Upsilon(K^G(Y))=K_G(X,f)$ and
 $\Upsilon(K^G(L))=K_G(X,f)_\uL$.
 
 \item
$K_G(X^2,f^{(2)})_\uZ$ is a subalgebra of $K_G(X^2,f^{(2)})_Z$ which acts on 
$K_G(X,f)_\uL$.
\end{enumerate}
\end{proposition}

\begin{proof}
For (a) we must define an associativity constraint and a unit 
satisfying the pentagon and the unit axioms. 
The associativity constraint follows from the flat base change and the projection formula as in
\cite[prop.~5.13]{BDFIK16}.
The unit is the factorization $\Upsilon\Delta_*\calO_X$.
For (b) we choose $X_1=X_2=X$, $X_3=\{x_0\}$, $f_1=f_2=f$,
$f_3=0$, $Z_{12}=Z$ and $Z_{23}=Z_{13}=L\times\{x_0\}$ or $Z_{23}=Z_{13}=X\times\{x_0\}$
and we apply \eqref{conv2}.
Part (c) follows from (a) and (b) by taking the Grothendieck groups.
Part (f) is proved using \eqref{conv3} instead of \eqref{conv2}.
This implies that the convolution product
preserves the subcategories $\DCoh_G(X^2,f^{(2)})_\uZ$ and $\DCoh_G(X,f)_\uL$. 
Part (g) follows from Lemma \ref{lem:surjectivity}.
Part (h) follows from \eqref{density}.
Now, we concentrate on (d) and (e).

The convolution functors 
\begin{align}\label{star0}
\begin{split}
\star:\Db\Coh_G(X_{12},f_{12})_{Z_{12}}\otimes\Db\Coh_G(X_{23},f_{23})_{Z_{23}}\to\Db\Coh_G(X_{13},f_{13})_{Z_{13}}
\end{split}
\end{align}
and 
\begin{align}\label{star1}
\begin{split}
\star:\Db\Coh_G(X_{12})_{Z_{12}}\times\Db\Coh_G(X_{23})_{Z_{23}}\to\Db\Coh_G(X_{13})_{Z_{13}}
\end{split}
\end{align}
are both given by
\begin{align}\label{star}
\calE\star\calF=R(\pi_{13})_*L(\pi_{12}\times\pi_{23})^*(\calE\boxtimes\calF).
\end{align}
We must compare the functors \eqref{star0} and \eqref{star1}.

To do this, it is convenient to use the formalism of derived schemes.
A derived scheme is a pair $X=(|X|,\calO_X)$ where $|X|$ is a topological space and $\calO_X$
is a sheaf on $|X|$ with values in the $\infty$-category of simplicial commutative rings such that
the ringed space $(|X|,\pi_0\calO_X)$ 
is a scheme and the sheaf $\pi_n\calO_X$ is a quasi-coherent 
$\pi_0\calO_X$-module over this scheme for each $n>0$.
Here, all derived schemes will be defined over $\bbC$, hence derived schemes can be
modeled locally by dg-algebras rather than simplicial ones.
Let $M$ be a smooth quasi-affine $G$-scheme
and $\sigma$ a $G$-invariant section of a $G$-equivariant vector bundle $E$ over $M$.
The derived zero locus 
is the derived $G$-scheme $$X=R(E\to M,\sigma)$$ given by the derived fiber product
$M\times_E^RM$ relative to the maps $\sigma:M\to E$ and $0:M\to E$.
The derived scheme $X$ is quasi-smooth, i.e., it is finitely presented and its cotangent 
complex is of cohomological amplitude $[-1,0]$.
For any derived $G$-scheme $X$, let $\Db\Coh_G(X)$ be the derived category
of modules over $\calO_X$ with bounded coherent cohomology.

Now, we consider the derived scheme 
$RY_{ab}=R(X_{ab}\times\bbC\to X_{ab}\,,\,f_{ab}).$
We have the following obvious embeddings of derived schemes
$$\xymatrix{Y_{ab}\ar[r]^-j&RY_{ab}\ar[r]^-i& X_{ab}},$$ 
which yield the following commutative diagram 
\begin{align*}
\xymatrix{
X_{12}\times X_{23}&&\ar[ll]_-{\pi_{12}\times \pi_{23}} X_{123}\ar[r]^-{\pi_{13}}&X_{13}\\
RY_{12}\times RY_{23}\ar[u]^-i&&\ar[ll]_-{\pi_{12}\times \pi_{23}} RY_{123}\ar[r]^-{\pi_{13}}\ar[u]_-i&RY_{13}\ar[u]_-i
}
\end{align*}
The left square is Cartesian.
The upper left horizontal map has finite $G$-flat dimension because $X_1$, $X_2$, $X_3$ are smooth. 
The lower one as well because it is quasi-smooth, see, e.g., \cite[lem.~1.15]{K22}.
Thus, we also have a convolution functor
\begin{align}\label{star2}
\star:\Db\Coh_G(RY_{12})_{Z_{12}}\times\Db\Coh_G(RY_{23})_{Z_{23}}\to\Db\Coh_G(RY_{13})_{Z_{13}}
\end{align}
given by the formula \eqref{star}.
The left square is Cartesian. 
The base change morphism
$$L(\pi_{12}\times\pi_{23})^*\circ Ri_*\to Ri_*\circ L(\pi_{12}\times\pi_{23})^*$$
is invertible by \cite[cor.~3.4.2.2]{L18}. 
Hence the direct image 
\begin{align*}
Ri_*:\Db\Coh_G(RY_{ab})_{Z_{ab}}\to\Db\Coh_G(X_{ab})_{Z_{ab}}
\end{align*}
intertwines the convolution functors \eqref{star2} and \eqref{star1}.
The homomorphism $j$ is a quasi-isomorphism
because the function $f_{ab}$ is locally non-zero everywhere on $X_{ab}$.
Hence, the pushforward and pull-back functors $Rj_*$ and $Lj^*$
are mutually inverse equivalences of categories 
$$\Db\Coh_G(Y_{ab})_{Z_{ab}}=\Db\Coh_G(RY_{ab})_{Z_{ab}}$$
Hence \eqref{star2} yields a convolution functor
\begin{align}\label{star3}
\star:\Db\Coh_G(Y_{12})_{Z_{12}}\times\Db\Coh_G(Y_{23})_{Z_{23}}\to\Db\Coh_G(Y_{13})_{Z_{13}}.
\end{align}
such that the direct image 
\begin{align*}
Rj_*:\Db\Coh_G(Y_{ab})_{Z_{ab}}\to\Db\Coh_G(RY_{ab})_{Z_{ab}}
\end{align*}
intertwines the convolution functors \eqref{star3} and \eqref{star2}.
We deduce that the composed functor
$Ri_*\circ Rj_*$ intertwines the convolution functors \eqref{star3} and \eqref{star1}.
Note that, since the K-theory satisfies the equivariant d\'evissage, the functor 
\begin{align}\label{functor1}
Ri_*\circ Rj_*:\Db\Coh_G(Y_{ab})_{Z_{ab}}\to\Db\Coh_G(X_{ab})_{Z_{ab}}
\end{align}
yields an isomorphism of Grothendieck groups.
Both Grothendieck groups are canonically identified with $K^G(Z_{ab})$, 
so that \eqref{functor1} induces the identity map of $K^G(Z_{ab})$.

Now, we consider the functor
\begin{align}\label{map2}
\Upsilon&:\Db\Coh_G(Y_{ab})_{Z_{ab}}\to\Db\Coh_G(X_{ab},f_{ab})_{Z_{ab}}.
\end{align}
By Lemma \ref{lem:Upsilon}, it intertwines the functors \eqref{star3} and \eqref{star0}.
It gives a map
\begin{align*}
\Upsilon:K^G(Z_{ab})\to K^G(X_{ab},f_{ab})_{Z_{ab}}.
\end{align*}
which intertwines the convolution products on both sides.
Composing \eqref{functor1} and \eqref{map2} proves Part (d) of the proposition.
Part (e) is proved similarly.

\end{proof}

\begin{proposition}\label{prop:TTK}
Let $(X,f)$ and $(V,g)$ be smooth $G$-equivariant LG-models with closed $G$-invariant subsets
$Z\subset f^{-1}(0)$ and $W\subset g^{-1}(0)$.
Assume that $\rho:V\to X$ is a $G$-equivariant vector bundle 
and that  $g=f\circ\rho$ and $\rho^{-1}(Z)=W$.
Let $s:X\to V$ be a $G$-invariant regular section, and
$i:M=\{s=0\}\to X$ be the embedding of the zero subscheme of $s$.
Let  $N=Z\cap M$ and $h=f\circ i$. Assume that the function $h$ is locally non-zero.
\hfill
\begin{enumerate}[label=$\mathrm{(\alph*)}$,leftmargin=8mm]
\item
The map $Li^*\circ Ri_*:K_G(M,h)_{\uN}\to K_G(M,h)_{\uN}$ 
is the tensor product by the class $i^*\Lambda_{-1}(V^\vee)$ in $K_G(M)$.

\item
The map $Ri_*\circ Li^*:K_G(X,f)_\uZ\to K_G(X,f)_\uZ$
is the tensor product by the class $\Lambda_{-1}(V^\vee)$ in $K_G(X)$.

\item 
The map
$L\rho^*:K_G(X,f)_\uZ\to K_G(V,g)_\uW$
is an isomorphism with inverse the pull-back by the zero section.

\item
Assume that $G=T$ is a torus, that $i:X^T\to X$ is the inclusion of the fixed points locus,
and that the function $h=f\circ i$ is locally non-zero.
The maps $Ri_*$ and $Li^*$
yield isomorphisms of $F_T$-vector spaces
$$K_T(X^T,h)_{\uZ^T}\otimes_{R_T}F_T=K_T(X,f)_\uZ\otimes_{R_T}F_T.$$
The composed map $Li^*\circ Ri_*$ is the tensor product with the class 
$\Lambda_{-1}(T^*_{X^T}X)$.
\end{enumerate}
\end{proposition}

\begin{proof}
Parts (a), (b) follow from the Koszul resolution as in K-theory, see, e.g., \cite[\S 5.4]{CG}.
More precisely, 
by Lemmas \ref{lem:Upsilon}, \ref{lem:surjectivity} we have the commutative diagram
with surjective horizontal maps
\begin{align*}
\xymatrix{
K^G(N)\ar[r]^-\Upsilon\ar[d]_-{Ri_*}&K_G(M,h)_{\uN}\ar[d]^-{Ri_*}\\
K^G(Z)\ar[r]^-\Upsilon\ar[d]_-{Li^*}&K_G(X,f)_\uZ\ar[d]^-{Li^*}\\
K^G(N)\ar[r]^-\Upsilon&K_G(M,h)_{\uN}}
\end{align*}
where the left map $Li^*$ is the restriction from $X$ to $M$ with supports in $Z$ and $N$.
Hence the formula (a) follows from the corresponding formula in K-theory and the surjectivity of the map
$\Upsilon$. The proof of (b) is similar.

Part (c) is a version of the Thom isomorphism.
The map $L\rho^*$ is well-defined because $\rho$ is flat.
Let $\sigma:X\to V$ be the zero section. 
The map $L\sigma^*$ is well defined because $\sigma$ is of finite $G$-flat dimension.
The composed map $L\sigma^*\circ L\rho^*$ is an isomorphism, hence $L\rho^*$ is injective.
By Lemmas \ref{lem:Upsilon}, \ref{lem:surjectivity} we have the commutative diagram
with surjective horizontal maps
\begin{align*}
\xymatrix{
K^G(Z)\ar[r]^-\Upsilon\ar[d]_-{L\rho^*}&K_G(X,f)_\uZ\ar[d]^-{L\rho^*}\\
K^G(W)\ar[r]^-\Upsilon&K_G(V,g)_\uW}
\end{align*}
Thus the surjectivity of the right map $L\rho^*$ follows from the Thom isomorphism, 
applied to the left map $L\rho^*$ in the diagram.

Finally, let us prove Part (d).
By Lemmas \ref{lem:Upsilon}, \ref{lem:surjectivity} we have the commutative diagram
with surjective horizontal maps
\begin{align*}
\xymatrix{
K^T(Z^T)\otimes_{R_T}F_T\ar[r]^-\Upsilon\ar[d]_-{Ri_*}&
K_T(X^T,h)_{\uZ^T}\otimes_{R_T}F_T\ar[d]^-{Ri_*}\\
K^T(Z)\otimes_{R_T}F_T\ar[r]^-\Upsilon\ar[d]_-{Li^*}&K_T(X,f)_\uZ\otimes_{R_T}F_T\ar[d]^-{Li^*}\\
K^T(Z^T)\otimes_{R_T}F_T\ar[r]^-\Upsilon&K_T(X^T,h)_{\uZ^T}\otimes_{R_T}F_T}
\end{align*}
where the left map $Li^*$ is the restriction from $X$ to $X^T$ with supports in $Z$ and $Z^T$.
By (a), the map $Li^*\circ Ri_*$ is the tensor product with the class 
$\Lambda_{-1}(T^*_{X^T}X)$. Hence, it is invertible and the map
$$Ri_*:K_T(X^T,h)_{\uZ^T}\otimes_{R_T}F_T\to K_T(X,f)_\uZ\otimes_{R_T}F_T$$
is injective.
It is also surjective, because the upper square commutes, $\Upsilon$ is surjective, and
$Ri_*$ is surjective onto $K^T(Z)\otimes_{R_T}F_T$.

\end{proof}

\medskip

\section{Quiver varieties and quantum loop groups}\label{sec:quiver varieties}

\subsection{Quiver varieties}\label{sec:quiver basic}
\subsubsection{Quiver representations}\label{sec:quiver basic1}
Let $Q$ be a finite quiver with sets of vertices and of arrows $Q_0$ and $Q_1$.
Let $s,t:Q_1\to Q_0$ be the source and target.
Let $\alpha^*$ be the arrow opposite to the arrow $\alpha\in Q_1$.
Fix a grading $\deg:Q_1\to\bbZ$.
We use the auxiliary sets 
$$
Q_1^*=\{\alpha^*\,;\,\alpha\in Q_1\}
,\quad
Q'_0=\{i'\,;\,i\in Q_0\}
,\quad
Q'_1=\{a_i:i\to i'\,;\,i\in Q_0\}
,\quad
\Omega=\{\varepsilon_i:i\to i\,;\,i\in Q_0\}.$$
From $Q$ we construct new quivers as follows:
\begin{itemize}[leftmargin=3mm]
\item[-]
$\overline Q$ is the double quiver:
$\overline Q_0=Q_0$, 
$\overline Q_1=Q_1\cup Q_1^*$,
\item[-] 
$\widetilde Q$ is the triple quiver:
$\widetilde Q_0=Q_0$, $\widetilde Q_1=\overline Q_1\cup\Omega$,
\item[-] 
$Q_f$ is the framed quiver:
$Q_{f,0}=Q_0\sqcup Q'_0$,
$Q_{f,1}=Q_1\sqcup Q'_1$,
\item[-]
$\overline Q_f=\overline{(Q_f)}$ is the double quiver with double framing,
\item[-] 
$\widetilde Q_f$ is the triple quiver with double framing:
$\widetilde Q_{f,0}=Q_{f,0}$,
$\widetilde Q_{f,1}=\overline{(Q_f)}_1\cup \Omega$,
\item[-] 
$\widehat Q_f=(\widetilde Q)_f$ is the triple quiver with simple framing,

\item[-]
$Q^\bullet$ is the graded quiver:
$Q^\bullet_0=Q_0\times\bbZ$,
$Q^\bullet_1=Q_1\times\bbZ$ with $s(\alpha,k)=(s(\alpha),k)$ and $t(\alpha,k)=(t(\alpha),\deg(\alpha)+k).$
\end{itemize}
We abbreviate $I=Q_0$, $I^\bullet=Q_0^\bullet$ and
$\widetilde Q_f^\bullet=(\widetilde Q_f)^\bullet$,
$\overline Q_f^\bullet=(\overline Q_f)^\bullet.$
Let $\bfC$ and $\bfC^\bullet$ be the categories of finite dimensional $I$-graded and 
$I^\bullet$-graded vector spaces.
For any $V$ in $\bfC$ or $\bfC^\bullet$ we write
$V=\bigoplus_{i\in I}V_i$ or
$V=\bigoplus_{(i,k)\in I^\bullet}V_{i,k}$ respectively.
Given $V\in\bfC^\bullet$ let $V$ denote also the underlying object in $\bfC$.
Let $\delta_i\in\bbN I$ and $\delta_{i,k}\in\bbN I^\bullet$ be the Dirac functions at $i$ and $(i,k)$.
The dimension vectors are $v=\sum_{i\in I}v_i\delta_i$ and
$v=\sum_{(i,k)\in I^\bullet}v_{i,k}\delta_{i,k}$ respectively.
Let 
\begin{align}\label{SS}
S_i\in\bfC
,\quad
S_{i,k}\in\bfC^\bullet
\end{align}
 be representations of dimensions
$\delta_i$ and $\delta_{i,k}$.
Given $V,W\in\bfC$ the representation varieties of $Q$ and $Q_f$ are
$$\X_Q(V)=\prod_{x\in Q_1}\Hom(V_{s(x)},V_{t(x)})
 ,\quad
\X_{Q_f}(V,W)=\prod_{x\in Q_1}\Hom(V_{s(x)},V_{t(x)})\times
\prod_{i\in Q_0}\Hom(V_i,W_i).$$
A representation of $\widetilde Q_f$  is a tuple 
$x=(x_\alpha,x_a,x_{a^*},x_\varepsilon)$ with
$\alpha\in \overline Q_1$, $a\in Q'_1$ and $\varepsilon\in\Omega$.
We abbreviate $h=x_h$ for each arrow $h$ and we write
$x=(\alpha\,,\,a\,,\,a^*\,,\,\varepsilon)$.
A representation $x$ of the quiver $Q$ in a $I$-graded vector space $V$ is said to be
nilpotent if and only if and only if there is a complete flag 
$0=V_0\subset V_1\subset\dots\subset V$ such that  $x(V_l)\subset V_{l-1}$ for each $l$.
We abbreviate 
$$\overline\X=\X_{\overline Q_f}
 ,\quad
\widetilde\X=\X_{\widetilde Q_f}
 ,\quad
\widehat\X=\X_{\widehat Q_f}
 ,\quad
\overline\X^\bullet=\X_{\overline Q_f^\bullet}
 ,\quad
\widetilde\X^\bullet=\X_{\widetilde Q_f^\bullet},$$
and
$$\widetilde\X(V)=\widetilde\X(V,0)=\X_{\widetilde Q}(V)
,\quad
\widehat\X(V)=\widehat\X(V,0)=\X_{\widehat Q}(V)
,\quad
\text{etc}.$$
Let $\widetilde\X^\nil\subset\widetilde\X$ and
$\widehat\X^\nil\subset\widehat\X$
be the subsets of nilpotent representations.

\medskip

\subsubsection{Cartan matrix}\label{sec:or}
Let $\bfc=(c_{ij})_{i,j\in I}$ be a Cartan matrix.
Let $\O\subset I\times I$ be an orientation of $\bfc$. We have
$$(i,j)\in\O\ \text{or}\ (j,i)\in\O\iff c_{ij}<0
,\quad
(i,j)\in\O\Rightarrow (j,i)\not\in\O.$$
Set $\sgn_{ij}=1$ if $(i,j)\in\O$, $\sgn_{ij}=-1$ if $(j,i)\in\O$, and $\sgn_{ij}=0$ else.
Let $(d_i)_{i\in I}$ be a symmetrizer for $\bfc$ with $d_i\in\bbN^\times$
and $t\in\{1,2,3\}$ be the lacing number. 
We abbreviate $b_{ij}=d_ic_{ij}$ and $t_i=t/d_i$. 
For all $i,j\in I$ with $o_{ij}\neq 0$, we have 
$b_{ij}=b_{ji}=-\max(d_i,d_j)$ and
$$c_{ij}<0\Rightarrow (c_{ij}, c_{ji})=(-1,-t)\ \text{or}\ (-t,-1).$$
Unless specified otherwise, let $Q$ be the simply laced quiver associated with
the Cartan matrix $\bfc$, which is given by
$$Q_0=I
,\quad
Q_1=\{\alpha_{ij}:j\to i\,;\,(i,j)\in \O\}.$$
We write
$$
w-\bfc v=\sum_{i\in I}(w_i-\sum_{j\in I}c_{ij}v_j)\delta_i
,\quad
(\alpha_i^\vee,w-\bfc v)=w_i-\sum_{j\in I}c_{ij}v_j
 ,\quad
v,w\in\bbZ I.$$

\medskip

\subsubsection{Triple quiver varieties}\label{sec:triple quiver}
Let $T$ be the torus $T=(\bbC^\times)^{\overline Q_1}\times(\bbC^\times)^2.$
An element of $T$ is a tuple $(z_\alpha,z,z_2)$ where $\alpha$ runs in $\overline Q_1$.
The representation ring of $T$ is
$$R_T=\bbC[t_\alpha^{\pm 1}\,,\,q_1^{\pm 1}\,,\,q_2^{\pm 1}\,;\,\alpha\in \overline Q_1].$$
Fix $V\in\bfC$. 
Let $G_V=\prod_{i\in I}\GL(V_i)$.
Let $\frakg_V$ be the Lie algebra of $G_V$ and
$\frakg_V^\nil$ be the set of all nilpotent elements in $\frakg_V$.
We abbreviate $G_v=G_{\bbC^v}$ and $\frakg_v=\frakg_{\bbC^v}$.
Let $W\in\bfC$.
The group $G_V\times G_W\times T$ acts on $\widetilde\X(V,W)$ in the following way: 
the groups $G_V$, $G_W$ act by conjugaison, and $(z_\alpha,z_1,z_2)$ takes
the tuple $x=(\alpha_{ij},a_i,a_i^*,\varepsilon_i)$ to the tuple
\begin{align*}
(z_1^{b_{ij}}z_{\alpha_{ij}}\alpha_{ij}\,,\,
z_2^{b_{ij}}z_{\alpha_{ji}}\alpha_{ji}\,,\,
z_1^{-d_i}a_i\,,\,z_2^{-d_i}a^*_i\,,\,z_1^{d_i}z_2^{d_i}\varepsilon_i)
,\quad
(i,j)\in\O.
\end{align*}
Let $\xi:\bbC^\times\to T$ be the cocharacter $z\mapsto (1,z,z)$. 
We abbreviate $\bbC^\times=\xi(\bbC^\times)$. We have
\begin{align*}
z\cdot x=(z^{b_{ij}}\alpha_{ij}\,,\,z^{-d_i}a_i\,,\,z^{-d_i}a_i^*\,,\,z^{2d_i}\varepsilon_i)
,\quad
z\in\bbC^\times.
\end{align*}
A representation in $\widetilde\X(V,W)$ is said to be stable if it has no non-zero 
subrepresentations supported on $V$.
Let 
\begin{align*}
\widetilde\X(V,W)_s=\{x\in\widetilde\X(V,W)\,;\,x\ \text{is\ stable}\}.\end{align*}
We consider the categorical quotients 
$$\widetilde\frakM(v,W)=\widetilde\X(V,W)_s/G_V
,\quad
\widetilde\frakM_0(v,W)=\widetilde\X(V,W)/G_V.$$
Let $\ux$ denote both the orbit of $x$ in $\widetilde\frakM(v,W)$ if $x$ is stable, 
and the orbit of $x$ in $\widetilde\frakM_0(v,W)$ if the $G_V$-orbit of $x$ is closed.
We have a 
$G_W\times T$-invariant projective map
$\tilde\pi:\widetilde\frakM(v,W)\to\widetilde\frakM_0(v,W).$
We abbreviate
$$\widetilde\frakM(W)=\bigsqcup_{v\in\bbN I}\widetilde\frakM(v,W)
,\quad
\widetilde\frakM_0(W)=\bigcup_{v\in\bbN I}\widetilde\frakM_0(v,W).$$
The second colimit is the extension of representations by 0 to the complementary subspace.
This colimit may not stabilize.
Thus $\widetilde\frakM_0(W)$ is an ind-scheme,
while $\widetilde\frakM(W)$ is a scheme locally of finite type.
Let $\widetilde\frakL(W)$ be the fiber 
$$\widetilde\frakL(W)=\tilde\pi^{-1}(0).$$
We also define the Steinberg variety 
\begin{align}\label{Steinberg}
\widetilde\calZ(W)
=\widetilde\frakM(W)\times_{\widetilde\frakM_0(W)}\widetilde\frakM(W)
=\bigsqcup_{v^1,v^2\in\bbN I}\widetilde\calZ(v^1,v^2,W).
\end{align}
Let $G_V\times G_W\times \bbC^\times$ denote the image of the homomorphism
$1\times\xi$ into $G_V\times G_W\times T$.
The group $G_V\times G_W\times\bbC^\times$ 
acts on the variety $\widetilde\X(V,W)$,
and the group $G_W\times\bbC^\times$ on $\widetilde\frakM(W)$
and $\widetilde\frakM_0(W)$.

\medskip

\subsubsection{Graded triple quiver varieties}
\label{sec:GQVD}
Fix $V$, $W$ in $\bfC^\bullet$.
Let $G_V$ and $G_V^0$ be the automorphism groups of $V$
in $\bfC$ and $\bfC^\bullet$ respectively. Let
\begin{align*}\frakg_V=\bigoplus_{l\in\bbZ}\frakg^l_V
,\quad
\frakg^l_V=\bigoplus_{(i,k)\in I^\bullet}\Hom(V_{i,k},V_{i,k+l}).\end{align*}
The Lie algebras of $G_V$, $G_V^0$ are $\frakg_V$, $\frakg_V^0$.
Given a dimension vector $v\in\bbN I^\bullet$ we abbreviate
$\frakg_v=\frakg_V$, $\frakg^0_v=\frakg^0_V$, $G_v=G_V$ and $G_v^0=G_V^0$ 
where $V=\bbC^v$.
We equip the quiver $\widetilde Q_f$ with the grading
\begin{align}\label{degree}
\begin{split}
\deg:\widetilde Q_{f,1}\to\bbZ
,\quad
\alpha_{ij}\mapsto b_{ij}
,\quad
a_i,a_i^*\mapsto -d_i
,\quad
\varepsilon_i\mapsto 2d_i.
\end{split}
\end{align}
Let $\widetilde Q_f^\bullet$ be the corresponding graded quiver, defined as in \S\ref{sec:quiver basic1}.
The graded quiver varieties $\widetilde\frakM^\bullet(W)$ and
$\widetilde\frakM^\bullet_0(W)$ are respectively the moduli space of stable representations of 
$\widetilde Q_f^\bullet$ and the corresponding categorical quotient.
We define the 0-fiber $\widetilde\frakL^\bullet(W)$ and the Steinberg variety
$\widetilde\calZ^\bullet(W)$ in the obvious way.
There is an obvious inclusion
$\widetilde\frakM^\bullet(W)\subset\widetilde\frakM(W)$. 
More precisely, we can realize $\widetilde\frakM^\bullet(W)$ 
as the fixed points loci of a torus action
on $\widetilde\frakM(W)$ as follows.
Let $\sigma:\bbC^\times\to G_W$ be the cocharacter given by
\begin{align}\label{sigma1}\sigma(z)=\bigoplus_{(i,k)\in I^\bullet}z^k\id_{W_{i,k}}.\end{align}
Recall the subgroup $G_W\times\bbC^\times$ of $G_W\times T$.
Let $a$ be the cocharacter given by
\begin{align}\label{a}
a=(\sigma,\xi):\bbC^\times\to G_W\times \bbC^\times
\end{align}
and let $A\subset  G_W\times \bbC^\times$ be the one parameter subgroup
such that $A=a(\bbC^\times)$.

\begin{lemma}\label{lem:fixedpoints}
We have 
$\widetilde\frakM^\bullet(W)=\widetilde\frakM(W)^A$.
\end{lemma}

\begin{proof}
Let $\underline x$ be the class in $\widetilde\frakM(W)$ of a tuple
$x=(\alpha,a,a^*,\varepsilon)$ in $\widetilde\X(V,W)_s$. 
The group $G_V$ acts freely on the set $\widetilde\X(V,W)_s$ of stable representations.
Since the tuple $x$ is stable, we have 
$\underline x\in \widetilde\frakM(W)^A$ if and only if there is a cocharacter
$\gamma:\bbC^\times\to G_V$ such that $a(z)\cdot x=\gamma(z)\cdot x$ for each $z\in\bbC^\times$.
The cocharacter $a$ in \eqref{a} yields an $I^\bullet$-grading on $W$.
The cocharacter $\gamma$ yields an $I^\bullet$-grading on $V$.
The arrows
$\alpha_{ij}$,
$a_i,$ $a_i^*$, $\varepsilon_i$
have degree 
$b_{ij}$,
$-d_i$ and $2d_i$ respectively with respect to these gradings.
We deduce that $\widetilde\frakM^\bullet(W)=\widetilde\frakM(W)^A$.
\end{proof}

For each $v,w\in\bbN I^\bullet$ we write
\begin{align}\label{w-cv}
\begin{split}
w-\bfc v&=\sum_{(i,k)\in I^\bullet}\Big(w_{i,k}-v_{i,k+d_i}-v_{i,k-d_i}+
\sum_{c_{ij}=-1}v_{j,k}+\sum_{c_{ij}=-2}(v_{j,k-1}+v_{j,k+1})+\\
&\quad\sum_{c_{ij}=-3}(v_{j,k-2}+v_{j,k}+v_{j,k+2})\Big)\,\delta_{i,k}.
\end{split}
\end{align}

\medskip

\subsubsection{Triple quiver varieties with simple framing}\label{sec:hat}
\label{sec:triple quiver simple framing}
Recall that$$\widehat\X(V,W)=\{x\in\widetilde\X(V,W)\,;\,a^*=0\}.$$ 
We define $\widehat\X(V,W)_s=\widehat\X(V,W)\cap\widetilde\X(V,W)_s$ and
\begin{align}\label{hatM}
\begin{split}
\widehat\frakM(W)=\widehat\X(V,W)_s\,/\,G_V,\quad
\widehat\calZ(W)=\widetilde\calZ(W)\cap \widehat\frakM(W)^2.
\end{split}
\end{align}
We define the subvariety $\widehat\frakL(W)\subset\widehat\frakM(W)$ similarly.
The group $G_W\times\bbC^\times$ acts on $\widehat\frakM(W)$
and $\widehat\frakM_0(W)$ as in \S\ref{sec:triple quiver}.

\begin{lemma}\label{lem:nil}
We have 
\hfill
\begin{enumerate}[label=$\mathrm{(\alph*)}$,leftmargin=8mm,itemsep=2mm]
\item
$\widetilde\frakL(v,W)=\widehat\frakL(v,W)
=\{(\alpha,\varepsilon, a)\in\widehat\X(V,W)_s\,;\,(\alpha,\varepsilon)\in\widetilde\X^\nil(V)\}\,/\,G_V,$
\item
$\widetilde\frakL^\bullet(v,W)=\widehat\frakL^\bullet(v,W)
=\{(\alpha,\varepsilon, a)\in\widehat\X^\bullet(V,W)_s\,;\,
(\alpha,\varepsilon)\in\widetilde\X^{\bullet,\nil}(V)\}\,/\,G_V.$
\end{enumerate}
\end{lemma}

\begin{proof}
We only prove Part (a), because the proof of (b) is similar.
The vanishing ideal of the point $0$ in the function ring $\bbC[\widetilde\frakM_0(v,W)]$ is generated by
the functions $h_{A,M}$ and $h_N$  given, for any tuple $x=(\alpha,\varepsilon, a, a^*)$, by 
\begin{align*}
h_{A,M}(\ux)=\Tr_W(AaM(x)a^*)
,\quad
h_N(\ux)=\Tr_V(N(x))
\end{align*}
where $A$, $M$ and $N$ run through
$\frakg_W$, $\bbC\widetilde Q$ and
$\bbC\widetilde Q_+$.
For each  point $\ux\in\widetilde\frakL(W)$, we have
$$\Big(h_{A,M}\pi(\ux)=0\,,\,\forall A\in\frakg_W\,,\,\forall M\in\bbC\widetilde Q\Big)
\Rightarrow 
\Big(aM(x)a^*=0\,,\,\forall M\in\bbC\widetilde Q\Big),$$
hence $a^*=0$ because $x$ is stable.
Further, the Hilbert-Mumford criterion implies that
$$\Big(h_N\pi(\ux)=0\,,\,\forall N\in\bbC\widetilde Q_+\Big)
\Rightarrow
\Big((\alpha,\varepsilon)\text{\  is\ nilpotent}\Big).$$
We deduce that 
$$\widetilde\frakL(v,W)\subseteq\widehat\frakL(v,W)
\subseteq\{x\in\widehat\X(V,W)_s\,;\,(\alpha,\varepsilon)\ \text{is\ nilpotent}\}\,/\,G_V.$$
The reverse inclusion is obvious, because the functions
$h_{A,M}$ and $h_N$ vanish on each element of the right hand side.
\end{proof}

\medskip

\subsubsection{Hecke correspondences}\label{sec:Hecke}
The Hecke correspondence $\widetilde\frakP(W)$ is the scheme given by
\begin{align}\label{Hecke}
\widetilde\frakP(W)=\{(x,y,\tau)\in\widetilde\frakM(W)^2\times\Hom_{\widetilde Q_f}(x,y)\,;\,
\tau|_W=\id_W\}.
\end{align}
For each triple $(x,y,\tau)$ the map $\tau$ is injective, because the representation $x$ is stable.
For the same reason, there is a closed embedding x
$$i:\widetilde\frakP(W)\to\widetilde\frakM(W)^2
,\quad
(x,y,\tau)\mapsto(x,y).$$
Hence, we may write
$$\widetilde\frakP(W)=\{(x,y)\in\widetilde\frakM(W)^2\,;\,x\subset y\}.$$
The opposite Hecke correspondence is
\begin{align}\label{opHecke}
\widetilde\frakP(W)^\op=\{(x,y)\in\widetilde\frakM(W)^2\,;\,y\subset x\}.
\end{align}
Let $\Rep$ be the moduli stack of representations of $\widetilde Q$.
We have
$$\Rep=\bigsqcup_{v\in\bbN I}\Rep_v
 ,\quad
\Rep_v=\big[\widetilde\X(v)\,/\,G_v\big]$$
Let $\pi:\widetilde\frakP(W)\to\Rep$ be the stack homomorphism such that $\pi(x,y)=y/x.$
A representation in $\Rep$ is nilpotent if its image in the categorical quotient  is zero.
Let $\Rep^\nil\subset\Rep$ be the closed substack parametrizing the nilpotent representations.
We define the nilpotent Hecke correspondence to be the fiber product 
$\widetilde\frakP(W)^\nil=\widetilde\frakP(W)\times_{\Rep}\Rep^\nil.$
For  $v_1\leqslant v_2$, we write
\begin{align*}
\widetilde\frakP(v_1,v_2,W)&=
\widetilde\frakP(W)\cap\big(\widetilde\frakM(v_1,W)\times\widetilde\frakM(v_2,W)\big),\\\widetilde\frakP(v_2,v_1,W)&=
\widetilde\frakP(W)^\op\cap\big(\widetilde\frakM(v_2,W)\times\widetilde\frakM(v_1,W)\big).
\end{align*}
We also write
\begin{align*}
\widetilde\frakP(\delta_i,W)=\bigsqcup_{v\in\bbN I}\widetilde\frakP(v,v+\delta_i,W)
 ,\quad
\widetilde\frakP(-\delta_i,W)=\bigsqcup_{v\in\bbN I}\widetilde\frakP(v+\delta_i,v,W).
\end{align*}
Considering the quiver $\widehat Q_f$ instead of $\widetilde Q_f$, we consider the Hecke correspondence 
\begin{align}\label{hatHecke}
\widehat\frakP(W)=\widetilde\frakP(W)\cap \widehat\frakM(W)^2.
\end{align}
We define $\widehat\frakP(W)^\nil$, $\widehat\frakP(\delta_i,W)$ and
$\widehat\frakP(-\delta_i,W)$ in the obvious way.

\begin{lemma}
\hfill
\begin{enumerate}[label=$\mathrm{(\alph*)}$,leftmargin=8mm]
\item
The schemes $\widetilde\frakP(W)$ and $\widehat\frakP(W)$ are smooth and locally of finite type.
\item
The maps $\pi:\widetilde\frakP(W)\to\Rep$ and $\pi:\widehat\frakP(W)\to\Rep$ are flat.
 \item
The map $i$ takes $\widetilde\frakP(W)^\nil$ into $\widetilde\calZ(W)$,
and $\widehat\frakP(W)^\nil$ into $\widehat\calZ(W)$.
\end{enumerate}
\end{lemma}

\begin{proof}
We give the proof for $\widetilde\frakP(W)$ only, because the case of $\widehat\frakP(W)$
is similar.
We write
$$v_1\leqslant v_2\iff v_2-v_1\in\bbN I.$$
Let $P_{v_1,v_2}\subset G_{v_2}$ be the stabilizer of the flag $V_1\subset V_2$.
To prove (a) and (b), note that 
$$\widetilde\frakP(v_1,v_2,W)=\widetilde\X(v_1,v_2,W)_s\,/\,P_{v_1,v_2}$$
is the categorical quotient of
\begin{align*}
\widetilde\X(v_1,v_2,W)_s=\{y\in\widetilde\X(v_2,W)_s\,;\,y(V_1\oplus W)
\subseteq V_1\oplus W\}.
\end{align*}
The $P_{v_1,v_2}$-action is proper and free 
because the point $y$ is stable.
Part (c) follows from the Hilbert-Mumford criterion.
For any pair $(x,y)$ in $\widetilde\X(v_1,W)\times\widetilde\X(v_2,W)$
representing a point in $\widetilde\frakP(W)^\nil$ there is a 1-parameter subgroup
$\lambda$ in $G_{v_2}$ such that 
$\lim_{t\to\infty}\lambda(t)\cdot y=x\oplus 0.$
Hence, we have $\pi(x)=\pi(y)$ in $\widetilde\frakM_0(W)$.
\end{proof}

\subsubsection{Potentials}\label{sec:potential}
We equip the quiver $\widetilde Q_f$ with the potentials $\bfw_1$ or $\bfw_2$ such that
\begin{align*}
\begin{split}
\bfw_1=
\bfw_2+\sum_{i\in I}\varepsilon_i\,a_i^*\,a_i
,\quad
\bfw_2=
\sum_{i,j\in I}\sgn_{ij}\,\varepsilon_i^{-c_{ij}}\,\alpha_{ij}\,\alpha_{ji}.
\end{split}
\end{align*}
The potential $\bfw_2$ appears already in \cite[\S 1.7.3]{GLS17}.
Both potentials are homogeneous of degree zero.
We equip $\widetilde Q_f^\bullet$ with the potentials $\bfw^\bullet_1$ or $\bfw^\bullet_2$ where
\begin{align*}
\begin{split}
\bfw^\bullet_1&=
\bfw^\bullet_2+\sum_{k\in\bbZ}\sum_{i\in I}\varepsilon_{i,k-2d_i}a^*_{i,k-d_i}a_{i,k},\\
\bfw^\bullet_2&=
\sum_{k\in\bbZ}\sum_{(i,j)\in\O}
\big(\varepsilon_{i,k-2d_i}\cdots\varepsilon_{i,k+2b_{ji}}\alpha_{ij,k+b_{ji}}\alpha_{ji,k}
-\varepsilon_{j,k-b_{jj}}\cdots\varepsilon_{j,k+2b_{ij}}\alpha_{ji,k+b_{ij}}\alpha_{ij,k}\big).
\end{split}
\end{align*}
Taking the traces of $\bfw_1$, $\bfw_2$ we get the functions
$$\tilde f_1,\tilde f_2:\widetilde\frakM(W)\to\bbC.$$ 
Let $\hat f_1$, $\hat f_2$ be their restriction to $\widehat\frakM(V,W)$.
Similarly, taking the traces of $\bfw_1^\bullet$, $\bfw_2^\bullet$ we get the functions
$$\tilde f_1^\bullet,\tilde f_2^\bullet:\widetilde\frakM^\bullet(W)\to\bbC.$$ 
We define $\hat f_1^\bullet$, $\hat f_2^\bullet$ in the obvious way.
Finally, taking the trace of $\bfw_2$ we get the function
$$h:\Rep\to\bbC.$$
By \S\ref{sec:Hecke} we have the following maps
$$\xymatrix{\Rep&\ar[l]_\pi\widetilde\frakP(W)\ar[r]^i&\widetilde\frakM(W)^2}$$

\begin{lemma}\label{lem:fh}
Let $f= f_1$ or $f_2$.
We have $i^*(\tilde f^{(2)})=\pi^*(h)$.
 \end{lemma}

\begin{proof}
Fix point $(x,y)\in\widetilde\frakP(v_1,v_2,W)$. We have 
$$y=(\alpha,\varepsilon, a,a^*)\in\widetilde\frakM(v_2,W)
,\quad
x=(\alpha|_{V_1},\varepsilon|_{V_1}, a,a^*)\in\widetilde\frakM(v_1,W)
,\quad
a^*(W)\subset V_1\subset V_2.$$ 
Then, we have $\pi(x,y)=y/x=(\alpha|_{V_2/V_1},\varepsilon|_{V_2/V_1})$.
Further, either $\widetilde f^{(2)}(x,y)=\tilde f_1(y)-\tilde f_1(x)=\tilde f_2(y/x)=h(y/x)$ or 
$\widetilde f^{(2)}(x,y)=\tilde f_2(y)-\tilde f_2(x)=\tilde f_2(y/x)=h(y/x)$.
\end{proof}

\subsubsection{Universal bundles}\label{sec:univbdl}
Let $\calV=\bigoplus_{i\in I}\calV_i$ and 
$\calW=\bigoplus_{i\in I}\calW_i$ denote both the tautological bundles on
$\widehat\frakM(W)$ and $\widetilde\frakM(W)$
and their classes in $K_{G_W\times\bbC^\times}\big(\widehat\frakM(W)\big)$ and
$K_{G_W\times\bbC^\times}\big(\widetilde\frakM(W)\big)$.
Given an orientation as in \S\ref{sec:or}, we define
\begin{align}\label{Vcirc}
\calV_{\circ i}=\bigoplus_{c_{ij}<0}\calV_j=\calV_{+i}\oplus\calV_{-i}
,\quad
\calV_{- i}=\bigoplus_{c_{ij},\sgn_{ij}<0}\calV_j
,\quad
v_{\circ i}=\sum_{c_{ij}<0}v_j=v_{+i}+v_{-i}.
\end{align}
Let $\calV^-_i$ and $\calV^+_i$ be the pull-back of the tautological vector bundle 
$\calV_i$ on $\widetilde\frakM(W)$
by the first and second projection 
$\widetilde\frakP(\delta_i,W)\to\widetilde\frakM(W)$.
Switching both components of $\widetilde\frakM(W)^2$,
we define similarly the vector bundles
$\calV^-_i$, $\calV^+_i$ on the Hecke correspondence $\widetilde\frakP(-\delta_i,W)$.
Let $\calL_i$ denote the invertible sheaf $\calV^+_i/\,\calV^-_i$ on the Hecke correspondence
$\widetilde\frakP(\pm\delta_i,W)$,
and its pushforward by the closed embedding into $\widetilde\calZ(W)$.
We define the bundles $\calV^-_i$, $\calV^+_i$, $\calL_i$
on $\widehat\frakP(\pm\delta_i,W)$ or $\widehat\calZ(W)$ in a similar way.

\medskip

\subsection{Quiver Grassmanians}\label{sec:QGr}

For each vertex $i$ let $e_i$ be the length 0 path in $\bbC \widetilde Q$ supported on $i$.
The generalized preprojective algebra $\widetilde\Pi$ is the Jacobian algebra of the quiver with potential
$(\widetilde Q,\bfw_2)$. See \cite[def.~1.4]{GLS17}. More precisely, 
it is the quotient of the path algebra $\bbC\widetilde Q$ of $\widetilde Q$
by the two-sided ideal generated by all cyclic derivations of $\bfw_2$, i.e., by the elements
\begin{align}\label{critf2}
\varepsilon_i^{-c_{ij}}\alpha_{ij}-\alpha_{ij}\varepsilon_j^{-c_{ji}}
,\quad 
\sum_{i,j\in I}\sum_{k=0}^{-c_{ij}-1}\sgn_{ij}\varepsilon_i^k\alpha_{ij}\alpha_{ji}\varepsilon_i^{-c_{ij}-1-k}.
\end{align}
We abbreviate $\alpha$, $\varepsilon$, $\omega$ for the following elements in $\widetilde\Pi$
\begin{align}\label{aeo}
\alpha=\sum_{i,j\in I}\alpha_{ij}
,\quad
\varepsilon=\sum_{i\in I}\varepsilon_i
,\quad
\omega=\sum_{i\in I}\varepsilon_i^{t_i}.
\end{align}
The element $\omega$ is central  in $\widetilde\Pi$.
For every positive integer $l$ we set 
$$\widetilde\Pi_l=\widetilde\Pi/(\widetilde\Pi\omega^l).$$
The grading \eqref{degree} 
yields a $\bbZ$-grading on the algebras  $\widetilde\Pi$ 
and $\widetilde\Pi_l$.
We abbreviate 
\begin{align}\label{Hl}
H=\bigoplus_{i\in I}\bbC[\varepsilon_i]
,\quad
H_l=H/H\varepsilon^l.
\end{align}
In particular $H_1$ is the semisimple algebra spanned by the $e_i$'s.
The ring $\widetilde\Pi$ is an $H$-bimodule.
Let $\tau:\widetilde\Pi\to\widetilde\Pi^\op$ be the algebra automorphism
such that 
$$\tau(e_i)=e_i
,\quad
\tau(\varepsilon_i)=\varepsilon_i
,\quad
\tau(\alpha_{ij})=\alpha_{ji}.$$
We equip the category of the graded $\widetilde\Pi$-modules
with the grading shift functor $[1]$ and the duality functor such that
$(M^\vee)_k=(M_{-k})^\vee$ with the transposed $\widetilde\Pi$-action
twisted by $\tau$
for each object $M=\bigoplus_{k\in\bbZ}M_k$.
A (graded) $\widetilde\Pi$-module
is nilpotent if it is killed by a power of the augmentation ideal.
Let $\bfD$ and $\bfD^\bullet$ be the categories of finite dimensional $\widetilde\Pi$-modules
and finite dimensional graded $\widetilde\Pi$-modules.
Let $\bfD^\nil$ and $\bfD^{\bullet,\nil}$ be the subcategories of nilpotent modules.
By \cite[\S11]{GLS17} the two-sided ideal of $\widetilde\Pi$ 
generated by the $\alpha_{ij}$'s is nilpotent. 
Since the element $\omega$ is homogeneous of degree $2t$,
the element $\varepsilon$ acts nilpotently on any module in $\bfD^\bullet$.
We deduce that $$\bfD^\bullet=\bfD^{\bullet,\nil}.$$
Similarly, a $\widetilde\Pi$-module in $\bfD$ lies in $\bfD^\nil$ if and only if $\varepsilon$ acts nilpotently.
We consider the following finite dimensional graded $\widetilde\Pi$-modules
\begin{align}\label{K}
K_{i,k,l}=\big(\widetilde\Pi e_i/\widetilde\Pi\varepsilon_i^l\big)^\vee[-k-ld_i]
,\quad
i\in I,\, l\in\bbN^\times,\, k\in\bbZ.
\end{align}
Let $\widetilde\Pi^\bullet$ be the Jacobian algebra of the quiver with potential 
$(\widetilde Q^\bullet,\bfw^\bullet_2)$.
By \cite[prop.~4.4, 5.1]{FM21}, a graded $\widetilde\Pi$-module is the same as a 
$\widetilde\Pi^\bullet$-module, and,
under this equivalence, the graded $\widetilde\Pi$-module $K_{i,k,l}$
is the same as the generic kernel associated in \cite{HL16} with the Kirillov-Reshetikhin module $KR_{i,k,l}$. 
Let $\Pi(\infty)$ be the 
projective limit of the $\widetilde\Pi_l$'s, see \cite[\S4.4]{FM21}.
We consider the graded $\widetilde\Pi$-modules $I_{i,k}$ given by
\begin{align}\label{I}
I_i=(\Pi(\infty) e_i)^\vee=\bigcup_{l>0}(\widetilde\Pi_le_i)^\vee
,\quad
I_{i,k}=I_i[-k].
\end{align}
It is an inductive limit of graded $\widetilde\Pi$-modules in $\bfD^\bullet$.
Given a module $M\in\bfD$ and a dimension vector $v\in\bbN I$, let  $\widetilde\Gr_v(M)$
be the Grassmanian of all $\widetilde\Pi$-submodules 
of dimension $v$. 
Given a graded $\widetilde\Pi$-module $M\in\bfD^\bullet$  and a dimension vector $v\in\bbN I^\bullet$, 
let $\widetilde\Gr^\bullet_v(M)$ be the Grassmanian of all graded $\widetilde\Pi$-submodules 
of dimension $v$. We define
$$\widetilde\Gr(M)=\bigsqcup_{v\in\bbN I}\widetilde\Gr_v(M)
,\quad
\widetilde\Gr^\bullet(M)=
\bigsqcup_{v\in\bbN {I^\bullet}}\widetilde\Gr^\bullet_v(M).$$
Both Grassmanians are complex varieties in the obvious way.
We also consider the Grassmanian $\widetilde\Gr^\bullet_v(M)$  for
$M$ a finite direct sum of $\widetilde\Pi$-modules $I_{i,k}$ as above. 
Since the subspace $M_k=\bigoplus_{i\in I}M_{i,k}$ of $M$ is finite dimensional by
\cite[prop.~4.5]{FM21}, this Grassmanian is also a complex variety, hence 
$\widetilde\Gr^\bullet(M)$ is a complex scheme locally of finite type.

\medskip

\subsection{Quantum loop groups}\label{sec:N00}
Let $Q$ be a Dynkin quiver and $\frakg$ be the corresponding complex simple Lie algebra.
Fix $\zeta\in\bbC^\times$. Although many of our results hold for arbitrary $\zeta$,
we assume that $\zeta$ is not a root of unity.

Let $\U_R(L\frakg)$ be the 
integral $R$-form of the quantum loop group of type $\bfc$, see Appendix A.
We define
$$\U_F(L\frakg)=\U_R(L\frakg)\otimes_{R}F
,\quad
\U_\zeta(L\frakg)=\U_R(L\frakg)|_\zeta,$$
where $(-)|_\zeta$ is the specialization along the map $R\to\bbC$, $q\mapsto\zeta$.
The $F$-algebra $\U_F(L\frakg)$ is generated by $x^\pm_{i,n}$, $\psi^\pm_{i,\pm m}$ with 
$n\in\bbZ$, $m\in\bbN$ satisfying some well-known relations, see \S\ref{sec:qg}.
The $R$-subalgebra $\U_R(L\frakg)$ of $\U_F(L\frakg)$
is generated by the quantum divided powers 
$(x^\pm_{i,n})^{[m]}$ with
$i\in I$, $n\in\bbZ$ and $m\in\bbN^\times$
and by the coefficients $h_{i,\pm m}$ as in \S\ref{sec:qg}.

For each $w\in\bbZ I$ let $\U_F^{-w}(L\frakg)$ be the simply-connected $(0,-w)$-shifted quantum loop group defined in \cite[\S 5.1]{FT19}.
Let $\U^{-w}_F(L\frakg)$ be its integral $R$-form and $\U^{-w}_\zeta(L\frakg)$
be its specialization, see \S\ref{sec:sqg} for more details.

\medskip

\section{Critical convolution algebras of triple quivers with potentials}\label{sec:CAT}

This section relates critical convolution algebras to quantum loop groups and
shifted quantum loop groups. The main results are
Theorems \ref{thm:notshifted}, \ref{thm:shifted} and
Corollaries \ref{cor:notshifted2}, \ref{cor:shifted2}.
We will use the same notation as in \S\ref{sec:quiver varieties}.

\subsection{K-theoretic critical convolution algebras and quantum loop groups}
\label{sec:w1}
Let $W\in\bfC$.
Fix a nilpotent element $\gamma_i\in\frakg_{W_i}^\nil$ for each $i\in I$, and set
$\gamma=\bigoplus_{i\in I}\gamma_i$.
Fix a cocharacter $\sigma:\bbC^\times\to G_W$ such that 
$\Ad_{\sigma(z)}(\gamma_i)=z^{2d_i}\gamma_i.$
We equip $W$ with the $I^\bullet$-grading \eqref{sigma1} for which
the operator $\gamma_i$ is 
homogeneous of degree $2d_i$.
Recall that
$a=(\sigma,\xi)$
and
$A=a(\bbC^\times).$
We define the $A$-invariant function
\begin{align*}
f_\gamma:\widetilde\frakM(W)\to\bbC
,\quad
\ux\mapsto\sum_{i\in I} \Tr(\gamma_i a_i a_i^*).
\end{align*}
Let $f_\gamma^\bullet:\widetilde\frakM^\bullet(W)\to\bbC$ be the restriction of $f_\gamma$ to the
$A$-fixed points locus.
Set
\begin{align}\label{f6}
\tilde f_{\!\gamma}:\widetilde\frakM(W)\to\bbC
,\quad
\tilde f^\bullet_{\!\gamma}:\widetilde\frakM^\bullet(W)\to\bbC
,\quad
\tilde f_{\!\gamma}=\tilde f_1-f_\gamma
,\quad
\tilde f^\bullet_{\!\gamma}=\tilde f^\bullet_1-f^\bullet_\gamma.
\end{align}
Recall that $R=R_A$ and $F=F_A$. Let 
$$K_A\big(\widetilde\frakM(W)^2,(\tilde f_\gamma)^{(2)}\big)
_{\widetilde\calZ(W)}\,/\,\tor\,\subset\,
K_A\big(\widetilde\frakM(W)^2,(\tilde f_\gamma)^{(2)}\big)
_{\widetilde\calZ(W)}\otimes_RF$$
be the image of 
$K_A\big(\widetilde\frakM(W)^2,(\tilde f_\gamma)^{(2)}\big)
_{\widetilde\calZ(W)}$
in the right hand side.
We define 
$$K_A(\widetilde\frakM(W),\tilde f_\gamma)/\tor
,\quad
K_A(\widetilde\frakM(W),\tilde f_\gamma)_{\widetilde\frakL(W)}/\tor$$
in a similar way.
We will prove the following.

\smallskip

\begin{theorem}\label{thm:notshifted}
\hfill
\begin{enumerate}[label=$\mathrm{(\alph*)}$,leftmargin=8mm]
\item
There is an $R$-algebra homomorphism
$\U_R(L\frakg)\to 
K_A\big(\widetilde\frakM(W)^2,(\tilde f_\gamma)^{(2)}\big)_{\widetilde\calZ(W)}/\tor$.
\item
The $R$-algebra $\U_R(L\frakg)$ acts on 
$K_A(\widetilde\frakM(W),\tilde f_\gamma)/\tor$ and
$K_A(\widetilde\frakM(W),\tilde f_\gamma)_{\widetilde\frakL(W)}/\tor.$
\end{enumerate}
\end{theorem}

Using Propositions \ref{prop:critalg1},  \ref{prop:TTK}, the theorem implies the following.

\begin{corollary}\label{cor:notshifted2}
The
algebra $\U_\zeta(L\frakg)$ acts on 
$K(\widetilde\frakM^\bullet(W),\tilde f^\bullet_\gamma)_{\underline{\widetilde\frakL^\bullet(W)}}$ and
$K(\widetilde\frakM^\bullet(W),\tilde f^\bullet_\gamma)$.
\qed
\end{corollary}

\begin{remark}
\hfill
\begin{enumerate}[label=$\mathrm{(\alph*)}$,leftmargin=8mm]
\item
Using Proposition \ref{prop:critalg1}, one can also prove that $\U_\zeta(L\frakg)$ acts on 
$K(\widetilde\frakM^\bullet(W),\tilde f_\gamma^\bullet)_{\widetilde\frakL^\bullet(W)}.$
\item
In Theorem \ref{thm:notshifted} the algebra structure on the right hand side is given by the convolution product $\star$.
We omit $\star$  if no confusion is possible.
\end{enumerate}
\end{remark}

\medskip

\subsection{Proof of Theorem $\ref{thm:notshifted}$}
\subsubsection{Definition of the homomorphism}\label{sec:421}
We first define an $F$-algebra homomorphism
\begin{align}\label{map8}
\U_F(L\frakg)\to 
K_A\big(\widetilde \frakM(W)^2,(\tilde f_\gamma)^{(2)}\big)_{\underline{\tilde\calZ(W)}}\otimes_RF,
\end{align}
then we apply Proposition \ref{prop:critalg1}(h).
The $F$-algebra $\U_F(L\frakg)$ is generated by the Fourier coefficients of the series
\begin{align*}
x^\pm_i(u)=\sum_{n\in\bbZ}x^\pm_{i,n}\,u^{-n}
 ,\quad
\psi^+_i(u)=\sum_{n\in\bbN}\psi^+_{i,n}\,u^{-n}
 ,\quad
\psi^-_i(u)=\sum_{n\in\bbN}\psi^-_{i,-n}\,u^{n}
\end{align*}
modulo the defining relations (A.2) to (A.7) and (A.9) in \S\ref{sec:QG}.
By Proposition \ref{prop:critalg1} we have the $R$-algebra homomorphism  
\begin{align}\label{form4g}
\Upsilon:
K^A(\widetilde\calZ(W))\to
K_A(\widetilde\frakM(W)^2,(\tilde f_\gamma)^{(2)})_{\underline{\widetilde\calZ(W)}}
\end{align}
Composing the pushforward by the diagonal embedding with the map
$\Upsilon$, we get the map
\begin{align}\label{deltag}
\Delta:K_A\big(\widetilde\frakM(W)\big)\to
K_A\big(\widetilde\frakM(W)^2,(\tilde f_\gamma)^{(2)}\big)_{\underline{\widetilde\calZ(W)}}
\end{align}
Let $\psi^m$ be the Adams operation in 
$K_A\big(\widetilde\frakM(W)\big).$
We define the following classes in $K_A\big(\widetilde\frakM(W)\big)$
\begin{align}\label{H}
\begin{split}
\calH_{i,1}&=\calW_i-\sum_j[c_{ij}]_{q_i}\calV_j
 ,\quad
\calH_{i,-1}=\calW_i^\vee-\sum_j[c_{ij}]_{q_i}\calV_j^\vee,\\
\calH_{i,\pm m}&=\frac{[m]_{q_i}}{m}\,\psi^m(\calH_{i,\pm 1})
,\quad
m>0.
\end{split}
\end{align}
We assign to $\psi_{i,n}^\pm$ the image by $\Delta$
of the coefficient of $u^{-n}$ in
$K_A\big(\widetilde\frakM(W)\big)$
in the formal series
\begin{align}\label{psi+-}
q_i^{\pm(\alpha_i^\vee,w-\bfc v)}
\,\exp\Big(\pm (q_i-q_i^{-1})\sum_{m>0}\calH_{i,\pm m}u^{\mp m}\Big)
\end{align}
Recall that $\calL_i$ denotes both the tautological invertible sheaf on the Hecke correspondence 
$\widetilde\frakP(\pm\delta_i,W)$, and its pushforward by the closed embedding into $\widetilde\calZ(W)$,
see \S\ref{sec:univbdl}.
We define 
$$A_{i,n}^\pm=\Upsilon\big(\calL_i^{\otimes n}\big)
,\quad
A_i^\pm(u)=\sum_{n\in\bbZ}A^\pm_{i,n}u^{-n}$$
in 
$K_A(\widetilde\frakM(W)^2,(\tilde f_\gamma)^{(2)})_{\underline{\widetilde\calZ(W)}}$
and
$K_A(\widetilde\frakM(W)^2,(\tilde f_\gamma)^{(2)})_{\underline{\widetilde\calZ(W)}}[[u,u^{-1}]]$.
Next, let $\calW_i$, $\calV_{\circ i}$ be the classes in $K_A\big(\widetilde\frakM(W)\big)$
defined in \S\ref{sec:univbdl}, and let  $v_{\circ i}$, $v_{\pm i}$ be as in \eqref{Vcirc}.
We assign to $x_{i,n}^\pm$ the following element in
$K_A(\widetilde\frakM(W)^2,(\tilde f_\gamma)^{(2)})_{\underline{\widetilde\calZ(W)}}$
\begin{align}\label{xpm}
\begin{split}
x_{i,n}^+\mapsto (-1)^{w_i+v_{+i}}A^+_{i,n}\star\Delta(\det(\calW_i\oplus\calV_{\circ i}))
,\quad
x_{i,n}^-\mapsto
(-1)^{v_{-i}}q_i^{-1}\, A^-_{i,n-w_i-v_{\circ i}}.
\end{split}
\end{align}
To prove that the assignments \eqref{psi+-} and
\eqref{xpm} give a well-defined morphism \eqref{map8}, 
we must check that the images of $x_{i,n}^\pm$ and $\psi_{i,n}^\pm$ in the algebra
$K_A(\widetilde\frakM(W)^2,(\tilde f_\gamma)^{(2)})_{\underline{\widetilde\calZ(W)}}$
satisfy the relations $\mathrm{(A.2)}$ to $\mathrm{(A.7)}$.
We only check the relations (A.5), (A.6) and (A.7) here. 
The other ones are obvious and are similar to relations in \cite{N00}.
By definition, the classes $A_{i,n}^\pm$, $x_{i,n}^\pm$ and $\psi_{i,n}^\pm$ have obvious lifting
to $K^A(\widetilde\calZ(W))$ relatively to the map $\Upsilon$ in \eqref{form4g}.
Let $A_{i,n}^\pm$, $x_{i,n}^\pm$ and $\psi_{i,n}^\pm$ denote also these liftings.

\subsubsection{Proof of the relation $\mathrm{(A.6)}$ for $i=j$}\label{sec:A6=}
By \eqref{form4g} it is enough to check
the relation in the algebra
$K^A(\widetilde\calZ(W))$.
We will prove the relation via a reduction to the case $A_1$, using
a reduction to the fixed points locus of a torus action.
To do so, we consider the subquiver 
$\widetilde Q_{f,\neq i}$ of $\widetilde Q_f$ such that
$$(\widetilde Q_{f,\neq i})_0=(\widetilde Q_f)_0\,\setminus\,\{i,i'\}
 ,\quad
(\widetilde Q_{f,\neq i})_1=\{h\in (\widetilde Q_f)_1\,;\,s(h),t(h)\neq i\}.$$
The representation variety of $\widetilde Q_f$ decomposes as
$$
\widetilde\X(V,W)=
\widetilde\X(V_i,W_i\oplus V_{\circ i})\times\X_{\widetilde Q_{f,\neq i}}(V_{\neq i},W_{\neq i})
$$
where 
$$V=\bbC^v
 ,\quad
V_{\circ i}=\bigoplus_{{c_{ij}<0}}V_j
 ,\quad
V_{\neq i}=\bigoplus_{j\neq i}V_j
 ,\quad
W_{\neq i}=\bigoplus_{j\neq i}W_j$$
$$\widetilde\X(V_i,W_i\oplus V_{\circ i})=\{(\alpha_{ji}, a_i,\alpha_{ij}, a_i^*,\varepsilon_i)\,;\,c_{ij}<0\}
,\quad
\X_{\widetilde Q_{f,\neq i}}(V_{\neq i},W_{\neq i})=\{(\alpha_{jk}, a_j, a_j^*,\varepsilon_j)\,;\,j,k\neq i\}.$$
We define the varieties
$$
\frakM(v,W)_\heartsuit=\X(V,W)_\heartsuit\,/\,G_{V_i}
,\quad
\frakM(v,W)_\spadesuit=\widetilde\X(V,W)_s\,/\,G_{V_i}$$
where
\begin{align*}
\X(V,W)_\heartsuit=\widetilde\X(V_i,W_i\oplus V_{\circ i})_s\times\X_{\widetilde Q_{f,\neq i}}(V_{\neq i},W_{\neq i}).
\end{align*}
We have the following diagram
$$\xymatrix{\widetilde\frakM(v_i,W_i\oplus V_{\circ i})&\ar[l]_-p\frakM(v,W)_\heartsuit&\ar[l]_-{\iota}\frakM(v,W)_\spadesuit\ar[r]^\pi
&\widetilde\frakM(v,W)}$$
The map $p$ is induced by the first projection $\X(V,W)_\heartsuit\to\widetilde\X(V_i,W_i\oplus V_{\circ i})_s$.
It is a vector bundle.
The map $\iota$ is an open embedding.
The map $\pi$ is a principal bundle.
We abbreviate
$$\frakM(W)_\diamondsuit=\bigsqcup_{v_i\in\bbN}\widetilde\frakM(v_i,W_i\oplus V_{\circ i})
,\quad
\calZ(W)_\diamondsuit=\bigsqcup_{v_i^1,v_i^2\in\bbN}
\widetilde\calZ(v_i^1,v_i^2,W_i\oplus V_{\circ i})
,\quad
G_\diamondsuit=G_{W_i}\times G_{V_{\circ i}}.$$
By \cite[\S 11.3]{N00}, the diagram above yields an algebra homomorphism
$$K^{G_\diamondsuit\times T}\big(\calZ(W)_\diamondsuit\big)
\to
K^{G_W\times T}\big(\widetilde\calZ(W)\big)$$
Composing it with the algebra homomorphism $\Upsilon$ in \eqref{form4g}
yields the algebra homomorphism 
\begin{align}\label{form19}
K^{G_\diamondsuit\times T}\big(\calZ(W)_\diamondsuit\big)\to 
K_A\big(\widetilde\frakM(W)^2,(\tilde f_2)^{(2)}\big)_{\underline{\widetilde\calZ(W)}}.
\end{align}
Thus it is enough to check the relations in the left hand side of \eqref{form19}.
Note that $\frakM(W)_\diamondsuit$ is a triple quiver variety of type $A_1$.
Fix a maximal torus $T_\diamondsuit$ in $G_\diamondsuit$.
From the description of the fixed points locus $\frakM(W)_\diamondsuit^{T_\diamondsuit\times T}$
in \eqref{fixpoints}, it is not difficult to see that
$$\calZ(W)_\diamondsuit^{T_\diamondsuit\times T}=
\frakM(W)_\diamondsuit^{T_\diamondsuit\times T}\times\frakM(W)_\diamondsuit^{T_\diamondsuit\times T}$$
Hence, using Lemma \ref{lem:(T)} and the localization theorem in K-theory, we deduce that
there is an algebra embedding
$$K^{G_\diamondsuit\times T}\big(\calZ(W)_\diamondsuit\big)
\subset
\End_{F_{T_\diamondsuit\times T}}\Big(K^{T_\diamondsuit\times T}\big(\frakM(W)_\diamondsuit\big)
\otimes_{R_{T_\diamondsuit\times T}}F_{T_\diamondsuit\times T}\Big).$$
Hence it is enough to check the relations in the right hand side.
To do this, note that the classes of the tangent bundles in equivariant K-theory are
\begin{align*}
\begin{split}
T\frakM(W)_\diamondsuit&=(q_1^{-d_i}q_2^{-d_i}-1)\End(\calV_i)+q_1^{d_i}\Hom(\calV_i,\calW_i)+
q_2^{d_i}\Hom(\calW_i,\calV_i)
+\sum_{o_{ij}=1}q_1^{-b_{ij}}\Hom(\calV_i,\calV_j)\\
&\quad+\sum_{o_{ij}=1}q_2^{-b_{ij}}\Hom(\calV_j,\calV_i)\\
T\frakP(W)_\diamondsuit&=(q_1^{-d_i}q_2^{-d_i}-1)\calP_i+q_1^{d_i}\Hom(\calV^+_i,\calW_i)+
q_2^{d_i}\Hom(\calW_i,\calV^-_i)
+\sum_{o_{ij}=1}q_1^{-b_{ij}}\Hom(\calV^+_i,\calV_j)\\
&\quad+\sum_{o_{ij}=1}q_2^{-b_{ij}}\Hom(\calV_j,\calV_i^-)
\end{split}
\end{align*}
where $\calP_i$ is given by
\begin{align}\label{Pi}
\calP_i=\End(\calV^-_i)+\Hom(\calL_i,\calV^+_i)=\End(\calV^+_i)-\Hom(\calV^-_i,\calL_i).
\end{align}
The fixed points in $\frakM(W)_\diamondsuit$ for the action of the torus 
$T_\diamondsuit\times T$
are labelled by tuples of nonnegative integers
as in \eqref{fixpoints}.
We write
\begin{align*}
\frakM(v_i,W)_\diamondsuit^{T_\diamondsuit\times T}&=
\{\underline x_\lambda\,;\,\lambda\in\bbN^{w_i+v_{\circ i}}\,,\,|\lambda|=v_i\},\\
\frakM(v_i+\delta_i,W)_\diamondsuit^{T_\diamondsuit\times T}&=
\{\underline x_\mu\,;\,\mu\in\bbN^{w_i+v_{\circ i}}\,,\,|\mu|=v_i+\delta_i\}.
\end{align*}
Restricting the universal vector bundles to the fixed points, we abbreviate
\begin{align}\label{VL}
\calV_\lambda=\calV|_{\{\underline x_\lambda\}}
,\quad
\calL_{\lambda,\mu}=\calL|_{\{(\underline x_\lambda,\,\underline x_\mu)\}}.
\end{align}
Let 
$[\lambda]$ be the fundamental class
of $\{\underline x_\lambda\}$.
For any linear operator $A$, 
let $\langle \lambda|A|\mu\rangle$ be the coefficient of the basis element $[\lambda]$ 
in the expansion of $A[\mu]$ in the
basis $\{[\lambda]\,;\,\lambda\in\bbN^w\}$.
Recall the universal vector bundles
$\calV^+$, $\calV^-$ and $\calL=\calV^+/\,\calV^-$ on the Hecke correspondences
and on $\widetilde\calZ(W)$ introduced in \ref{sec:univbdl}.
To simplify the writing, from now one we will assume that $q_1=q_2$ and we write $q$ for both of them.
In other words, we compute the matrix coefficients in $F_{T_\diamondsuit\times\bbC^\times}$ rather than in
$F_{T_\diamondsuit\times T}$.
See  \S\ref{sec:triple quiver} for details.
Recall that $q_i=q^{d_i}$ for all $i\in I$.
Using the formula for the tangent vector bundle given above, 
we get the following formulas 
\begin{align*}
\begin{split}
\langle \lambda |A^-_{i,n}|\mu\rangle
&=\calL_{i,\lambda,\mu}^{\otimes n}\otimes
\Lambda_{-1}\Big((q^{-2}_i-1)\calV^\vee_{i,\lambda}\otimes\calL_{i,\lambda,\mu}+
q_i\calL_{i,\lambda,\mu}\otimes\calW_i^\vee+
\sum_{c_{ij}<0}q_i^{-c_{ij}}\calV_{j}^\vee\otimes\calL_{i,\lambda,\mu}\Big)\\
&=\ev_{u=\calL_{\lambda,\mu}}\Big(u^n
\Lambda_{-u}\big((q^{-2}_i-1)\calV^\vee_{i,\lambda}+
q_i\calW_i^\vee+
\sum_{c_{ij}<0}q_i^{-c_{ij}}\calV_{j}^\vee\big)\Big)\\
\langle \mu|A^+_{i,m}|\lambda\rangle
&=\calL_{i,\lambda,\mu}^{\otimes m}\otimes\Lambda_{-1}
\Big((1-q_i^{-2})\calL^\vee_{i,\lambda,\mu}\otimes\calV_{i,\mu}-
q_i\calL^\vee_{i,\lambda,\mu}\otimes\calW_i-
\sum_{c_{ij}<0}q_i^{-c_{ij}}\calL^\vee_{i,\lambda,\mu}\otimes\calV_{j}\Big)\\
&=(1-q_i^{-2})^{-1}\Res_{u=\calL_{\lambda,\mu}}\Big(u^{m-1}\Lambda_{-u^{-1}}
\big((1-q_i^{-2})\calV_{i,\lambda}
-q_i\calW_i-\sum_{c_{ij}<0}q_i^{-c_{ij}}\calV_{j}\big)\Big)
\end{split}
\end{align*}
Fix a second fixed point 
$\underline x_{\lambda'}$ in
$\frakM(v_i,W)_\diamondsuit$.
If $\lambda\neq\lambda'$, then we deduce that
\begin{align*}
\langle\lambda'| A^-_{i,n}A^+_{i,m}|\lambda\rangle
=\langle\lambda'| A^+_{i,m}A^-_{i,n}|\lambda\rangle.
\end{align*}
Let $v_\lambda$ be the dimension of the vector space $\calV_\lambda$. 
If $\lambda=\lambda'$, then we have
\begin{align*}
(1-q_i^{-2})\langle\lambda| A^-_{i,n}A^+_{i,m}|\lambda\rangle
&=(-u)^{w_i+v_{\circ i}}q_i^{(\alpha_i^\vee,w-\bfc v_\lambda)}\det\big(\calW_i+\calV_{\circ i}\big)^{-1}
\sum_\mu\Res_{u=\calV_\mu/\calV_\lambda}\Big(u^{m+n-1}\\
&\quad\Lambda_{-u^{-1}}\big(-(q_i-q_i^{-1})\calH_{i,1,\lambda}\big)\Big),\\
(1-q_i^{-2})\langle\lambda| A^+_{i,m}A^-_{i,n}|\lambda\rangle
&=-(-u)^{w_i+v_{\circ i}}q_i^{(\alpha_i^\vee,w-\bfc v_\lambda)}\det\big(\calW_i+\calV_{\circ i}\big)^{-1}
\sum_\mu\Res_{u=\calV_\lambda/\calV_\mu}\Big(u^{m+n-1}\\
&\quad\Lambda_{-u^{-1}}\big(-(q_i-q_i^{-1})\calH_{i,1,\lambda}\big)\Big).
\end{align*}
The sums are over all $\mu$'s such that $\ux_\lambda\subset\ux_\mu$ and 
$\ux_\mu\subset\ux_\lambda$
are of codimension $\delta_i$ respectively.
We deduce that 
$$(q_i-q_i^{-1})\,\langle \lambda'|[A^+_i(u),A^-_i(v)]|\lambda\rangle=
\delta_{\lambda,\lambda'}\delta(u/v)(\phi^+_{i,\lambda}(u)-\phi^-_{i,\lambda}(u)).$$
where
$$\phi_{i,\lambda}(u)=
(-1)^{w_i+v_{\circ i}}q_i^{1+(\alpha_i^\vee,w-\bfc v)}
\det\big(\calL_i^\vee\otimes\calW_i+\calL_i^\vee\otimes\calV_{\circ i}\big)^{-1}
\Lambda_{-u^{-1}}\big(-(q_i-q_i^{-1})\calH_{i,1,\lambda}\big)$$
Let $\phi^\pm_i(u)$ be the formal series of operators acting on
$K^A(\widetilde\frakM(v,W))$ 
by multiplication by the Fourier coefficients of the
expansions in non-negative powers of $u^{\mp 1}$ of the following rational function
$$\phi_i(u)=
(-1)^{w_i+v_{\circ i}}q_i^{1+(\alpha_i^\vee,w-\bfc v)}
\det\big(\calL_i^\vee\otimes\calW_i+\calL_i^\vee\otimes\calV_{\circ i}\big)^{-1}
\Lambda_{-u^{-1}}\big(-(q_i-q_i^{-1})\calH_{i,1}\big)$$
We have
\begin{align}\label{Apm}
(q_i-q_i^{-1})\,[A^+_i(u),A^-_i(v)]=
\delta(u/v)(\phi^+_i(u)-\phi^-_i(u)).
\end{align}
We define similarly
$$\psi_i(u)=
q_i^{(\alpha_i^\vee,w-\bfc v)}
\Lambda_{-u^{-1}}\big(-(q_i-q_i^{-1})\calH_{i,1}\big)$$
From \eqref{xpm} and \eqref{Apm} we deduce that
$$(q_i-q_i^{-1})[x^+_i(u)\,,\,x^-_i(v)]=\delta(u/v)\,(\psi^+_i(u)-\psi^-_i(u))$$
where $\psi^\pm_i(u)$ is the expansion of  $\psi_i(u)$
in non-negative powers of $u^{\mp 1}$. 
Recall that
$$q_i^{\rk(\calE)}\Lambda_{-u^{-1}}\big(-(q_i-q_i^{-1})\calE\big)=
q_i^{-\rk(\calE)}\Lambda_{-u}\big((q_i-q_i^{-1})\calE^\vee\big).$$
Hence, we also have
\begin{align}\label{ppsi+-}
\begin{split}
\psi_i^\pm(u)&=q_i^{\pm(\alpha_i^\vee,w-\bfc v)}
\,\Lambda_{-u^{\mp 1}}\big(\mp(q_i-q_i^{-1})\calH_{i,\pm 1}\big)^\pm.
\end{split}
\end{align}
Further, we have the following relation between wedges and Adams operations 
\begin{align}\label{WA}
\Lambda_{-u}(\calE)=\exp\Big(-\sum_{m>0}\psi^m(\calE)u^m/m\Big).
\end{align}
From \eqref{ppsi+-} and \eqref{WA} we deduce that the series 
$\psi_i^\pm(u)$ coincide with the series in \eqref{psi+-}.

\subsubsection{Proof of the relation $\mathrm{(A.6)}$ for $i\neq j$}\label{sec:423}
We will check that the 
elements $x_{i,m}^+$ and $x_{j,n}^-$ in the algebra $K^A(\widetilde\calZ(W))$
commute with each other.
This follows as in \cite[\S 10.2]{N00}
from the transversality result in Lemma \ref{lem:transverse3}.
Set $v_2=v_1+\delta_i=v_3+\delta_j$ and $v_4=v_1-\delta_j=v_3-\delta_i$.
Set
\begin{align*}
I_{v_1,v_2,v_3}
&=\big(\widetilde\frakP(v_1,v_2,W)\times\widetilde\frakM(v_3,W)\big)\cap
\big(\widetilde\frakM(v_1,W)\times\widetilde\frakP(v_2,v_3,W)\big),\\
I_{v_1,v_4,v_3}
&=\big(\widetilde\frakP(v_1,v_4,W)\times\widetilde\frakM(v_3,W)\big)\cap
\big(\widetilde\frakM(v_1,W)\times\widetilde\frakP(v_4,v_3,W)\big).
\end{align*}

\begin{lemma}\label{lem:transverse3}
\hfill
\begin{enumerate}[label=$\mathrm{(\alph*)}$,leftmargin=8mm]
\item 
The intersections $I_{v_1,v_2,v_3}$ and $I_{v_1,v_4,v_3}$ are both transversal
in $\widetilde\frakM(W)^3$.
\item
There is a $G_W\times\bbC^\times$-equivariant isomorphism 
$I_{v_1,v_2,v_3}\simeq I_{v_1,v_4,v_3}$ which intertwines the sheaves
$(\calL_i\boxtimes\calO)|_{I_{v_1,v_2,v_3}}$ and $(\calO\boxtimes\calL_i)|_{I_{v_1,v_4,v_3}}$,
and the sheaves
$(\calO\boxtimes\calL_j)|_{I_{v_1,v_2,v_3}}$ and $(\calL_j\boxtimes\calO)|_{I_{v_1,v_4,v_3}}$.
\end{enumerate}
\end{lemma}

\begin{proof} 
We first prove that the intersection $I_{v_1,v_2,v_3}$
is transversal at any point $(\ux_1,\ux_2,\ux_3)$. 
Let $\pi_i$ be the projection of
$\widetilde\frakM(v_1,W)\times\widetilde\frakM(v_2,W)\times\widetilde\frakM(v_3,W)$
to the $i$th factor along the other ones.
We abbreviate $\ux_{12}=(\ux_1,\ux_2)$ and $\ux_{23}=(\ux_2,\ux_3)$.
The Hecke correspondences
$\widetilde\frakP(v_1,v_2,W)$ and $\widetilde\frakP(v_2,v_3,W)$ are smooth.
Set
$$W_1=(d_{\ux_{12}}\pi_2)\big(\widetilde\frakP(v_1,v_2,W)\big)
 ,\quad
W_3=(d_{\ux_{23}}\pi_2)\big(\widetilde\frakP(v_2,v_3,W)\big).
$$
We claim that
$W_1+W_3=T_{\ux_2}\widetilde\frakM(v_2,W).$
The tangent space of $\widetilde\frakM(v_2,W)$ is
$$T_{\ux_2}\widetilde\frakM(v_2,W)=\widetilde\X(v_2,W)\,/\,\frakg_{v_2}\cdot \ux_2$$
and the tangent spaces of the Hecke correspondences are
\begin{align*}
T_{\ux_{12}}\widetilde\frakP(v_1,v_2,W)=\widetilde\X(v_1,v_2,W)\,/\,\frakp_{v_1,v_2}\cdot \ux_{12}
 ,\\
T_{\ux_{23}}\widetilde\frakP(v_2,v_3,W)=\widetilde\X(v_2,v_3,W)\,/\,\frakp_{v_2,v_3}\cdot \ux_{23}
\end{align*}
where 
\begin{align*}
&\widetilde\X(v_1,v_2,W)=\{y\in\widetilde\X(v_2,W)\,;\,y(V_1\oplus W)
\subseteq V_1\oplus W\},\\
&\widetilde\X(v_2,v_3,W)=\{y\in\widetilde\X(v_2,W)\,;\,y(V_3\oplus W)
\subseteq V_3\oplus W\}.
\end{align*}
It is enough to prove that
$$\pi_2(\widetilde\X(v_1,v_2,W))+\pi_2(\widetilde\X(v_2,v_3,W))=\widetilde\X(v_2,W).$$
To prove this recall that $i\neq j$. 
Hence we have 
$$V_2=\bbC^{\delta_j}\oplus(V_1\cap V_3)\oplus\bbC^{\delta_i}.$$
Let $p_1:V_2\to V_2$ be the projection along $\bbC^{\delta_i}$ 
onto $V_1=\bbC^{\delta_j}\oplus(V_1\cap V_3)$.
Fix any tuple $x_2=(\alpha_2,a_2,a_2^*,\varepsilon_2)\in\widetilde\X(v_2,W)$.
We define $x_1=(\alpha_1,a_1,a_1^*,\varepsilon_1)$ and $x_3=x_2-x_1$ with 
$$\alpha_1=p_1\alpha_2+(1-p_1)\alpha_2(1-p_1)
,\quad
a_1=a_2
,\quad
a_1^*=p_1a_2^*
,\quad
\varepsilon_1=p_1\varepsilon_2+(1-p_1)\varepsilon_2(1-p_1).$$
We have
$x_1\in\pi_2(\widetilde\X(v_1,v_2,W))$
and
$x_3\in\pi_2(\widetilde\X(v_2,v_3,W)).$
The transversality of
$I_{v_1,v_4,v_3}$
is proved in a similar way.

Next we prove (b). Let $\Gr(\delta_i,V)$ be the set of all 
$I$-graded subspaces of codimension $\delta_i$. We have
\begin{align*}
I_{v_1,v_2,v_3}=&\{(S_1,S_3,x)\,;\,x(S_1)\subset S_1\,,\,x(S_3)\subset S_3\}\,/\,G_{v_2}\\
I_{v_1,v_4,v_3}=&\{(T_1,T_3,x_1,x_3,\phi)\,;\,
x_1(T_1)\subset T_1\,,\,x_3(T_3)\subset T_3\,,\,\phi\circ x_1|_{T_1}=x_3|_{T_3}\circ\phi\}\,/\,
G_{v_1}\times G_{v_3}.
\end{align*}
where $(S_1,S_3,x)\in\Gr(\delta_i,V)\times\Gr(\delta_j,V)\times
\widetilde\X(V_2,W)_s$ and
\begin{align*}
(T_1,T_3,x_1,x_3,\phi)&\in\Gr(\delta_j,V_1)\times\Gr(\delta_i,V_3)\times
\widetilde\X(v_1,W)_s\times\widetilde\X(v_3,W)_s\times\Isom(T_1,T_3).
\end{align*}
The isomorphism $I_{v_1,v_2,v_3}\simeq I_{v_1,v_4,v_3}$ is given by 
\begin{align*}
(S_1\,,\,S_3\,,\,x)&\mapsto(S_1\cap S_3\,,\,S_1\cap S_3\,,\,x|_{S_1}\,,\,x|_{S_3}\,,\,\id_{S_1\cap S_3})\\
(T_1\,,\,T_3\,,\,x_1\,,\,x_3\,,\,\phi)&\mapsto(V'_1\,,\,V'_3\,,\,x')
\end{align*}
where $V'_2=V_1\oplus V_3/(\id\times\phi)(T_1)$, the subspaces
$V'_1, V'_3\subset V'_2$ are the images of $V_1$, $V_3$ in
$V'_2$, and
$x'$ is the image of $x_1\oplus x_3$ in $\widetilde\X(V'_2,W)$.
Note that $x|_{S_1}$, $x|_{S_3}$ and $x'$ are stable.
\end{proof}

\subsubsection{Proof of the relation $\mathrm{(A.5)}$ for $i=j$}\label{sec:424}
We will use the reduction to $Q=A_1$ and the algebra homomorphism
\eqref{form19} as in \S\ref{sec:A6=}.
For any $(w_i+v_{\circ i})$-tuples 
$\lambda_-$, $\lambda$ and $\lambda_+$ of weight $v_i$, $v_i+1$ and $v_i+2$ 
we consider the following formal series
\begin{align*}
A^\pm_{u,v}&=\sum_{m,n\in\bbZ}u^{-m}v^{-n}\langle \lambda_\pm |A^\pm_{i,m}|\lambda\rangle\,\langle \lambda |A^\pm_{i,n}|\lambda_\mp\rangle
,\\
A_{v,u}^\pm&=
\sum_{m,n\in\bbZ}u^{-m}v^{-n}
\langle\lambda_\pm |A^\pm_{i,n}|\lambda\rangle\,
\langle \lambda |A^\pm_{i,m}|\lambda_\mp\rangle.
\end{align*}
The formulas above for the matrix coefficients
$\langle \lambda |A^-_{i,n}|\mu\rangle$ and
$\langle \mu|A^+_{i,m}|\lambda\rangle$ yield
\begin{align*}
A_{u,v}^-
&=\frac{u-vq_i^{-2}}{u-v}\sum_{m,n\in\bbZ}
(\calL_{i,\lambda_-,\lambda}/u)^m(\calL_{i,\lambda,\lambda_+}/v)^n
\Lambda_{-v}\Big((q_i^{-2}-1)\calV_{i,\lambda_-}^\vee+
\sum_{c_{ij}<0}q_i^{-c_{ij}}\calV^\vee_j\Big)\\
&\quad\Lambda_{-u}\Big((q_i^{-2}-1)\calV_{i,\lambda_-}^\vee+
\sum_{c_{ij}<0}q_i^{-c_{ij}}\calV^\vee_j\Big),\\
A_{u,v}^+
&=\frac{q_i^{-2}u-v}{u-v}\sum_{m,n\in\bbZ}
(\calL_{i,\lambda,\lambda_+}/u)^{m}(\calL_{i,\lambda_-,\lambda}/v)^{n}
\Lambda_{-u^{-1}}\Big((1-q_i^{-2})\calV_{i,\lambda_+}-q_i\calW_i-\sum_{c_{ij}<0}q_i^{-c_{ij}}\calV_{j}\Big)\\
&\quad
\Lambda_{-v^{-1}}\Big((1-q_i^{-2})\calV_{i,\lambda_+}-q_i\calW_i-\sum_{c_{ij}<0}q_i^{-c_{ij}}\calV_{j}\Big)
\end{align*}
Thus, we have
$A_{u,v}^\pm=A_{v,u}^\pm\,g_{ii}(u/v)^{\pm 1}$, from which we deduce that
\begin{align}\label{drii}
A_i^\pm(u)A_i^\pm(v)=A_i^\pm(v)A_i^\pm(u)\,g_{ii}(u/v)^{\pm 1}.
\end{align}
The relation (A.5) for $i=j$ follows from \eqref{xpm} and \eqref{drii}.

\subsubsection{Proof of the relation $\mathrm{(A.5)}$ for $i\neq j$}\label{sec:425}
Set $v_2=v_1+\delta_i=v_3-\delta_j$, $v_4=v_1+\delta_j=v_3-\delta_i$.
Set also
\begin{align*}
&J_{v_1,v_2,v_3}=\big(\widetilde\frakP(v_1,v_2,W)\times\widetilde\frakM(v_3,W)\big)\cap
\big(\widetilde\frakM(v_1,W)\times\widetilde\frakP(v_2,v_3,W)\big),\\
&J_{v_1,v_4,v_3}=\big(\widetilde\frakP(v_1,v_4,W)\times\widetilde\frakM(v_3,W)\big)\cap
\big(\widetilde\frakM(v_1,W)\times\widetilde\frakP(v_4,v_3,W)\big).
\end{align*}
Since $i\neq j$, the restriction of the projection $\pi_{13}$ to $J_{v_1,v_2,v_3}$ and $J_{v_1,v_4,v_3}$ 
is an isomorphism onto its image. This yields the following isomorphisms
where $\alpha_{ij}$ and $\alpha_{ji}$ are given by the representation $\ux_3$ 
\begin{align*}
J_{v_1,v_2,v_3}
&=\{(\ux_1,\ux_3)\in\widetilde\frakM(v_1,W)\times\widetilde\frakM(v_3,W)\,;\,
\ux_1\subset \ux_3\,,\,\alpha_{ji}(V_{3,i})\subset V_{1,j}\},\\
J_{v_1,v_4,v_3}
&=\{(\ux_1,\ux_3)\in\widetilde\frakM(v_1,W)\times\widetilde\frakM(v_3,W)\,;\,
\ux_1\subset \ux_3\,,\,\alpha_{ij}(V_{3,j})\subset V_{1,i}\}.
\end{align*}
We view $\alpha_{ji}$ as a $G_W\times\bbC^\times$-equivariant section of the 
$G_W\times\bbC^\times$-equivariant vector bundle 
$$\calH om(\calV_{3,i}/\calV_{1,i},\calV_{3,j}/\calV_{1,j})$$ over
$J_{v_1,v_4,v_3}$. We view $\alpha_{ij}$ as a $G_W\times\bbC^\times$-equivariant section 
of the bundle $$\calH om(\calV_{3,j}/\calV_{1,j},\calV_{3,i}/\calV_{1,i})$$ over $J_{v_1,v_2,v_3}$. 
Let $s_{ji}$ and $s_{ij}$ denote these sections.
Let $Z(s_{ji})$ and $Z(s_{ij})$ be their 0 sets.

\begin{lemma}\label{lem:transverse4}
\hfill
\begin{enumerate}[label=$\mathrm{(\alph*)}$,leftmargin=8mm]
\item
The intersections $J_{v_1,v_2,v_3}$ and $J_{v_1,v_4,v_3}$ are both transversal
in $\widetilde\frakM(W)^3$.
\item
The sections $s_{ij}$ and $s_{ji}$ are both transverse to the zero section.
\item
The subsets $Z(s_{ji}), Z(s_{ij})\subset\widetilde\frakM(v_1,W)\times\widetilde\frakM(v_3,W)$  coincide.
\end{enumerate}
\end{lemma}

\begin{proof}
We first prove that the intersection $J_{v_1,v_2,v_3}$
is transversal at any point $(\ux_1,\ux_2,\ux_3)$. 
The transversality of $J_{v_1,v_4,v_3}$ can be proved in a similar way.
Let $\pi_i$ be the projection of
$\widetilde\frakM(v_1,W)\times\widetilde\frakM(v_2,W)\times\widetilde\frakM(v_3,W)$
to the $i$th factor along the other ones.
Set $\ux_{12}=(\ux_1,\ux_2)$ and $\ux_{23}=(\ux_2,\ux_3)$.
The Hecke correspondences
$\widetilde\frakP(v_1,v_2,W)$ and $\widetilde\frakP(v_2,v_3,W)$ are smooth.
Set
$$W_1=(d_{\ux_{12}}\pi_2)\big(\widetilde\frakP(v_1,v_2,W)\big)
 ,\quad
W_3=(d_{\ux_{23}}\pi_2)\big(\widetilde\frakP(v_2,v_3,W)\big).
$$
We claim that
$W_1+W_3=T_{\ux_2}\widetilde\frakM(v_2,W).$
The tangent space of $\widetilde\frakM(v_2,W)$ is
$$T_{\ux_2}\widetilde\frakM(v_2,W)=\widetilde\X(V_2,W)\,/\,\frakg_{V_2}\cdot \ux_2.$$
The tangent spaces of the Hecke correspondences are
\begin{align*}
T_{\ux_{12}}\widetilde\frakP(v_1,v_2,W)=\widetilde\X(V_1,V_2,W)\,/\,\frakp_{V_1,V_2}\cdot \ux_{12}
 ,\\
T_{\ux_{23}}\widetilde\frakP(v_2,v_3,W)=\widetilde\X(V_2,V_3,W)\,/\,\frakp_{V_2,V_3}\cdot \ux_{23}
\end{align*}
Here $\widetilde\X(V_1,V_2,W)$ is the vector subspace of $\widetilde\X(V_2,W)$ given by
\begin{align*}
\widetilde\X(V_1,V_2,W)=\{y\in\widetilde\X(V_2,W)\,;\,y(V_1\oplus W)
\subseteq V_1\oplus W\},
\end{align*}
The vector subspace $\widetilde\X(V_2,V_3,W)\subset\widetilde\X(V_3,W)$ is defined similarly.
It is enough to prove that
$$\pi_2(\widetilde\X(V_1,V_2,W))+\pi_2(\widetilde\X(V_2,V_3,W))=\widetilde\X(V_2,W).$$
This is obvious because
$\pi_2(\widetilde\X(V_2,V_3,W))=\widetilde\X(V_2,W).$

Now, we concentrate on (b).
We must check that the section $s_{ij}$ of 
$\calH om(\calV_{3,j}/\calV_{1,j},\calV_{3,i}/\calV_{1,i})$ over
$J_{v_1,v_2,v_3}$ is transverse to the zero section. 
It is enough to prove that the map
\begin{align*}
\{y\in\widetilde\X(V_1,V_3,W)\,;\,y(V_1\oplus W)\subseteq V_1\oplus W\,,\,
y(V_{3,i})\subset V_{1,j}\}\to\Hom(V_{3,j}/V_{1,j},V_{3,i}/V_{1,i})
,\quad
y\mapsto\alpha_{ij}
\end{align*}
is surjective. This is obvious.
The proof for $s_{ji}$ is done similarly exchanging $i$ and $j$.

Finally, Claim (c) is obvious.
\end{proof}

Recall the rational function $g_{ij}(u)$ in \eqref{gij}.
We define the function $h_{ij}$ such that
$$h_{ij}(u)=-u\,g_{ij}(u)\ \text{if}\ c_{ij}<0
,\quad
h_{ij}(u)=1\ \text{if}\ c_{ij}=0.$$
Using Proposition \ref{prop:TTK} (b), Lemma \ref{lem:transverse4} and the projection formula, we deduce that the formula
\begin{align*}
\calO_{Z(s_{ij})}\otimes(A_i^+(u)\star A_j^+(v))=\calO_{Z(s_{ji})}\otimes(A_j^+(v)\star A_i^+(u))
\end{align*}
follows  as in \cite[\S10.3]{N00}.
Here we have
\begin{align*}
A_i^+(u)\star A_j^+(v)\,,\,A_j^+(v)\star A_i^+(u)&\in
K_A(\widetilde\frakM(W)^2,(\tilde f_\gamma)^{(2)})_{\underline{\widetilde\calZ(W)}}[[u,u^{-1},v,v^{-1}]],\\
\calO_{Z(s_{ij})}\,,\,\calO_{Z(s_{ji})}&\in K_A(\widetilde\frakM(W)^2).
\end{align*}
and the tensor product is as in \eqref{otimes}.
Hence, we have
\begin{align}\label{dr+}
A_i^+(u)A_j^+(v)=A_j^+(v)A_i^+(u)\,h_{ij}(u/v)
\end{align}
In a similar way we prove that
\begin{align}\label{dr-}A_i^-(u)A_j^-(v)=A_j^-(v)A_i^-(u)\,h_{ij}(u/v)^{-1}.
\end{align}
The relation (A.5) for $i\neq j$ follows from \eqref{xpm}, \eqref{dr+} and  \eqref{dr-}.

\subsubsection{Proof of the relation $\mathrm{(A.7)}$}\label{sec:A7}
We have
$$K_A\big(\widetilde\frakM(v,W),\tilde f_\gamma\big)=0$$
for all but finitely many $v$'s by Remark \ref{rem:fd}.
Hence, the operator $x^\pm_{i,0}$ acting on 
$$K_A(\widetilde\frakM(W)^2,(\tilde f_\gamma)^{(2)})_{\underline{\widetilde\calZ(W)}}$$
is locally nilpotent and the operators $\psi_{i,0}^+$, $\psi_{i,0}^-$ 
are diagonalizable and inverse one to each other.
Thus, the constant term of the relation (A.7), i.e.,  the relation
$$\sum_{r=0}^{s}(-1)^r(x^\pm_{i,0})^{[r]}\,x^\pm_{j,0}\,(x^\pm_{i,0})^{[s-r]}=0
,\quad
s=1-c_{ij}
,\quad
i\neq j$$
can be deduced from the relation (A.6). The relation (A.7)
can be derived from this constant term as in \cite[\S10.4]{N00}.

\subsubsection{End of the proof of Theorem $\ref{thm:notshifted}$}
We have proved that the $F$-algebra homomorphism
\begin{align*}
\U_F(L\frakg)\to 
K_A\big(\widetilde \frakM(W)^2,(\tilde f_\gamma)^{(2)}\big)_{\underline{\tilde\calZ(W)}}\otimes_RF
\end{align*}
in \eqref{map8} is well-defined. The compatibility with the 
$R$-lattices follows from the formulas \eqref{psi+-} and \eqref{xpm}.
More precisely, the $R$-subalgebra $\U_R(L\frakg)$ of $\U_F(L\frakg)$ is generated by 
the elements
$$\psi_{i,0}^\pm
,\quad
h_{i,\pm m}/[m]_{q_i}
,\quad
(x^\pm_{i,n})^{[m]},
\quad
i\in I, n\in\bbZ, m\in\bbN^\times$$
in \eqref{LGR}.
By \eqref{H} and \eqref{psi+-} 
the map \eqref{map8} takes $h_{i,\pm m}/[m]_{q_i}$ into the subset
$$K_A\big(\widetilde\frakM(W)^2,(\tilde f_\gamma)^{(2)}\big)
_{\underline{\widetilde\calZ(W)}}\,/\,\tor\,\subset\,
K_A\big(\widetilde\frakM(W)^2,(\tilde f_\gamma)^{(2)}\big)
_{\underline{\widetilde\calZ(W)}}\otimes_RF.$$
Using \eqref{xpm}, a computation similar to the proof of
\cite[thm.~12.2.1]{N00} or \cite[lem.~2.4.8]{VV22}, 
shows that \eqref{map8} maps $(x^\pm_{i,n})^{[m]}$ to the same lattice.

\subsubsection{Comparison with the Nakajima construction}\label{sec:rem}
Recall that $R=R_{\bbC^\times}$ and $F=F_{\bbC^\times}$. Let
$$K_{G_W\times\bbC^\times}\big(\widetilde\frakM(W)^2,(\tilde f_\gamma)^{(2)}\big)_{\widetilde\calZ(W)}/\tor
\subset K_{G_W\times\bbC^\times}\big(\widetilde\frakM(W)^2,(\tilde f_\gamma)^{(2)}\big)_{\widetilde\calZ(W)}
\otimes_RF$$ be the image of
$K_{G_W\times\bbC^\times}\big(\widetilde\frakM(W)^2,(\tilde f_\gamma)^{(2)}\big)_{\widetilde\calZ(W)}$.
The proof also yields a commutative triangle of $R$-algebras
$$
\xymatrix{
\U_R(L\frakg)\ar[r]\ar[dr]& 
K_{G_W\times\bbC^\times}\big(\widetilde\frakM(W)^2,(\tilde f_\gamma)^{(2)}\big)_{\widetilde\calZ(W)}/\tor
\ar[d]\\
&K_A\big(\widetilde\frakM(W)^2,(\tilde f_\gamma)^{(2)}\big)_{\widetilde\calZ(W)}/\tor
}$$
where the vertical map is given by the forgetful homomorphism.

Now, let us assume that the Cartan matrix is symmetric. 
Let $\frakM(W)$ be the Nakajima quiver variety associated with the quiver $Q$ 
and the $I$-graded vector space $W$.
The representation variety $\overline\X(V,W)$ is holomorphic symplectic with an Hamiltonian action of
the groups $G_V$ and $G_W$.
Let
$\mu:\overline\X(V,W)\to\frakg^\vee_V$ 
be the moment map for the $G_V$-action.
Set
\begin{align*}
\overline\X(V,W)_s=\{x\in\overline\X(V,W)\,;\,x\ \text{is\ stable}\}
,\quad
\mu^{-1}(0)_s=\overline\X(V,W)_s\cap\mu^{-1}(0).
\end{align*}
The Nakajima quiver variety is the good quotient 
$\frakM(v,W)=\mu^{-1}(0)_s/G_V.$
Let $\calZ(W)$ be the corresponding Steinberg variety in $\frakM(W)\times\frakM(W)$.
The equivariant complexified Grothendieck group $K^{G_W\times\bbC^\times}\big(\calZ(W)\big)$
has an $R_{G_W\times\bbC^\times}$-algebra structure and
$K^{G_W\times\bbC^\times}\big(\frakM(W)\big)$ is a module over this algebra. See \cite{N00} for details.

Now, we choose the nilpotent matrix $\gamma$ to be 0.
In this case we have $\tilde f_\gamma=\tilde f_1$ by \eqref{f6}.
Then the following holds.

\begin{proposition}\label{prop:crit1}
\hfill
\begin{enumerate}[label=$\mathrm{(\alph*)}$,leftmargin=8mm]
\item The extension by zero yields an isomorphism
$\frakM(W)=\crit(\tilde f_1).$

\item
We have an algebra isomorphism
\begin{align*}
K_{G_W\times\bbC^\times}\big(\widetilde\frakM(W)^2,(\tilde f_1)^{(2)}\big)_{\widetilde\calZ(W)}=
K^{G_W\times\bbC^\times}\big(\calZ(W)\big).
\end{align*}
\item
We have the module isomorphism
\begin{align*}
K_{G_W\times\bbC^\times}\big(\widetilde\frakM(W),\tilde f_1\big)
=K^{G_W\times\bbC^\times}\big(\frakM(W)\big).
\end{align*}

\end{enumerate}
\end{proposition}

\begin{proof}
We have 
$$\{(x,\varepsilon)\in\widetilde\X(V,W)_s\,;\,
[\varepsilon,x]=\mu(x)=0\}\,/\,G_V=\crit(\tilde f_1)\cap\widetilde\frakM(v,W).$$
For any tuple $(x,\varepsilon)$ as above,
the subspace $\Im(\varepsilon)$ of $V$ is preserved by the
action of the path algebra $\bbC\widetilde Q$ of $\widetilde Q$, 
and it is contained in the kernel of $a$.
Hence, we have $\varepsilon=0$ and $x\in\mu_V^{-1}(0)_s$.
Thus, the assignment $x\mapsto (x,0)$ yields an isomorphism
$\frakM(W)=\crit(\tilde f_1)$, proving Part (a). 
Now, we prove Part (b).
Forgetting the arrows $\varepsilon_i$ for each $i\in I$ yields a map
$\widetilde\X(v,W)\to \overline\X(V,W)$.
Let $\widetilde\X(v,W)_\circ$ be the inverse image of the open subset $\overline\X(V,W)_s$.
We consider the open subset of $\widetilde\frakM(W)$ given by
\begin{align*}
\widetilde\frakM(W)_\circ=\bigsqcup_{v\in\bbN I}\widetilde\frakM(v,W)_\circ
,\quad
\widetilde\frakM(v,W)_\circ=\widetilde\X(v,W)_\circ/G_v.
\end{align*}
By (a) we have
$\crit(\tilde f_1)\subset\widetilde\frakM(W)_\circ$.
Hence \eqref{excision} and \eqref{crit} yield
$$K_{G_W\times\bbC^\times}\big(\widetilde\frakM(W),\tilde f_1\big)=
K_{G_W\times\bbC^\times}\big(\widetilde\frakM(W)_\circ,\tilde f_1\big).$$
Next, we use the dimensional reduction in K-theory. 
More precisely, we apply \cite{I12} or \cite[thm.~1.2]{H17b}
to the vector bundle
$\widetilde\frakM(W)_\circ\to\big\{\ux\in\widetilde\frakM(W)_\circ\,;\,
\varepsilon=0\big\}$
given by forgetting the variables $\varepsilon_i$ for each $i\in I$. 
We have
\begin{align*}
\frakM(W)&=\big\{\ux\in\widetilde\frakM(W)_\circ\,;\,
\varepsilon=0\,,\,\partial\tilde f_1/\partial\varepsilon(\ux)=0\big\}
\end{align*}
We deduce that
$$K_{G_W\times\bbC^\times}\big(\widetilde\frakM(W)_\circ,\tilde f_1\big)
=K^{G_W\times\bbC^\times}\big(\frakM(W)\big).$$
In a similar way we prove that
\begin{align*}
K_{G_W\times\bbC^\times}\big(\widetilde\frakM(W)^2,(\tilde f_1)^{(2)}\big)_{\widetilde\calZ(W)}
&=
K_{G_W\times\bbC^\times}\big(\widetilde\frakM(W)_\circ^2,(\tilde f_1)^{(2)}
\big)_{\widetilde\calZ(W)\cap\widetilde\frakM(W)_\circ^2}\\
&=
K^{G_W\times\bbC^\times}\big(\calZ(W)\big).
\end{align*}
\end{proof}

\subsection{K-theoretic critical convolution algebras and shifted quantum loop groups}\label{sec:w2}
Let $\sigma:\bbC^\times\to G_W$ be a cocharacter compatible with an $I^\bullet$-grading on $W$
as in \eqref{sigma1}.
Let $A\subset G_W\times\bbC^\times$ be as in \S\ref{sec:GQVD}.
Since the set of stable points 
does not depend on the variable $a^*$. Forgetting $a^*$ yields a vector bundle
\begin{align}\label{VB}
\nu:\widetilde\frakM(W)\to\widehat\frakM(W).
\end{align}
Since the potential $\bfw_2$ does not depend on $a^*$ either,
we have $\tilde f_2=\hat f_2\circ\nu$.
The Thom isomorphism in Proposition \ref{prop:TTK} (b) yields an isomorphism
$$K_A(\widehat\frakM(W),\hat f_2)=
K_A(\widetilde\frakM(W),\tilde f_2).$$

\begin{theorem}\label{thm:shifted}
\hfill
\begin{enumerate}[label=$\mathrm{(\alph*)}$,leftmargin=8mm]
\item
There is an $R$-algebra homomorphism
$\U_R^{-w}(L\frakg)\to K_A\big(\widehat\frakM(W)^2,(\hat f_2)^{(2)}\big)_{\widehat\calZ(W)}\,/\,\tor$
which takes the central element $\psi^+_{i,0}\,\psi^-_{i,-w_i}$ to $(-1)^{w_i}q_i^{-w_i}\det(W_i)^{-1}$ for each 
vertex $i\in I$. 
\item
The $R$-algebra $\U_R^{-w}(L\frakg)$ acts on 
$K_A(\widehat\frakM(W),\hat f_2)_{\widehat\frakL(W)}\,/\,\tor$ and
$K_A(\widehat\frakM(W),\hat f_2)\,/\,\tor$.
\end{enumerate}
\end{theorem}

Using Propositions \ref{prop:critalg1}, \ref{prop:TTK}, the theorem implies the following.

\begin{corollary}\label{cor:shifted2}
The algebra $\U_\zeta^{-w}(L\frakg)$ acts on
$K(\widehat\frakM^\bullet(W),\hat f^\bullet_2)_{\underline{\widehat\frakL^\bullet(W)}}$
and
$K(\widehat\frakM^\bullet(W),\hat f^\bullet_2).$
\qed
\end{corollary}

\medskip

\subsection{Proof of Theorem $\ref{thm:shifted}$}
\label{shifted}

\subsubsection{Definition of the homomorphism}

We must define an $F$-algebra homomorphism
\begin{align}\label{map8g}
\U_F^{-w}(L\frakg)\to 
K_A\big(\widehat\frakM(W)^2,(\hat f_2)^{(2)}\big)_{\widehat\calZ(W)}\otimes_RF
\end{align}
The $F$-algebra $\U_F^{-w}(L\frakg)$ is generated by the Fourier coefficients of 
\begin{align*}
x^\pm_i(u)=\sum_{n\in\bbZ}x^\pm_{i,n}\,u^{-n}
 ,\quad
\psi^+_i(u)=\sum_{n\in\bbN}\psi^+_{i,n}\,u^{-n}
 ,\quad
\psi^-_i(u)=\sum_{n\geqslant w_i}\psi^-_{i,-n}\,u^{n}
\end{align*}
modulo the defining relations (A.2) to (A.7) in \S\ref{sec:QG}.
We define the following classes in $K_A\big(\widehat\frakM(W)\big)$
\begin{align}\label{Hg}
\calH_{i,1}=\calW_i-\sum_j[c_{ij}]_{q_i}\calV_j
 ,\quad
\calH_{i,-1}=\calW_i^\vee-\sum_j[c_{ij}]_{q_i}\calV_j^\vee
 ,\quad
\calH_{i,\pm m}=\frac{[m]_{q_i}}{m}\,\psi^m(\calH_{i,\pm 1})
\end{align}
Note that these classes are the pullback of the classes \eqref{H} under the isomorphism
$K_A\big(\widehat\frakM(W)\big)\cong K_A\big(\widetilde\frakM(W)\big)$.
By Proposition \ref{prop:critalg1} we have the $R_A$-algebra homomorphism  
\begin{align}\label{form4}
\Upsilon:
K^A(\widehat\calZ(W))\to
K_A(\widehat\frakM(W)^2,(\hat f_2)^{(2)})_{\widehat\calZ(W)}
\end{align}
Composing the pushforward by the diagonal embedding with the map
$\Upsilon$, we get the map
\begin{align}\label{delta}
\Delta:K_A\big(\widehat\frakM(W)\big)\to
K_A\big(\widehat\frakM(W)^2,(\hat f_2)^{(2)}\big)_{\widehat\calZ(W)}
\end{align}
We assign to $\psi_{i,n}^\pm$ the image by $\Delta$
of the coefficient of $u^{-n}$ in
$K_A\big(\widehat\frakM(W)\big)$
in the formal series
\begin{align}\label{psi+-g}
q_i^{-w_i\pm(\alpha_i^\vee,w-\bfc v)}\,\Lambda_{-u^{-1}}(q_i^{-1}\calW_i)^{-1}
\,\exp\Big(\pm (q_i-q_i^{-1})\sum_{m>0}\calH_{i,\pm m}u^{\mp m}\Big).
\end{align}
We define 
$$A_{i,n}^\pm=\Upsilon\big(\calL_i^{\otimes n}\big)
,\quad
A_i^\pm(u)=\sum_{m\in\bbZ}A^\pm_{i,m}u^{-m}.$$
We assign to $x_{i,n}^\pm$ the following element in
$K_A(\widehat\frakM(W)^2,(\hat f_2)^{(2)})_{\widehat\calZ(W)}$
\begin{align}\label{xpmg}
\begin{split}
x_{i,n}^+\mapsto (-1)^{v_{+i}}A^+_{i,n}\star\Delta(\det(\calV_{\circ i}))
,\quad
x_{i,n}^-\mapsto
(-1)^{v_{-i}}q_i^{-1}\,A^-_{i,n-v_{\circ i}}
\end{split}
\end{align}
To prove that the assignments \eqref{psi+-g} and
\eqref{xpmg} give a well-defined morphism \eqref{map8g}, 
we must check that
the images of $x_{i,n}^\pm$ and $\psi_{i,n}^\pm$ in the algebra
$K_A(\widehat\frakM(W)^2,(\hat f_2)^{(2)})_{\widehat\calZ(W)}$
satisfy the relations 
$\mathrm{(A.2)}$ to $\mathrm{(A.7)}$.
The classes $A_{i,n}^\pm$, $x_{i,n}^\pm$ and $\psi_{i,n}^\pm$ have obvious liftings
to $K^A(\widehat\calZ(W))$ relatively to the map $\Upsilon$ in \eqref{form4}.
Let $A_{i,n}^\pm$, $x_{i,n}^\pm$ and $\psi_{i,n}^\pm$ denote also these liftings.

\subsubsection{Case $Q=A_1$}\label{sec:Q=A1}
We assume that $Q=A_1$. 
In this case the computation is done 
via a reduction to the fixed points of a torus action as in \cite{VV99}.
We have $I=\{i\}$, $\bfw_2=0$,
$\deg(a_i)=\deg(a_i^*)=-d_i$ and $\deg(\varepsilon_i)=2d_i$.
Further, we have $\hat f_2=0$, hence
\begin{align*}
K^A(\widehat\calZ(W))=
K_A(\widehat\frakM(W)^2,(\hat f_2)^{(2)})_{\widehat\calZ(W)}
\end{align*}
The variety $\widehat\frakM(W)$ satisfies the property ($T$) 
by Lemma \ref{lem:(T)}. To check the relations between $A_{i,n}^+$ and $A_{i,n}^-$
in the algebra $K^A(\widehat\calZ(W))/\tor$ it is enough to compute their action on
$$K^{G_W\times\bbC^\times}(\widehat\frakM(W))\otimes_{R_{G_W\times\bbC^\times}}F_{G_W\times\bbC^\times}.$$ 
See \S\ref{sec:A6=} for details.
To do so, let $T_W\subset G_W$ be a maximal torus.
By Lemma \ref{lem:fixedpoints2},
the $T_W\times\bbC^\times$-fixed point locus is
\begin{align}\label{fixpoints}
\widehat\frakM(v,W)^{T_W\times\bbC^\times}=
\{\ux_\lambda\,;\,\lambda\in\bbN^{w_i}\,,\,|\lambda|=v_i\}.
\end{align} 
We may abbreviate $\lambda=\ux_\lambda$ if no confusion is possible.
Let $[\lambda]$ be the fundamental class
of $\{\underline x_\lambda\}$.
For any linear operator $A$, 
let $\langle \lambda|A|\mu\rangle$ be the coefficient of the basis element $[\lambda]$ 
in the expansion of $A[\mu]$ in the
basis $\{[\lambda]\,;\,\lambda\in\bbN^{w_i}\}$.
Let $\lambda$ and $\mu$ be $w_i$-tuples of weight $v_i$ and $v_i+1$
such that $\ux_\lambda\subset\ux_\mu$.
We abbreviate
\begin{align}\label{VLg}
\calV_\lambda=\calV|_{\{\underline x_\lambda\}}
,\quad
\calL_{\lambda,\mu}=\calL|_{\{(\underline x_\lambda,\,\underline x_\mu)\}}.
\end{align}
The matrix coefficients $\langle \lambda |A^-_{i,n}|\mu\rangle$ and 
$\langle \mu|A^+_{i,m}|\lambda\rangle$
in $F_{G_W\times\bbC^\times}$ are given by
\begin{align*}
\langle \lambda|A^-_{i,n}|\mu\rangle
&=(\calL_{i,\mu,\lambda})^{\otimes n}\otimes
\Lambda_{-1}\Big(T_{\mu}\widehat\frakM(v+\delta_i,W)-T_{\mu,\lambda}\widehat\frakP(v+\delta_i,v,W)\Big)\\
\langle \mu|A^+_{i,m}|\lambda\rangle
&=(\calL_{i,\lambda,\mu})^{\otimes m}\otimes
\Lambda_{-1}\Big(T_\lambda\widehat\frakM(v,W)-T_{\lambda,\mu}\widehat\frakP(v,v+\delta_i,W)\Big)
\end{align*}
The class of the tangent bundle $T\widehat\frakM(W)$ in the equivariant Grothendieck group of 
$\widehat\frakM(W)$ is
\begin{align}\label{C1}
T\widehat\frakM(W)=(q_i^{-2}-1)\End(\calV_i)+q_i\Hom(\calV_i,\calW_i).
\end{align}
We abbreviate
\begin{align}\label{Pig}
\calP_i=\End(\calV^-_i)+\Hom(\calL_i,\calV^+_i)=\End(\calV^+_i)-\Hom(\calV^-_i,\calL_i).
\end{align}
The class of $T\widehat\frakP(W)$ in the equivariant Grothendieck group of $\widehat\frakP(W)$ is
\begin{align}\label{C2}
T\widehat\frakP(W)=(q_i^{-2}-1)\calP_i+q_i\Hom(\calV^+_i,\calW_i)
\end{align}
We write 
\begin{align}\label{form10}
\calV_i=\sum_{r=1}^{v_i}z_r
 ,\quad
\calW_i=\sum_{s=1}^{w_i}\chi_s
\end{align} 
where $z_1,\dots, z_{v_i}$ and $\chi_1,\dots,\chi_{w_i}$ are the fundamental characters of the tori
$T_v$ and $T_W$.
The labelling of the fixed points in \eqref{fixpoints} is
such that
\begin{align}\label{form7}
\begin{split}
\calL_{i,\lambda,\mu}=z_{v_i+1}
,\quad
\calV_{i,\lambda}=\sum_{r=1}^{v_i}z_r=\sum_{s=1}^{w_i}\sum_{r=0}^{\lambda_s-1}\chi_sq_i^{1-2r}
\in R_{T_W\times\bbC^\times}
\end{split}
\end{align} 
We deduce that
\begin{align}\label{form3}
\begin{split}
\langle \lambda |A^-_{i,n}|\mu\rangle
&=(\calL_{i,\lambda,\mu})^{\otimes n}\otimes
\Lambda_{-1}\Big((q_i^{-2}-1)\calV^\vee_{i,\lambda}\otimes\calL_{i,\lambda,\mu}\Big)\\
&=\ev_{u=z_{v_i+1}}\Big(u^{n}\prod_{r=1}^{v_i}\frac{uq_i^{-2}-z_r}{u-z_r}\Big)\\
\langle \mu|A^+_{i,m}|\lambda\rangle
&=(\calL_{i,\lambda,\mu})^{\otimes m}\otimes\Lambda_{-1}
\Big((1-q_i^{-2})\otimes
\calL^\vee_{i,\lambda,\mu}\otimes\calV_{i,\mu}-q_i\calL^\vee_{i,\lambda,\mu}\otimes\calW_i\Big)\\
&=(1-q_i^{-2})^{-1}\Res_{u=z_{v_i+1}}\Big(\frac{u^{m+w_i-1}}{\prod_{s=1}^{w_i}(u-\chi_sq_i)}
\prod_{r=1}^{v_i}\frac{u-z_r}{u-z_rq_i^{-2}}\Big)
\end{split}
\end{align}
We consider the rational function
$\psi_{i,\lambda}(u)\in F_{G_W\times\bbC^\times}$ 
such that
\begin{align}\label{h}
\begin{split}
\psi_{i,\lambda}(u)=
u^{w_i}\frac{\prod_{s=1}^{w_i}\prod_{r=0}^{\lambda_s-1}g_{ii}(u/\chi_sq_i^{1-2r})}
{\prod_{s=1}^{w_i}(u-\chi_s q_i)}
=
q_i^{-2v_i}u^{w_i}\prod_{s=1}^{w_i}\frac{u-\chi_sq_i^3}
{(u-\chi_sq_i^{1-2\lambda_s})(u-\chi_sq_i^{3-2\lambda_s})}
\end{split}
\end{align}
Since the poles of 
$\psi_{i,\lambda}(u)$ belong to the set 
$\{\chi_sq_i^{1-2\lambda_s}\,,\,\chi_sq_i^{3-2\lambda_s}\,;\,s=1,\dots w_i\}$,
we have
\begin{align*}
(q_i-q_i^{-1})\langle\lambda|[A^+_{i,m},A^-_{i,n}|\lambda\rangle
&=
-q_i\sum_{s=1}^{w_i}\Res_{u=\chi_sq_i^{t_i+2-2\lambda_s}}\big(f(u)\big)+
\Res_{u=\chi_sq_i^{t_i-2\lambda_s}}\big(f(u)\big)\\
&=
q_i\Res_{u=0}\big(f(u)\big)+
q_i\Res_{u=\infty}\big(f(u)\big)
\end{align*}
where $f(u)=u^{m+n-1}\psi_{i,\lambda}(u)$.
Let $\psi^\pm_{i,\lambda}(u)$ be the expansion of  $\psi_{i,\lambda}(u)$
in non-negative powers of $u^{\mp 1}$. 
We deduce that the matrix coefficient
$(q_i-q_i^{-1})\,\langle \lambda|[A^+_{i,m},A^-_{i,n}]|\mu\rangle$ 
is equal to the Kronecker symbol $\delta_{\lambda,\mu}$ 
times the coefficient of $u^{-m-n}$ in the formal series 
$-q_i\psi^+_{i,\lambda}(u)+q_i\psi^-_{i,\lambda}(u)$.
Now, let $\psi^\pm_i(u)$ be the formal series of operators acting on
$K^{G_W\times\bbC^\times}(\widehat\frakM(v,W))$ 
by multiplication by the Fourier coefficients of the
expansions in non-negative powers of $u^{\mp 1}$ of the following rational function
\begin{align*}
\begin{split}
\psi_i(u)=\frac{\prod_{r=1}^{v_i}g_{ii}(u/z_r)}{\prod_{s=1}^{w_i}(1-\chi_s q_i^{t_i}/u)}.
\end{split}
\end{align*}
We have
$(q_i-q_i^{-1})[x_i^+(u)\,,\,x_i^-(v)]=\delta(u/v)\,(\psi_i^+(u)-\psi_i^-(u))$
with
\begin{align}\label{psipmi}
\begin{split}
\psi_i^+(u)&=q_i^{-w_i}q_i^{(\alpha_i^\vee,w-\bfc v)}
\Lambda_{-u^{-1}}\big((q_i^2-q_i^{-2})\calV_i-q_i\calW_i\big)^+,
\\
\psi_i^-(u)&=(-u)^{w_i}q_i^{-(\alpha_i^\vee,w-\bfc v)}\det(\calW_i)^{-1}
\Lambda_{-u}\big((q_i^{-2}-q_i^2)\calV_i^\vee-q_i^{-1}\calW_i^\vee\big)^-
\end{split}
\end{align}
where the upperscript $\pm$ holds for the expansion in non-negative powers of $u^{\mp 1}$.

\subsubsection{Proof of the relation $\mathrm{(A.6)}$ for $i=j$}\label{sec:A6=s}
Let $Q$ be any Dynkin quiver.
Using the homomorphism $\Upsilon$ in \eqref{form4}, it is enough to check
the relation in the algebra
$K^A(\widehat\calZ(W))$.
We will prove the relation via a reduction to the case $A_1$,
which is proved above.
Consider the subquiver 
$\widehat Q_{f,\neq i}$ of $\widehat Q_f$ such that
$$(\widehat Q_{f,\neq i})_0=(\widehat Q_f)_0\,\setminus\,\{i,i'\}
 ,\quad
(\widehat Q_{f,\neq i})_1=\{h\in (\widehat Q_f)_1\,;\,s(h),t(h)\neq i\}.$$
The representation variety of $\widehat Q_f$ decomposes as
$$
\widehat\X(V,W)=\widehat\X(V_i,W_i\oplus V_{\circ i})\times
\Hom(V_{\circ i},V_i)\times\X_{\widehat Q_{f,\neq i}}(V_{\neq i},W_{\neq i})
$$
where $V=\bbC^v$, $V_{\circ i}$ is as in \eqref{Vcirc} and
$$
V_{\neq i}=\bigoplus_{j\neq i}V_j
 ,\quad
W_{\neq i}=\bigoplus_{j\neq i}W_j.$$
We define
$$\frakM(v,W)_\diamondsuit=\X(V,W)_\diamondsuit\,/\,G_{V_i}
,\quad
\frakM(v,W)_\heartsuit=\X(V,W)_\heartsuit\,/\,G_{V_i}
,\quad
\frakM(v,W)_\spadesuit=\widehat\X(V,W)_s\,/\,G_{V_i}$$
where
\begin{align*}
\X(V,W)_\diamondsuit=\widehat\X(V_i,W_i\oplus V_{\circ i})_s\times
\Hom(V_{\circ i},V_i),\quad
\X(V,W)_\heartsuit=\X(V,W)_\diamondsuit\times\X_{\widehat Q_{f,\neq i}}(V_{\neq i},W_{\neq i})
\end{align*}
We consider the diagram
\begin{align}\label{diag}
\xymatrix{\widehat\frakM(v_i,W_i\oplus V_{\circ i})&\ar[l]_-\rho
\frakM(v,W)_\diamondsuit&\ar[l]_p\frakM(v,W)_\heartsuit&\ar[l]_-{\iota}\frakM(v,W)_\spadesuit\ar[r]^\pi
&\widehat\frakM(v,W)}
\end{align}
The maps $\rho$, $p$ are the obvious projections. 
They are vector bundles.
The map $\iota$ is an open embedding.
The map $\pi$ is a principal bundle.
Let $v=(v_i,v_{\neq i})$ with $v_{\neq i}$ fixed and $v_i$ running in $\bbN$.
We abbreviate
$$\widehat\frakM(W_i\oplus V_{\circ i})=\bigsqcup_{v_i\in\bbN}\widehat\frakM(v_i,W_i\oplus V_{\circ i})
 ,\quad
\frakM(W)_\flat=\bigsqcup_{v_i\in\bbN}\frakM(v,W)_\flat
,\quad
\flat=\diamondsuit,\heartsuit,\spadesuit.$$
The Hecke correspondence $\frakP(W)_\diamondsuit$ and the Steinberg variety 
$\calZ(W)_\diamondsuit$ in $\frakM(W)_\diamondsuit^2$
are defined as in \eqref{Hecke} and \eqref{Steinberg}.
The Hecke correspondence $\frakP(W)_\heartsuit$ in $\frakM(W)_\heartsuit^2$
is the product of $\frakP(W)_\diamondsuit$
and the diagonal of  
$\X_{\widehat Q_{f,\neq i}}(V_{\neq i},W_{\neq i})$. 
The map $\iota$ satisfies the condition \cite[(11.2.1)]{N00},
and $\pi$ the condition \cite[(11.2.9)]{N00}.
Hence we can apply the argument in \cite[\S 11.3]{N00}.
Set $G_\diamondsuit=G_{W_i}\times G_{V_{\circ i}}$.
We get an algebra homomorphism
$$K^{G_\diamondsuit\times\bbC^\times}\big(\calZ(W)_\diamondsuit\big)
\to
K^{G_W\times\bbC^\times}\big(\widehat\calZ(W)\big).$$
Composing it with the algebra homomorphism $\Upsilon$ in \eqref{form4}
yields an algebra homomorphism 
\begin{align}\label{form11}
K^{G_\diamondsuit\times\bbC^\times}\big(\calZ(W)_\diamondsuit\big)\to K_A\big(\widehat\frakM(W)^2,(\hat f_2)^{(2)}\big)_{\widehat\calZ(W)}
\end{align}
The elements $x_{i,n}^\pm$ and $A^\pm_{i,n}$ of the right hand side
 introduced in \eqref{xpm} have obvious lifts in the left one.
Hence, we are reduced to prove the relation  (A.6) among those lifts.
We will prove it as in the case $Q=A_1$ in \S\ref{sec:Q=A1}, using the action of
$K^{G_\diamondsuit\times\bbC^\times}(\calZ(W)_\diamondsuit)$
on $K^{G_\diamondsuit\times\bbC^\times}(\frakM(W)_\diamondsuit)$.
Note that $\widehat\frakM(v_i,W_i\oplus V_{\circ i})$ is a quiver variety of type $A_1$, 
hence $\frakM(v,W)_\diamondsuit$ can also be viewed as a quiver variety of type $A_1$, 
up to the vector bundle $\rho$.
To prove the relation, we use the following formulas, 
to be compared with \eqref{C1} and \eqref{C2}
\begin{align*}
\begin{split}
T\frakM(W)_\diamondsuit&=(q_i^{-2}-1)\End(\calV_i)+q_i\Hom(\calV_i,\calW_i)
+\sum_{o_{ij}=1}q_i^{-c_{ij}}\Hom(\calV_i,\calV_j)+\sum_{o_{ij}=1}q_i^{-c_{ij}}\Hom(\calV_j,\calV_i)\\
T\frakP(W)_\diamondsuit&=(q_i^{-2}-1)\calP_i+q_i\Hom(\calV^+_i,\calW_i)
+\sum_{o_{ij}=1}q_i^{-c_{ij}}\Hom(\calV^+_i,\calV_j)
+\sum_{o_{ij}=1}q_i^{-c_{ij}}\Hom(\calV_j,\calV_i^-)
\end{split}
\end{align*}
where $\calP_i$ is as in \eqref{Pig}.
Let $(\lambda,\mu)$ be a pair of $(w_i+v_{\circ i})$-tuples of weights $v_i$ and $v_i+1$.
Fix a maximal torus $T_\diamondsuit$ in $G_\diamondsuit$.
Let $\underline x_\lambda$ and $\underline x_\mu$ be the corresponding to
a $T_\diamondsuit\times\bbC^\times$-fixed points in
$\frakP(v_i,v_i+1,W)_\diamondsuit.$ 
Let $\calV_{\lambda}$ and $\calL_{\lambda,\mu}$ be as in \eqref{VLg} and let
$\calV_{i,\lambda}$ be the degree $i$-component of $\calV_\lambda$.
If $i\neq j$ then $\calV_{j,\lambda}=\calV_{j,\mu}$, hence we abbreviate $\calV_j$ for both.
We get the following formulas 
\begin{align}\label{form9g}
\begin{split}
\langle \lambda |A^-_{i,n}|\mu\rangle
&=(\calL_{\lambda,\mu})^{\otimes n}\otimes
\Lambda_{-1}\Big((q^{-2}_i-1)\calV^\vee_{i,\lambda}\otimes\calL_{\lambda,\mu}+
\sum_{c_{ij}<0}q_i^{-c_{ij}}\calV_{j}^\vee\otimes\calL_{\lambda,\mu}\Big)\\
&=\ev_{u=\calL_{\lambda,\mu}}\Big(u^n
\Lambda_{-u}\big((q^{-2}_i-1)\calV^\vee_{i,\lambda}+
\sum_{c_{ij}<0}q_i^{-c_{ij}}\calV_{j}^\vee\big)\Big)\\
\langle \mu|A^+_{i,m}|\lambda\rangle
&=(\calL_{\lambda,\mu})^{\otimes m}\otimes\Lambda_{-1}
\Big((1-q_i^{-2})\calL^\vee_{\lambda,\mu}\otimes\calV_{i,\mu}-
q_i\calL^\vee_{\lambda,\mu}\otimes\calW_i-
\sum_{c_{ij}<0}q_i^{-c_{ij}}\calL^\vee_{i,\lambda,\mu}\otimes\calV_{j}\Big)\\
&=(1-q_i^{-2})^{-1}\Res_{u=\calL_{\lambda,\mu}}\Big(u^{m-1}\Lambda_{-u^{-1}}
\big((1-q_i^{-2})\calV_{i,\lambda}
-q_i\calW_i-\sum_{c_{ij}<0}q_i^{-c_{ij}}\calV_{j}\big)\Big)
\end{split}
\end{align}
Recall that for each $\calE$ and each $n\in\bbZ$ we have
\begin{align*}
\Lambda_{-u}\big(q^n\calE^\vee\Big)\,
\Lambda_{-u^{-1}}\big(-q^n\calE\Big)
&=(-u)^{\rk(\calE)}\det\big(q^n\calE^\vee\Big)\,
\Lambda_{-u^{-1}}\big((q^{-n}-q^n)\calE\Big),\\
&=(-u)^{\rk(\calE)}\det\big(q^n\calE\Big)^{-1}\,
\Lambda_{-u}\big((q^n-q^{-n})\calE^\vee\Big)
\end{align*}
Let $v_\lambda$ be the rank of $\calV_\lambda$. We deduce that
\begin{align*}
(1-q_i^{-2})\langle\lambda| A^-_{i,n}A^+_{i,m}|\lambda\rangle
&=(-1)^{v_{\circ i}}q_i^{-(\alpha_i^\vee,\bfc v_\lambda)}\det\big(\calV^\vee_{\circ i}\big)\sum_\mu
\Res_{u=\calV_\mu/\calV_\lambda}\Big(u^{m+n-1+v_{\circ i}}
\\
&\quad\Lambda_{-u^{-1}}\big((q_i-q_i^{-1})\sum_j[c_{ij}]_{q_i}\calV_{j,\lambda}
-q_i\calW_i\big)
\Big),\\
(1-q_i^{-2})\langle\lambda| A^+_{i,m}A^-_{i,n}|\lambda\rangle
&=(-1)^{1+v_{\circ i}}q_i^{-(\alpha_i^\vee,\bfc v_\lambda)}\det\big(\calV^\vee_{\circ i}\big)\sum_\mu
\Res_{u=\calV_\lambda/\calV_\mu}\Big(u^{m+n-1+v_{\circ i}}
\\
&\quad\Lambda_{-u^{-1}}\big((q_i-q_i^{-1})\sum_j[c_{ij}]_{q_i}\calV_{j,\lambda}-q_i\calW_i\big)
\Big).
\end{align*}
The sums are over all $\mu$'s such that $\ux_\lambda\subset\ux_\mu$ and 
$\ux_\mu\subset\ux_\lambda$
are of codimention $\delta_i$ respectively.
Using the residue theorem, we get
\begin{align*}
(q_i-q_i^{-1})\langle\lambda|[A^+_{i,m},A^-_{i,n}]|\lambda\rangle
=
-\Res_{u=0}\Big(u^{m+n-1}\phi_{i,\lambda}(u)\Big)
-\Res_{u=\infty}\Big(u^{m+n-1}\phi_{i,\lambda}(u)\Big)
\end{align*}
where
$$\phi_{i,\lambda}(u)=
(-u)^{v_{\circ i}}
q_i^{1-(\alpha_i^\vee,\bfc v_\lambda)}
\det\big(\calV_{\circ i}\big)^{-1}
\Lambda_{-u^{-1}}\Big((q_i-q_i^{-1})\sum_j[c_{ij}]_{q_i}\calV_{j,\lambda}-q_i\calW_i\Big)
$$
Similarly, given $\lambda$, $\lambda'$ as above such that
$\calV_\lambda\cap\calV_{\lambda'}$ is of codimension one in  
$\calV_\lambda$ and $\calV_{\lambda'}$, we get
\begin{align*}
\langle\lambda'| A^-_{i,n}A^+_{i,m}|\lambda\rangle
&=\langle\lambda'| A^-_{i,n}|\mu\rangle\langle\mu|A^+_{i,m}|\lambda\rangle\\
\langle\lambda'| A^+_{i,m}A^-_{i,n}|\lambda\rangle
&=\langle\lambda'| A^+_{i,m}|\nu\rangle\langle\nu|A^-_{i,n}|\lambda\rangle
\end{align*}
where $\mu$, $\nu$ are such that 
$\calV_\mu=\calV_\lambda+\calV_{\lambda'}$ and $\calV_\nu=\calV_\lambda\cap\calV_{\lambda'}$.
We deduce that
$\langle\lambda'| A^-_{i,n}A^+_{i,m}|\lambda\rangle$ is equal to 
\begin{align*}
&(\calL_{i,\lambda',\mu})^{\otimes n}\otimes(\calL_{i,\lambda,\mu})^{\otimes m}\otimes
\Lambda_{-1}\big((1-q_i^{-2})(\calL_{i,\lambda,\mu}^\vee\otimes\calV_{i,\mu}
-\calL_{i,\lambda',\mu}\otimes\calV_{i,\lambda'}^\vee)-
q_i\calL_{i,\lambda,\mu}^\vee\otimes\calW_i+
\\
&
+\sum_{c_{ij}<0}q_i^{-c_{ij}}(\calL_{i,\lambda',\mu}\otimes\calV_j^\vee-\calL_{i,\lambda,\mu}^\vee\otimes\calV_j)
\big)
\end{align*}
and $\langle\lambda'| A^+_{i,m}A^-_{i,n}|\lambda\rangle$ to
\begin{align*}
&(\calL_{i,\nu,\lambda})^{\otimes n}\otimes(\calL_{i,\nu,\lambda'})^{\otimes m}\otimes
\Lambda_{-1}\big((1-q_i^{-2})(\calL_{i,\nu,\lambda'}^\vee\otimes\calV_{i,\lambda'}
-\calL_{i,\nu,\lambda}\otimes\calV_{i,\nu}^\vee)-
q_i\calL_{i,\nu,\lambda'}^\vee\otimes\calW_i+\\
&\quad
+\sum_{c_{ij}<0}q_i^{-c_{ij}}(\calL_{i,\nu,\lambda}\otimes\calV_j^\vee-\calL_{i,\nu,\lambda'}^\vee\otimes\calV_j)
\big)
\end{align*}
Let $\phi^\pm_{i,\lambda}(u)$ be the expansion of  $\phi_{i,\lambda}(u)$
in non-negative powers of $u^{\mp 1}$. 
We deduce that 
$$(q_i-q_i^{-1})\,\langle \lambda'|[A^+_i(u),A^-_i(v)]|\lambda\rangle=
\delta_{\lambda,\lambda'}\delta(u/v)(\phi^+_{i,\lambda}(u)-\phi^-_{i,\lambda}(u)).$$
Let $\phi^\pm_i(u)$ be the formal series of operators acting on
$K^{G_W\times\bbC^\times}(\widehat\frakM(v,W))$ 
by multiplication by the Fourier coefficients of the
expansions in non-negative powers of $u^{\mp 1}$ of the following rational function
$$\phi_i(u)=
(-1)^{v_{\circ i}}q_i^{1-(\alpha_i^\vee,\bfc v)}
\det\big(\calL_i^\vee\otimes\calV_{\circ i}\big)^{-1}
\,\Lambda_{-u^{-1}}(q_i^{-1}\calW_i)^{-1}
\Lambda_{-u^{-1}}\big(-(q_i-q_i^{-1})\calH_{i,1}\big)$$
We have
\begin{align}\label{Apmg}
(q_i-q_i^{-1})\,[A^+_i(u),A^-_i(v)]=
\delta(u/v)(\phi^+_i(u)-\phi^-_i(u)).
\end{align}
We define similarly 
$$\psi_i(u)=
q_i^{-(\alpha_i^\vee,\bfc v)}
\,\Lambda_{-u^{-1}}(q_i^{-1}\calW_i)^{-1}
\Lambda_{-u^{-1}}\big(-(q_i-q_i^{-1})\calH_{i,1}\big)$$
Then, we have
$$(q_i-q_i^{-1})[x^+_i(u)\,,\,x^-_i(v)]=\delta(u/v)\,(\psi^+_i(u)-\psi^-_i(u)).$$
Note that
\begin{align*}
\begin{split}
\psi_i^\pm(u)&=q_i^{- w_i}q_i^{\pm(\alpha_i^\vee,w-\bfc v)}
\,\Lambda_{-u^{-1}}(q_i^{-1}\calW_i)^{-1}
\,\Lambda_{-u^{\mp 1}}\big(\mp(q_i-q_i^{-1})\calH_{i,\pm 1}\big)^\pm.
\end{split}
\end{align*}
From \eqref{WA}, we deduce that the series $\psi_i^\pm(u)$ above coincide with the series in \eqref{psi+-g},
proving the relation (A.6) with $i=j$.
Note that 
\begin{align}\label{CT}\psi^+_{i,0}=q_i^{-w_i+(\alpha_i^\vee,w-\bfc v)}
 ,\quad
\psi^-_{i,-w_i}=(-1)^{w_i}q_i^{-(\alpha_i^\vee,w-\bfc v)}\det(\calW_i)^{-1}.
\end{align}

\subsubsection{Proof of the relation $\mathrm{(A.6)}$ for $i\neq j$}\label{sec:A6neqs}
It is enough to check that the liftings of the
elements $x_{i,m}^+$ and $x_{j,n}^-$ in the algebra $K^A(\widehat\calZ(W))$
commute with each other.
This follows from the transversality result in Lemma \ref{lem:transverse1} below.
Set $v_2=v_1+\delta_i=v_3+\delta_j$ and $v_4=v_1-\delta_j=v_3-\delta_i$.
Consider the intersections
\begin{align*}
I_{v_1,v_2,v_3}
&=\big(\widehat\frakP(v_1,v_2,W)\times\widehat\frakM(v_3,W)\big)\cap
\big(\widehat\frakM(v_1,W)\times\widehat\frakP(v_2,v_3,W)\big),\\
I_{v_1,v_4,v_3}
&=\big(\widehat\frakP(v_1,v_4,W)\times\widehat\frakM(v_3,W)\big)\cap
\big(\widehat\frakM(v_1,W)\times\widehat\frakP(v_4,v_3,W)\big)
\end{align*}

\begin{lemma}\label{lem:transverse1}
\hfill
\begin{enumerate}[label=$\mathrm{(\alph*)}$,leftmargin=8mm]
\item 
The intersections $I_{v_1,v_2,v_3}$ and $I_{v_1,v_4,v_3}$ are both transversal
in $\widehat\frakM(W)^3$.
\item
There is a $G_W\times\bbC^\times$-equivariant isomorphism 
$I_{v_1,v_2,v_3}\simeq I_{v_1,v_4,v_3}$ which intertwines the sheaves
$(\calL_i\boxtimes\calO)|_{I_{v_1,v_2,v_3}}$ and $(\calO\boxtimes\calL_i)|_{I_{v_1,v_4,v_3}}$,
and the sheaves
$(\calO\boxtimes\calL_j)|_{I_{v_1,v_2,v_3}}$ and $(\calL_j\boxtimes\calO)|_{I_{v_1,v_4,v_3}}$.
\end{enumerate}
\end{lemma}

\begin{proof}
The proof is similar to the proof of Lemma \ref{lem:transverse3} above.
\end{proof}

\subsubsection{Proof of the relations $\mathrm{(A.5)}$ and $\mathrm{(A.7)}$}
We have proved the relation (A.6).
The relations (A.2) to (A.4) are easy to check.
Now, we concentrate on the relations (A.5) and (A.7).
The proof of the relation (A.7) for quantum loop groups
in \S\ref{sec:A7} does not extend to the case of the shifted quantum loop groups,
because in the shifted case the action of the elements $x_{i,0}^-$ may not be locally nilpotent.
We use instead another argument using the critical K-theoretic Hall algebra.
More precisely, we claim that these relations follow from 
Proposition \ref{prop:double} and Corollary \ref{cor:relations57} below.
Indeed, applying the proposition with $\hat f=\hat f_2$ yields 
an $R$-algebra homomorphism
$$\omega^+:K_{\bbC^\times}(\Rep,h)_{\Rep^\nil}\to
K_A(\widehat\frakM(W)^2,(\hat f_2)^{(2)})_{\widehat\calZ(W)}$$
which takes the elements $x_{i,n}$ in \eqref{x+3} to the elements $A^+_{i,n}$ in \eqref{xpmg}.
Hence the relations (A.5) and (A.7) in the algebra
$K_A(\widehat\frakM(W)^2,(\hat f_2)^{(2)})_{\widehat\calZ(W)}/\tor$
follow from the corollary. The relations (A.5) and (A.7) for the elements $x^-_{i,n}$
are proved similarly using the homomorphism $\omega^-$ in \eqref{omega-}.

\subsubsection{End of the proof of Theorem $\ref{thm:shifted}$}
We have defined the map \eqref{map8g}. The compatibility of \eqref{map8g} with the 
$R$-lattices follows from the formulas \eqref{psi+-g} and \eqref{xpmg}.
More precisely, the $R$-subalgebra $\U_R^{-w}(L\frakg)$ of $\U_F^{-w}(L\frakg)$ is generated by 
$$\psi_{i,\mp w^\pm_i}^\pm
,\quad
(\psi^\pm_{i,\mp w_i^\pm})^{-1}
,\quad
h_{i,\pm m}/[m]_q
,\quad
(x^\pm_{i,n})^{[m]}
,\quad
i\in I, n\in\bbZ, m\in\bbN^\times,$$
where $h_{i,\pm m}$ is as in \S\ref{sec:sqg}.
By \eqref{Hg} and \eqref{psi+-g} 
the map \eqref{map8g} takes $h_{i,\pm m}/[m]_q$ into 
$$K_A\big(\widehat\frakM(W)^2,(\hat f_2)^{(2)}\big)_{\widehat\calZ(W)}\,/\,\tor
\,\subset\,
K_A\big(\widehat\frakM(W)^2,(\hat f_2)^{(2)}\big)_{\widehat\calZ(W)}\otimes_RF.$$
Using \eqref{xpmg}, an easy computation similar to the proof of
\cite[thm.~12.2.1]{N00} or \cite[lem.~2.4.8]{VV22} shows that it maps $(x^\pm_{i,n})^{[m]}$ to the same lattice.
The second claim of Part (a) of the theorem follows from the formula \eqref{CT}.

\medskip

\section{Representations of K-theoretic critical convolution algebras}

We now apply the previous constructions to realize geometrically some modules
of quantum loop groups and shifted quantum loop groups.
We will use the generalized preprojective algebra $\widetilde\Pi$ and the corresponding 
quiver Grassmanians $\widetilde\Gr$ and $\widetilde\Gr^\bullet$ introduced in \S\ref{sec:QGr}.
We refer to \S\ref{sec:QG} below for an introduction to the representation theory of quantum loop groups and 
shifted quantum loop groups, including the notions of $q$-characters, Kirillov-Reshetikhin 
modules and prefundamental modules.
The main results are Theorem \ref{thm:HL1} and Theorem \ref{thm:PF1}.
We will use the same notation as in \S\ref{sec:quiver varieties}.

\subsection{Representations of quantum loop groups}\label{sec:HL1}
For each triple
$(i,k,l)$ in $I^\bullet\times\bbN^\times$ we fix a graded vector space
$W_{i,k,l}\in\bfC^\bullet$ of dimension
$$w_{i,k,l}=\delta_{i,k-(l-1)d_i}+\delta_{i,k-(l-3)d_i}+\cdots+\delta_{i,k+(l-1)d_i}.$$
We also fix a regular nilpotent element
$\gamma_{i,k,l}$ in $\frakg^{2d_i}_{W_{i,k,l}}$. 
Fix a triple
$(i_r,k_r,l_r)\in I^\bullet\times\bbN^\times$ for each $r=1,2,\dots,s$.
Recall the notation $K_{i,k,l}$ from \eqref{K}. We define
\begin{align*}
W=\bigoplus_{r=1}^sW_{i_r,k_r,l_r}
,\quad
\gamma=\bigoplus_{r=1}^s\gamma_{i_r,k_r,l_r}
,\quad
K_\gamma=\bigoplus_{r=1}^sK_{i_r,k_r,l_r},
,\quad
KR_\gamma=\bigotimes_{r=1}^sKR_{i_r,k_r-(l_r-1),l_r}.
\end{align*}

\smallskip

\begin{proposition}\label{prop:HL1} 
We have the following homeomorphisms
\hfill
\begin{enumerate}[label=$\mathrm{(\alph*)}$,leftmargin=8mm]
\item 
$\crit(\tilde f_\gamma)\cap\widetilde\frakL(v,W)=\widetilde\Gr_v(K_\gamma)$,
\item
$\crit(\tilde f_{\gamma}^\bullet)\cap\widetilde\frakL^\bullet(v,W)=\widetilde\Gr_v^\bullet(K_\gamma).$
\end{enumerate}
\end{proposition}
 
\begin{proof} 
By \eqref{f6}, the critical set of $\tilde f_\gamma$ in $\widetilde\frakM(v,W)$ is given by the following equations
$$\varepsilon_i^{-c_{ij}}\alpha_{ij}-\alpha_{ij}\varepsilon_j^{-c_{ji}}=
\sum_{i,j\in I}\sum_{k=0}^{-c_{ij}-1}\sgn_{ij}\varepsilon_i^k\alpha_{ij}\alpha_{ji}\varepsilon_i^{-c_{ij}-1-k}
+\sum_{i\in I}a_i^*a_i=
\sum_{i\in I}\varepsilon_ia_i^*-a_i^*\gamma_i=\sum_{i\in I}a_i\varepsilon_i-\gamma_ia_i=0
.$$
From Lemma \ref{lem:nil} we deduce that, as reduced schemes, we have
\begin{align}\label{claim2}
\begin{split}
\crit(\tilde f_{\gamma})\cap\widehat\frakM(v,W)&=
\Big\{x\in\widehat\X(v,W)_s\,;\,(\alpha,\varepsilon)\in\bfD\,,\,
\sum_{i\in I}a_i\varepsilon_i-\gamma_ia_i=0\Big\}\,\Big/\,G_v,\\
\crit(\tilde f_{\gamma})\cap\widetilde\frakL(v,W)&=
\Big\{x\in\widehat\X(v,W)_s\,;\,(\alpha,\varepsilon)\in\bfD^\nil\,,\,
\sum_{i\in I}a_i\varepsilon_i-\gamma_ia_i=0\Big\}\,\Big/\,G_v.
\end{split}
\end{align}
We must prove that there is an homeomorphism
\begin{align*}
\crit(\tilde f_\gamma)\cap\widetilde\frakL(v,W)=\widetilde\Gr_v(K_\gamma).
\end{align*}
Since $\widetilde\frakL^\bullet(W)=\widetilde\frakL(W)^A$ 
by Lemma \ref{lem:fixedpoints}, we will deduce that we also have an homeomorphism
\begin{align*}
\crit(\tilde f_{\gamma}^\bullet)\cap\widetilde\frakL^\bullet(v,W)
=\widetilde\Gr_v^\bullet(K_\gamma).
\end{align*}

To prove the claim,  we first identify $W$ with the $I^\bullet$-graded vector space 
\begin{align*}
W=\bigoplus_{r=1}^sH_{l_r}e_{i_r}[-k_r+(l_r-1)d_{i_r}]
\end{align*}
where $H_{l_r}$ is as in \eqref{Hl}.
Then, we equip $W$ with corresponding obvious $H$-action.
The operator $\gamma$ on $W$ is identified with the multiplication by $\varepsilon$.
The dual $H$-module is
$$W^\vee=\bigoplus_{r=1}^sH_{l_r}e_{i_r}[k_r+(l_r-1)d_{i_r}].$$
Indeed, note that, since the vector space $H_{l_r}$ is graded by $\{0,2d_{i_r},\dots,2(l_r-1)d_{i_r}\}$, the vector space
$H_{l_r}e_{i_r}[k_r+(l_r-1)d_{i_r}]$ is graded by $\{-k_r-(l_r-1)d_{i_r},\dots,-k_r+(l_r-1)d_{i_r}\}$ as 
the graded vector space $W_{i_r,k_r,l_r}^\vee$.
Next we equip the coinduced module
$$\bigoplus_{i\in I}\Hom_H(\widetilde\Pi e_i ,W)[-d_i]$$
with the $(\widetilde\Pi,H)$-bimodule structure given by the regular action on $\widetilde\Pi e_i$ twisted by $\tau$ and 
the $H$-action on $W$, i.e., 
$$((p, h)\cdot f)(x)=h\cdot f(\tau(p)x)
,\quad
\forall p\in\widetilde\Pi
,\quad
x\in\widetilde\Pi e_i
,\quad
h\in H
,\quad
f\in\Hom_H(\widetilde \Pi e_i,W).$$
It is isomorphic  to $K_\gamma$ as a graded $\widetilde\Pi$-module
by \eqref{K}.
Set also
\begin{align}\label{KGV}
K_\gamma^\vee=\bigoplus_{i\in I}\widetilde\Pi e_i\otimes_HW^\vee [d_i].
\end{align}
There is an obvious $\widetilde\Pi$-invariant non-degenerate pairing
$K_\gamma\times K_\gamma^\vee\to\bbC$ which allows us to view $K_\gamma^\vee$ as the dual of $K_\gamma$.
Let $a:K_\gamma\to W$ be the transpose of the inclusion 
$a^\vee:W^\vee\to K^\vee_\gamma$ given by $z\mapsto 1\otimes z$.
We have $a=\bigoplus_{i\in I}a_i$ where $a_i$ is the following map
$$
a_i:K_\gamma\to W_i
,\quad
f\mapsto f(e_i).$$
The map $a_i$ kills the subspace $e_jK_\gamma$ for each $j\neq i$ and 
it is homogeneous of degree $-d_i$.
The map $a$ intertwines the left $\varepsilon$-action on $K_\gamma$ with the operator $\gamma$ of $W$.

Now, we consider the quiver Grassmanian.
Let $\widehat\Gr_v(K_\gamma)$ be the set of all injective $I$-graded linear maps $f:\bbC^v\to K_\gamma$
whose image is a $\widetilde\Pi$-submodule of $K_\gamma$.
For each $\widetilde\Pi$-submodule $V\subset K_\gamma$, the action of $\alpha,$ $\varepsilon$ 
on $V$ and the restriction of the map $a$ to $V$ yields a tuple 
$$(\alpha_V\,,\,\varepsilon_V\,,\,a_V)\in\widehat\X(V,W).$$ 
This tuple is stable because $K^\vee_\gamma$ 
is generated by $\Im(a^\vee)$ as a $\widetilde\Pi$-module, 
hence $\{0\}$ is the only $\widetilde\Pi$-submodule of $K_\gamma$ contained into $\Ker(a)$.
Further,  the pair $(\alpha_V,\varepsilon_V)$ is nilpotent because the 
$\widetilde\Pi$-module $K_\gamma$ is nilpotent.
Thus, we have a $G_v$-equivariant morphism
\begin{align}\label{map7}
\widehat\Gr_v(K_\gamma)\to\Big\{x\in\widehat\X(v,W)_s\,;\,(\alpha,\varepsilon)\in\bfD^\nil\,,\,
\sum_{i\in I}a_i\varepsilon_i-\gamma_i a_i=0\Big\}
,\quad
f\mapsto x_f=(\alpha_f,\varepsilon_f,a_f)
\end{align}
where $V$ is the image of $f$ and
$$\alpha_f=f^{-1}\circ\alpha_V\circ f
,\quad
\varepsilon_f=f^{-1}\circ\varepsilon_V\circ f
,\quad
a_f=a_V\circ f.$$ 

We claim that the morphism \eqref{map7} is injective on closed points. 
Indeed, assume that 
$$x_{f_1}=x_{f_2}
,\quad
f_1,f_2\in\widehat\Gr_v(K_\gamma).$$
Consider the map $h=f_1-f_2$.
Since $a_{f_1}=a_{f_2}$, we have 
$$\Im(h)\subset\Ker(a)\subset K_\gamma$$ and $\Im(h)$
is preserved by the action of $\alpha$ and $\varepsilon$. 
Thus $h=0$ because $\{0\}$ is the only $\widetilde\Pi$-submodule of $K_\gamma$ contained into $\Ker(a)$. 

 The quotient by the $G_v$-action yields a torsor
$$\widehat\Gr_v(K_\gamma)\to\widetilde\Gr_v(K_\gamma).$$
 Hence, by \eqref{claim2}, the morphism \eqref{map7} descends to a morphism of reduced schemes
 \begin{align}\label{map9}
\widetilde\Gr_v(K_\gamma)\to\crit(\tilde f_{\gamma})\cap\widetilde\frakL(v,W)
\end{align}
which is injective on closed points.
We claim that it is indeed a bijection on closed points.
To prove this, it is enough to check that  the inclusion \eqref{map9} is surjective.
To do so, we fix a point $x=(\alpha_v,a_v,\varepsilon_v)$ in the right hand side of \eqref{map7}.
The pair $(\alpha_v,\varepsilon_v)$ equips $\bbC^v$ with compatible actions
of $\widetilde\Pi$ and $H$.
The adjunction of coinduction yields an isomorphism
$$\Hom_H(\bbC^v,W)=\Hom_{\widetilde\Pi}(\bbC^v,K_\gamma).$$
Thus the map $a_v:\bbC^v\to W$ yields a $\widetilde\Pi$-module homomorphism 
$f:\bbC^v\to K_\gamma$ such that $a_v=a\circ f$.
The map $f$ is injective because the tuple $x$ is stable.
By construction, the tuple $\underline x_v$ is the image of the point $\Im(f)\in \widetilde\Gr_v(K_\gamma)$ by the inclusion \eqref{map9}.

To conclude, we have proved that the morphism \eqref{map9} is bijective on closed points. Since both sides are projective varieties,
this morphism takes a closed subset to a closed subset, hence the inverse is continuous, proving that \eqref{map9} is an homeomorphism.
A priori the varieties there are not known to be normal, hence \eqref{map9} could not be invertible. 
Probably this can be proved as in Shipman's work \cite{S10} but we will not need this.
\end{proof}

\begin{corollary}\label{cor:HL1}
Let $A\subset G_W\times\bbC^\times$ be any closed subgroup.\hfill
\begin{enumerate}[label=$\mathrm{(\alph*)}$,leftmargin=8mm]
\item
If $\widetilde\Gr^\bullet_v(K_\gamma)=\emptyset$ then
$K(\widetilde\frakM^\bullet(v,W),\tilde f_\gamma^\bullet)=
K(\widetilde\frakM^\bullet(v,W),\tilde f_\gamma^\bullet)_{\widetilde\frakL^\bullet(v,W)}=0.$ 
\item 
If $\widetilde\Gr_v(K_\gamma)=\emptyset$ then
$K_A(\widetilde\frakM(v,W),\tilde f_\gamma)=
K_A(\widetilde\frakM(v,W),\tilde f_\gamma)_{\widetilde\frakL(v,W)}=0$.
\end{enumerate}
\end{corollary}

\begin{proof}
If $\widetilde\Gr^\bullet_v(K_\gamma)=\emptyset$,
then Proposition \ref{prop:HL1} implies that
$$\widetilde\frakL^\bullet(v,W)\subset\widetilde\frakM^\bullet(v,W)\setminus\crit(\tilde f^\bullet_\gamma).$$
By \eqref{excision} and \eqref{crit}, we deduce that
$$
K(\widetilde\frakM^\bullet(v,W),\tilde f^\bullet_\gamma)_{\widetilde\frakL^\bullet(v,W)}=
K(\widetilde\frakM^\bullet(v,W)\setminus\crit(\tilde f^\bullet_\gamma),\tilde f^\bullet_\gamma)_{\widetilde\frakL^\bullet(v,W)}=0
.$$
Similarly, if $\widetilde\Gr_v(K_\gamma)=\emptyset$ then we have
$$
K_A(\widetilde\frakM(v,W),\tilde f_\gamma)_{\widetilde\frakL(v,W)}=0
.$$
We must check that we also have
$$K_A(\widetilde\frakM(v,W),\tilde f_\gamma)=0.$$
To do so, by \eqref{crit} and Proposition \ref{prop:HL1}, it is enough to prove that 
$$\crit(\tilde f_\gamma)\cap\widetilde\frakL(v,W)=\emptyset\Rightarrow\crit(\tilde f_\gamma)=\emptyset.$$
Indeed, we will prove that
$$\crit(\tilde f_\gamma)\neq\emptyset
\Rightarrow\crit(\tilde f_\gamma)\cap\widetilde\frakL(v,W)\neq\emptyset.$$
We equip $\widetilde\frakM(W)$ and $\widetilde\frakM_0(W)$  
with the $\bbC^\times$-action $\diamond$ such that $\alpha_{ij}$, $a_i$, $a^*_i$ have degree 1 and 
$\varepsilon_i$ degree 0. This $\bbC^\times$-action preserves the map 
$\tilde\pi:\widetilde\frakM(W)\to\widetilde\frakM_0(W)$ and the subset $\crit(\tilde f_\gamma)$ in
$\widetilde\frakM(W)$.
Hence it is enough to prove that 
\begin{align}\label{cc}
\crit(\tilde f_\gamma)\neq\emptyset
\Rightarrow
\crit(\tilde f_\gamma)^{\bbC^\times}\neq\emptyset\ \text{and}\ \crit(\tilde f_\gamma)^{\bbC^\times}
\subset\widetilde\frakL(W).
\end{align}
To do this, we claim that
the element $\varepsilon$ is nilpotent
for each $\underline x\in\crit(\tilde f_\gamma)$ with $x=(\alpha,\varepsilon,a,a^*)$. 
Recall that the ideal of $\{0\}$ in $\bbC[\widetilde\frakM_0(W)]$ is generated by
the functions $h_{A,M}$ and $h_N$ such that
\begin{align*}
h_{A,M}(\ux)=\Tr_W(AaMa^*)
,\quad
h_N(\ux)=\Tr_V(N)
\end{align*}
where $A$, $M$ and $N$ run into
$\frakg_W$, $\bbC\widetilde Q$ and
$\bbC\widetilde Q_+$.
We deduce that the limit $$\lim_{t\to 0}t\diamond x$$ exists for each closed point 
$x\in\crit(\tilde f_\gamma)$. Further, we have 
$\crit(\tilde f_\gamma)^{\bbC^\times}\subset\widetilde\frakL(W)$.
This proves \eqref{cc}.
To prove that $\varepsilon$ is nilpotent, fix an integer $d>0$ such that $\gamma^d=0$.
Recall the element $\omega$ in \eqref{aeo}.
Then, for each $x$ as above,
the subspace $\omega^d(V)\subset V$ is preserved by the elements
$\alpha_{ij}$ and $\varepsilon_i$ and it is contained into $\Ker(a)$.
Thus it is zero because the tuple $x$ is stable.
Hence $\varepsilon$ is nilpotent.
Part (b) is proved. Part (a) is proved in the same way.

\end{proof}

\begin{remark}\label{rem:fd}
The varieties $\widetilde\frakM(v,W)$ and $\widetilde\frakM^\bullet(v,W)$ may be
non-empty for infinitely many dimension vectors $v$.
However, since the graded $\widetilde\Pi$-module $K_\gamma$ is finite dimensional,
the Grassmanian $\widetilde\Gr^\bullet(K_\gamma)$ is a variety.
Hence, Corollary \ref{cor:HL1} implies that the vector spaces
$K_A(\widetilde\frakM(v,W),\tilde f_\gamma)$,
$K_A(\widetilde\frakM(v,W),\tilde f_\gamma)_{\widetilde\frakL(W)}$,
$K(\widetilde\frakM^\bullet(v,W),\tilde f_\gamma^\bullet)$ and
$K(\widetilde\frakM^\bullet(v,W),\tilde f_\gamma^\bullet)_{\widetilde\frakL^\bullet(W)}$
vanish for all but finitely many $v$'s.
\end{remark}

\begin{theorem}\label{thm:HL1} 
Assume that $W=W_{i,k,l}$ and $\gamma=\gamma_{i,k,l}$. The $\U_\zeta(L\frakg)$-modules
$K(\widetilde\frakM^\bullet(W),\tilde f_\gamma^\bullet)$ and
$K(\widetilde\frakM^\bullet(W),\tilde f_\gamma^\bullet)_{\underline{\widetilde\frakL^\bullet(W)}}$
are simple and isomorphic to the Kirillov-Reshetikhin module $KR_{i,k,l}$.
\end{theorem}

\begin{proof}
Let first prove that the $U_\zeta(L\frakg)$-module
$K(\widetilde\frakM^\bullet(W),\tilde f_\gamma^\bullet)_{\underline{\widetilde\frakL^\bullet(W)}}$
is isomorphic to $KR_\gamma$.
The proof uses the $q$-characters.
We will prove that the $q$-characters of $K(\widetilde\frakM^\bullet(W),\tilde f_\gamma^\bullet)$ and
$K(\widetilde\frakM^\bullet(W),\tilde f_\gamma^\bullet)_{\underline{\widetilde\frakL^\bullet(W)}}$ 
contain only one 
$\ell$-dominant monomial. This implies that both modules are special, hence simple by \S\ref{sec:qg}. 
Recall that
$$w=w_{i,k,l}=\delta_{i,k-(l-1)d_i}+\delta_{i,k-(l-3)d_i}+\cdots+\delta_{i,k+(l-1)d_i}.$$
By \eqref{w-cv} and \eqref{AY}, we have
$$e^{w-\bfc v}=m_{i,k,l}\prod_{j,r}A_{j,r}^{-v_{j,r}},\quad
m_{i,k,l}=Y_{i,k-(l-1)d_i}\cdots Y_{i,k+(l-1)d_i}.$$
The graded $\widetilde\Pi$-module $K_\gamma^\vee$ in \eqref{KGV}
is generated by the element
$e_i\otimes e_i$. 
Since $e_i\otimes e_i$ has the degree $-k-ld_i$,
any non-zero graded quotient $\widetilde\Pi$-module $V^\vee$ of $K_\gamma^\vee$
has a non-zero element of degree $(i,-k-ld_i)$.
Taking the dual, we deduce that any non-zero graded $\widetilde\Pi$-submodule $V\subset K_\gamma$
has a non-zero element of degree $(i,k+ld_i)$.
So, given $v\in\bbN I^\bullet\setminus\{0\}$ such that
the quiver Grassmanian $\widetilde\Gr_v^\bullet(K_\gamma)$ is non-empty, we have
$$e^{w-\bfc v}\in m_{i,k,l}\,A_{i,k+ld_i}^{-1}\,\bbZ[A_{j,r}^{-1}\,;\,(j,r)\in I^\bullet].$$
By Corollary \ref{cor:HL1} we have
$$\widetilde\Gr^\bullet_v(K_\gamma)=\emptyset
\Rightarrow
K(\widetilde\frakM^\bullet(v,W),\tilde f_\gamma^\bullet)_{\underline{\widetilde\frakL^\bullet(v,W)}}=0.$$ 
Further, by definition of the $\U_\zeta(L\frakg)$-action 
on $K(\widetilde\frakM^\bullet(W),\tilde f_\gamma^\bullet)_{\underline{\widetilde\frakL^\bullet(W)}}$, 
the subspace 
$$K(\widetilde\frakM^\bullet(v,W),\tilde f_\gamma^\bullet)_{\underline{\widetilde\frakL^\bullet(v,W)}}$$
is an $\ell$-weight space of $\ell$-weight $e^{w-\bfc v}$.
Thus, since $\widetilde\frakM^\bullet(0,W)$ is a point, we have
$$q\ch(K(\widetilde\frakM^\bullet(W),\tilde f_\gamma^\bullet)_{\underline{\widetilde\frakL^\bullet(W)}})
\in m_{i,k,l}\,\big(1+A_{i,k+ld_i}^{-1}\,\bbZ[A_{j,r}^{-1}\,;\,(j,r)\in I^\bullet]\big).$$
The monomial $m_{i,k,l}\,A_{i,k+ld_i}^{-1}$ is right-negative by \cite[lem.~4.4]{H06},
see \eqref{RN1} for more details.
Using \eqref{RN2}, we deduce that
the $q$-character of
$K(\widetilde\frakM^\bullet(W),\tilde f^\bullet_\gamma)_{\underline{\widetilde\frakL^\bullet(W)}}$
contains a unique $\ell$-dominant monomial, which is equal to $e^w$.
Hence 
$$K(\widetilde\frakM^\bullet(W),\tilde f_\gamma^\bullet)_{\underline{\widetilde\frakL^\bullet(W)}}=KR_\gamma.$$
By Corollary \ref{cor:HL1} we also have
$K(\widetilde\frakM^\bullet(v,W),\tilde f_\gamma^\bullet)=0$ 
if $\widetilde\Gr^\bullet_v(K_\gamma)=\emptyset$.
Hence, the same argument as above implies that
$$K(\widetilde\frakM^\bullet(W),\tilde f_\gamma^\bullet)=KR_\gamma.$$
\end{proof}

\begin{remark}
The obvious maps
$$K(\widetilde\frakM^\bullet(v,W),\tilde f_\gamma^\bullet)_{\underline{\widetilde\frakL^\bullet(W)}}
\to
K(\widetilde\frakM^\bullet(v,W),\tilde f_\gamma^\bullet)_{{\widetilde\frakL^\bullet(W)}}
\to K(\widetilde\frakM^\bullet(v,W),\tilde f_\gamma^\bullet)$$
are invertible for each $v\in\bbN I$, because the left hand side is non-zero for $v=0$ and
both sides
are simple $\U_\zeta(L\frakg)$-modules.
\end{remark}

A similar result holds for irreducible tensor products of Kirillov-Reshetikhin modules.
We will only write it in the symmetric case.

\begin{proposition}\label{prop:TPKR} Assume that the Cartan matrix is symmetric.
Fix $(i_1,k_1,l_1),\dots,(i_s,k_s,l_s)$ in $I^\bullet\times\bbN^\times$ such
that  either the condition $\mathrm{(a)}$ or $\mathrm{(b)}$ below holds
for some integer $l$
\hfill
\begin{enumerate}[label=$\mathrm{(\alph*)}$,leftmargin=8mm]
\item
$k_r\geqslant l$ and $[k_r-2(l_r-1),k_r]=(k_r-2\bbN)\cap[l,k_r]$ for all $r$, and
$$W=\bigoplus_{r=1}^sW_{i_r,k_r-(l_r-1),l_r}
 ,\quad
\gamma=\bigoplus_{r=1}^s\gamma_{i_r,k_r-(l_r-1),l_r}
,\quad
KR_\gamma=\bigotimes_{r=1}^sKR_{i_r,k_r-(l_r-1),l_r}
$$
\item
$k_r\leqslant l$ and $[k_r,k_r+2(l_r-1)]=(k_r+2\bbN)\cap[k_r,l]$ for all  $r$, and
$$W=\bigoplus_{r=1}^sW_{i_r,k_r+(l_r-1),l_r}
 ,\quad
\gamma=\bigoplus_{r=1}^s\gamma_{i_r,k_r+(l_r-1),l_r}
,\quad
KR_\gamma=\bigotimes_{r=1}^sKR_{i_r,k_r+(l_r-1),l_r}
$$
\end{enumerate}
Then the $\U_\zeta(L\frakg)$-modules
$K(\widetilde\frakM^\bullet(W),\tilde f_\gamma^\bullet)$ and
$K(\widetilde\frakM^\bullet(W),\tilde f_\gamma^\bullet)_{\underline{\widetilde\frakL^\bullet(W)}}$
 are simple and isomorphic to $KR_\gamma$.
\end{proposition}

\begin{proof}
In both cases the $\U_\zeta(L\frakg)$-module $KR_\gamma$ is irreducible by
\cite[thm.~4.11]{FH15}. 
Let $M$ denote either the $\U_\zeta(L\frakg)$-module
$K(\widetilde\frakM^\bullet(W),f_\gamma^\bullet)$ or
$K(\widetilde\frakM^\bullet(W),f_\gamma^\bullet)_{\underline{\widetilde\frakL^\bullet(W)}}$.
We consider accordingly the subspace
$$M_v=K(\widetilde\frakM^\bullet(v,W),f_\gamma^\bullet)\text{\ or\ }
K(\widetilde\frakM^\bullet(v,W),f_\gamma^\bullet)_{\underline{\widetilde\frakL^\bullet(v,W)}}.$$
Then $M_v$ is an $\ell$-weight space of $M$ of $\ell$-weight $\Psi_{w-\bfc v}$.
Corollary \ref{cor:HL1} gives an upper bound on the monomials in the $q$-character of $M$.
Now, we consider the cases (a) and (b) separately.
We abbreviate $$\calA=\bbZ[A_{j,r}^{-1}\,;\,(j,r)\in I^\bullet].$$

Let first consider the case (b).
The graded $\widetilde\Pi$-module $K_{i,k,l}$ is cogenerated by an element of degree $(i,k+l)$,
see the proof of Theorem \ref{thm:HL1}.
Hence, for any non-zero graded $\widetilde\Pi$-submodule $V\subset K_\gamma$, we have
$$\sum_{r=1}^sV_{i_r,k_r+(2l_r-1)}=
\sum_{\substack{1\leqslant r\leqslant s\\k_r\in l+2\bbZ}}V_{i_r,l+1}+
\sum_{\substack{1\leqslant r\leqslant s\\k_r\in l-1+2\bbZ}}V_{i_r,l}\neq\{0\}.$$
Using Corollary \ref{cor:HL1} and \eqref{AY} we deduce that
$$q\ch(M)\in m\Big(1+\sum_i(A_{i,l+1}^{-1}\calA+A_{i,l}^{-1}\calA)\Big)
,\quad
m=\prod_{r=1}^sY_{i_r,k_r}Y_{i_r,k_r+2}\cdots Y_{i_r,k_r+2(l_r-1)}.$$
Therefore, all monomials in $q\ch(M)$ are right-negative except $m$ by \eqref{RN2}, since
$mA_{i,l+1}^{-1}$ and $mA_{i,l}^{-1}$ are right-negative because $l+1>k_r+2(l_r-1)$.
Hence the $\U_\zeta(L\frakg)$-module $M$ is irreducible isomorphic to $KR_\gamma$.

Now we consider the case (a). 
We equip the categories $\bfC$ and $\bfC^\bullet$ with the duality functors such that
$D(W)_i=(W_i)^\vee$ and $D(W)_{i,r}=(W_{i,-r})^\vee$ respectively.
By \cite[\S 4.6]{VV03}, 
for each $W\in\bfC$ there is an isomorphism of algebraic varieties
$$\omega:\frakM(W)\to\frakM(D(W))$$ which intertwines the action of the element
$(g,z)$ of $G_W\times\bbC^\times$ with the action of the element 
$({}^tg^{-1},z)$ of $G_{D(W)}\times\bbC^\times$.
Taking the fixed points locus of some one parameter subgroups of
$G_W\times\bbC^\times$ and $G_{D(W)}\times\bbC^\times$
acting on the quiver varieties, we get for each $W\in\bfC^\bullet$
an isomorphism of algebraic varieties
$$\omega:\frakM^\bullet(W)\to\frakM^\bullet(D(W))$$
which intertwines the functions $f^\bullet_\gamma$ and $f^\bullet_{{}^t\gamma}$ for each element
$\gamma\in\frakg_W^2$. 
Here, the transpose ${}^t\gamma$ is viewed as an element in $\frakg^2_{D(W)}$.
Let
$\overline{M}$ be equal either to $K(\frakM^\bullet(D(W)),f_{{}^t\gamma}^\bullet)$ or to
$K(\frakM^\bullet(D(W)),f_{{}^t\gamma}^\bullet)_{\underline{\frakL^\bullet(D(W))}}$.
The map $\omega$ yields a vector space isomorphism $M\to \overline{M}$.
Both spaces $M$ and $\overline{M}$ are equipped with a representation of $\U_\zeta(L\frakg)$.
Let $f\mapsto \overline f$ be the involution of the ring $\bbZ[Y_{i,r}^{\pm 1}]$ such that
$$\overline{Y_{i,r}}=Y_{i^*,h-2-r}.$$ 
Here $h$ is the Coxeter number and
$i\mapsto i^*$ is the involution of the set $I$ such that $$w_0\alpha_i=-\alpha_{i^*},$$
with $\alpha_i$ the simple root corresponding to the vertex $i$.
By \cite[lem.~4.6]{VV03} we have
$$q\ch(\overline{M})=\overline{q\ch(M)}.$$
Now, we apply the argument in the proof of case (b) with $M$ replaced by $\overline M$.
We deduce that the $q$-character $q\ch(\overline{M})$ admits at most one $\ell$-dominant monomial.
Hence $q\ch(M)$ admits also at most one $\ell$-dominant monomial.
Thus the $\U_\zeta(L\frakg)$-module $M$ is irreducible and the isomorphism $M=KR_\gamma$ follows.
\end{proof}

\medskip

\subsection{Representations of shifted quantum loop groups}

Fix $W\in\bfC^\bullet$ and
fix tuples $(i_1,k_1,l_1),\dots,(i_s,k_s,l_s)$ in $I^\bullet\times\bbN^\times$ 
such that $W=\bigoplus_{r=1}^sS_{i_r,k_r}$, see \eqref{SS} for the notation.  
We abbreviate
\begin{align}\label{WWg}
W_l=\bigoplus_{r=1}^sW_{i_r,k_r-(l_r-1)d_{i_r},l_r}
 ,\quad
\gamma_l=\bigoplus_{r=1}^s\gamma_{i_r,k_r-(l_r-1)d_{i_r},l_r}
\end{align}
Here $W_{i,k,l}$ and $\gamma_{i,k,l}$ are as in \S\ref{sec:HL1}.
The following result can be viewed as a geometric analogue of
the limit procedure of normalized $q$-characters which is used in \cite{HJ12}.

\begin{theorem}\label{thm:limH}
Fix $v\in\bbN I^\bullet$.
Let $l_1,\dots,l_s$ be large enough. We have the isomorphism
\hfill
\begin{gather*}
\begin{split}
K(\widehat\frakM^\bullet(v,W)\,,\,\hat f^\bullet_2)=
K(\widetilde\frakM^\bullet(v,W_l)\,,\,\tilde f^\bullet_{\gamma_l}).
\end{split}
\end{gather*}
\end{theorem}

\begin{proof}
Recall the vector bundle $\nu:\widetilde\frakM(W)\to\widehat\frakM(W)$ in \eqref{VB}.
Recall that $\widehat\frakM^\bullet(W_l)\subset \widehat\frakM(W_l)$, 
see Lemma \ref{lem:fixedpoints}.
We define
$$\overline\frakM(W_l)
=\big\{\ux\in\widehat\frakM(W_l)\,;\,\sum_{i\in I}a_i\varepsilon_i -\gamma_{l,i} a_i=0\big\}
,\quad
\overline\frakM^\bullet(W_l)
=\widehat\frakM^\bullet(W_l)\cap\overline\frakM(W_l).
$$
We first claim that there are isomorphisms
\begin{align}\label{form8}
\begin{split}
K(\widehat\frakM^\bullet(v,W)\,,\,\hat f^\bullet_2)&=
K(\overline\frakM^\bullet(v,W_l)\,,\,\hat f^\bullet_2).
\end{split}
\end{align}
To prove this, we view the algebra $H_1$ in \eqref{Hl}
as an $I^\bullet$-graded algebra whose elements have degrees in $I\times\{0\}$.
Then $\bfC^\bullet$ is identified with the category of all finite dimensional graded $H_1$-modules.
Let $\bfC_H^\bullet$ be the category of all finite dimensional graded $H$-modules.
We abbreviate
$$\Hom^\bullet_H(X,Y)=\bigoplus_{k\in\bbZ}\Hom_{\bfC_H^\bullet}(X,Y[k])
,\quad
\Hom^\bullet_{H_1}(X,Y)=\bigoplus_{k\in\bbZ}\Hom_{\bfC^\bullet}(X,Y[k]).$$
As an $I^\bullet$-graded vector space, we have 
$$W=\bigoplus_{r=1}^sS_{i_r}[-k_r]
,\quad
W_l=\bigoplus_{r=1}^s\Hom^\bullet_{H_1}(H_{l_r},S_{i_r}[-k_r]).$$
We equip $W_l$ with the corresponding $I^\bullet$-graded $H$-module structure.
We have $W_l\in\bfC_H^\bullet$.
The socle of $W_l$ is $\Ker(\varepsilon)$. It is isomorphic to $W$ as an 
$I^\bullet$-graded vector space. 
The nilpotent operator $\gamma_{l,i}$ acts on $W_l$ by multiplication by $\varepsilon_i$. 
For each $H$-module $V$ such that the $\varepsilon_i$'s act nilpotently and for any large enough
$l_1,\dots,l_s$, the $H$-action on $V$ descends to an $H_{l_r}$-action for each $r$.
Thus, we have 
\begin{align}\label{isomW}\Hom^\bullet_H(V,W_l)=\Hom^\bullet_H(V,\Hom^\bullet_{H_1}(H,W)).
\end{align}
Let $i\in\Hom_{\bfC^\bullet}(W,W_l)$ be the inclusion of the socle.
Fix  $p\in\Hom_{\bfC^\bullet}(W_l,W)$ such that $p\circ i=\id$.
By \eqref{isomW} we have the isomorphism
\begin{align}\label{map11}
\Hom^\bullet_H(V,W_l)=\Hom^\bullet_{H_1}(V,W)
,\quad
a\mapsto p\circ a.
\end{align}
Thus the assignment 
$(\alpha,a,\varepsilon)\mapsto (\alpha,p\circ a,\varepsilon)$ yields an isomorphism
\begin{align}\label{sub}
\overline\frakM^\bullet(v,W_l)=\widehat\frakM^\bullet(v,W)
\end{align}
which proves the claim.

Now, we apply the deformed dimensional reduction \eqref{DDR}
along the variable $a^*$. 
We have
\begin{align}\label{split}
\tilde f_{\gamma_l}=
\Tr_{W_l}\Big((\sum_{i\in I}a_i\varepsilon_i -\gamma_{l,i} a_i )a_i^*\Big)+\nu^*\hat f_2.\end{align}
Hence \cite[thm~1.2]{H17b}  yields the isomorphisms
\begin{align}\label{form6}
\begin{split}
K(\overline\frakM^\bullet(v,W_l)\,,\,\hat f^\bullet_2)
&=K(\widetilde\frakM^\bullet(v,W_l)\,,\,\tilde f^\bullet_{\gamma_l}).
\end{split}
\end{align}
Note that, to apply the dimensional reduction as above, we need the map
$\hat f^\bullet_2$ on $\overline\frakM^\bullet(v,W_l)$ to be regular.
This follows from the isomorphism \eqref{sub}
if $l_1,\dots l_s$ are large enough, because $\widehat\frakM^\bullet(v,W)$ is smooth. 
\end{proof}


We now explain an analogue of Propositions \ref{prop:HL1}, \ref{prop:TPKR} and
Theorem \ref{thm:HL1} for shifted quantum loop groups.
Fix $W\in\bfC^\bullet$ of dimension $w$. 
Fix tuples $(i_1,k_1),\dots,(i_s,k_s)$ in $I^\bullet$
such that $w=\sum_{r=1}^s\delta_{i_r,k_r}.$  
We set
$I_W=\bigoplus_{r=1}^sI_{i_r,k_r}.$
See \eqref{I} for the notation.

\begin{proposition}\label{prop:HL3} For any $W\in\bfC^\bullet$,
we have an homeomorphism
$$\crit(\hat f_2^\bullet)=\widetilde\Gr^\bullet\!(I_W)
\subset\widehat\frakL^\bullet(W).$$
\end{proposition}

\begin{proof}
From \eqref{critf2}  we deduce that
\begin{align}\label{f8}
\begin{split}
\crit(\hat f^\bullet_2)&=\Big\{x\in\widehat\X^\bullet(v,W)_s\,;\,(\alpha,\varepsilon)\in\bfD^\bullet\Big\}\,\Big/\,G^0_v.
\end{split}
\end{align}
By \S\ref{sec:QGr}, any module in $\bfD^\bullet$ is nilpotent.
Hence, from \eqref{f8} and Lemma \ref{lem:nil} we deduce that
\begin{align*}
\crit(\hat f_2^\bullet)=\crit(\hat f^\bullet_2)\cap\widehat\frakL^\bullet(v,W)&=
\Big\{x\in\widehat\X^\bullet(v,W)_s\,;\,(\alpha,\varepsilon)\in\bfD^\bullet
\Big\}\,\Big/\,G^0_v.
\end{align*}
Using this isomorphism, the proof of the proposition is similar to the proof of Proposition \ref{prop:HL1}. 
More precisely, let $\underline\alpha$ and $\underline\varepsilon$ denote the action of the
elements $\alpha,\varepsilon\in\widetilde\Pi$ on the module $I_W$. 
Recall that $$I_{i_r,k_r}=(\Pi(\infty)e_{i_r})^\vee[-k_r].$$
Hence, the evaluation at the element
$e_{i_r}$ in $\Pi(\infty)e_{i_r}$ yields a map $I_{i_r,k_r}\to\bbC$.
Taking the sum over all $r$'s
we get an $I^\bullet$-graded linear map $$\ua:I_W\to W.$$
Let $\widehat\Gr_v^\bullet(I_W)$ be the set of injective $I^\bullet$-graded linear maps $f:\bbC^v\to I_W$
whose image is a $\widetilde\Pi$-submodule of $I_W$.
There is a $G^0_v$-equivariant map 
 \begin{align}\label{homeo2}
 \widehat\Gr_v^\bullet(I_W)\to\widehat\X^\bullet(v,W)
 ,\quad
f\mapsto x=
(f^{-1}\circ\underline\alpha\circ f\,,\,\ua\circ f\,,\,0\,,\,f^{-1}\circ\underline\varepsilon\circ f)\end{align}
The tuple $x$ above is stable because $W$ is isomorphic to the socle of the $\widetilde\Pi$-module $I_W$
and the composed map $W\to I_W\to W$ is the identity.
Thus, the map \eqref{homeo2} factorizes to a morphism 
$$\widetilde\Gr_v^\bullet(I_W)\to
\crit(\hat f^\bullet_2)\cap\widehat\frakL^\bullet(v,W)$$
of reduced schemes.
This morphism is a bijection on closed points.
\end{proof}

Following \S\ref{sec:sqg} below, for any graded vector space
$W\in\bfC^\bullet$ of dimension $w=(w_{i,k})$ in $\bbN I^\bullet$, the symbol
$L^-(w)$ denotes the simple $\U^{-w}_\zeta(L\frakg)$-module in $\bfO_w$ with $\ell$-highest weight 
$$\Psi_{-w}=\Big(\prod_{k\in\bbZ}(1-\zeta^k_i/u)^{-w_{i,k}}\Big)_{i\in I}.$$

\begin{theorem}\label{thm:PF1}
Let $w=\delta_{i,k}$.
The representations of $\U_\zeta^{-w}(L\frakg)$ in 
$$K(\widehat\frakM^\bullet(W),\hat f^\bullet_2)
,\quad
K(\widehat\frakM^\bullet(W),\hat f^\bullet_2)_{\underline{\widehat\frakL^\bullet(W)}}$$
are both isomorphic to the simple module $L^-(w)$. 
\end{theorem}

\begin{proof}
We first prove that there is an isomorphism
$L^-(w)=K(\widehat\frakM^\bullet(W)\,,\,\hat f^\bullet_2)$.
The $\U_\zeta^{-w}(L\frakg)$-module
$K(\widehat\frakM^\bullet(W)\,,\,\hat f^\bullet_2)$
is of highest $\ell$-weight $\Psi_{-w}$.
Hence, it is enough to prove that $L^-(w)$ 
and $K(\widehat\frakM^\bullet(W)\,,\,\hat f^\bullet_2)$
have the same normalized $q$-characters. 
Let 
$$W_l=W_{i,k-(l-1)d_i,l}
 ,\quad
\gamma_l=\gamma_{i,k-(l-1)d_i,l}. $$
By Theorem \ref{thm:HL1}, we have
$$L(W_l)=KR_{W_l}=
K(\widetilde\frakM^\bullet(W_l)\,,\,\tilde f^\bullet_{\gamma_l}).$$
By \cite{HJ12} the normalized $q$-character of $L^-(w)$ is the limit of the 
normalized $q$-characters of the finite dimensional simple $\U_\zeta(L\frakg)$-modules $L(W_l)$ as $l\to\infty$.
Further, Theorem \ref{thm:limH} implies
 that, for each $v\in\bbN I^\bullet$, for $l$ large enough we have
$$K(\widehat\frakM^\bullet(v,W)\,,\,\hat f^\bullet_2)=
K(\widetilde\frakM^\bullet(v,W_l)\,,\,\tilde f^\bullet_{\gamma_l}).$$
Hence, taking the limit as $l\to\infty$, we deduce that the normalized $q$-characters of 
$L^-(w)$ and $K(\widehat\frakM^\bullet(W)\,,\,\hat f^\bullet_2)$ are the same.
By \cite[cor.~3.18]{PV11}, any object of $\DCoh(\widehat\frakM^\bullet(W)\,,\,\hat f^\bullet_2)$ is supported
on the critical set of $\hat f^\bullet_2$. 
Further, by Proposition \ref{prop:HL3}, we have $\crit(\hat f_2^\bullet)\subset\widehat\frakL^\bullet(W).$
Hence \eqref{crit} yields
$$K(\widehat\frakM^\bullet(W)\,,\,\hat f^\bullet_2)_{\widehat\frakL^\bullet(W)}=K(\widehat\frakM^\bullet(W)\,,\,\hat f^\bullet_2).$$
Taking the K-theory, we deduce that
$L^-(w)=K(\widehat\frakM^\bullet(W),\tilde f^\bullet_2)_{\underline{\widehat\frakL^\bullet(W)}}$.
\end{proof}

\begin{remark}
Using Proposition \ref{prop:TPKR} instead of Theorem \ref{thm:HL1}, we can also prove the following.
Assume that the Cartan matrix is symmetric.
Fix any $W\in\bfC^\bullet$.
The simple $\U_\zeta^{-w}(L\frakg)$-module $L^-(w)$
is isomorphic to
$$K(\widehat\frakM^\bullet(W),\hat f^\bullet_2)
,\quad
K(\widehat\frakM^\bullet(W),\hat f^\bullet_2)_{\underline{\widehat\frakL^\bullet(W)}}$$
\end{remark}

\section{Relation with K-theoretic Hall algebras}
The goal of this section is to give an algebra homomorphism from K-theoretic Hall algebras to K-theoretic 
convolution algebras, to be used in  the proof of Theorem \ref{thm:shifted}.
The main results here are Proposition \ref{prop:double} and Corollary \ref{cor:relations57}.

\subsection{K-theoretic Hall algebras of a triple quiver with potential}\label{sec:KHA}
We first recall the definition of the K-theoretic Hall algebra of the quiver with potential
$(\widetilde Q,\bfw)$, following \cite{P21}.
The quiver 
$Q$ is as in \S\ref{sec:or}.
Let $\bfw$ be the potential $\bfw_2$ in \S\ref{sec:potential}.
Let $\Rep$ be the moduli stack of representations of the quiver
$\widetilde Q$.
Let $\Rep^\nil\subset\Rep$ be the closed substack parametrizing the nilpotent representations.
Both are defined in \S\ref{sec:Hecke}.
Let $\Rep_v$ be the substack
of $v$-dimensional representations.
Let $\Rep'$ be the  stack of pairs of representations $(x,y)$ 
with an inclusion $x\subset y$. 
The stacks $\Rep$ and $\Rep'$ are smooth and locally of finite type.
Let $\Rep^0\subset\Rep^\nil$ be the zero locus of the function $h$ 
defined in \S\ref{sec:potential}. 
We abbreviate $\frakR^\nil_v=\frakR_v\cap\frakR^\nil$ and $\frakR^0_v=\frakR_v\cap\frakR^0$.
Consider the diagram 
\begin{align*}
\xymatrix{\Rep\times\Rep&\ar[l]_-q\Rep'\ar[r]^-p&\Rep}
 ,\quad
q(x,y)=(x,y/x)
 ,\quad
p(x,y)=y.
\end{align*}
The map $p$ is proper. The map $q$ is smooth.
We equip the stack $\Rep$ with the $\bbC^\times$-action in \S\ref{sec:triple quiver}.
We have $h^{\oplus 2}\circ q=h\circ p$. Hence, there is a functor
\begin{align}\label{Hall1}
\star:\DCoh_{\bbC^\times}(\Rep,h)_{\Rep^\nil}
\times\DCoh_{\bbC^\times}(\Rep,h)_{\Rep^\nil}
\to\DCoh_{\bbC^\times}(\Rep,h)_{\Rep^\nil}
\end{align}
such that
$(\calE,\calF)\mapsto Rp_*Lq^*(\calE\boxtimes\calF).$
This yields a monoidal structure on the triangulated category 
$\DCoh_{\bbC^\times}(\Rep,h)_{\Rep^\nil}$, see \cite{P21} for details.
Recall that $R=R_{\bbC^\times}$ and $F=F_{\bbC^\times}$.
Let $K_{\bbC^\times}(\Rep,h)_{\Rep^\nil}$ be the Grothendieck group of 
$\DCoh_{\bbC^\times}(\Rep,h)_{\Rep^\nil}$.
It is an $R$-algebra.
By \eqref{Upsilon2}, there is an $R$-linear map
\begin{align}\label{UPM}\Upsilon:K^{\bbC^\times}(\Rep^0)\to K_{\bbC^\times}(\Rep,h)_{\Rep^\nil}.
\end{align}
The stack $\Rep^0_{\delta_i}$ is the classifying stack of the group $G_{\delta_i}$.
Let $\calL_i$ be the line bundle on $\Rep^0_{\delta_i}$ associated with
the linear character of $G_{\delta_i}$.
Let $\calU_F^+$ be the $F$-subalgebra of
$$K_{\bbC^\times}(\Rep,h)_{\Rep^\nil}\otimes_{R}F$$
generated by the elements
\begin{align}\label{x+3}
x_{i,n}=\Upsilon(\calL_i^{\otimes n})
,\quad
i\in I,
n\in\bbZ.
\end{align}
Let $\calU_R^+$ be the $R$-subalgebra of $\calU_F^+$ 
generated by the elements
$$(x_{i,n})^{[m]}
,\quad
i\in I, n\in\bbZ, m\in\bbN.$$

\subsection{K-theoretic Hall algebras and critical convolution algebras}\label{sec:double}
Let $\tilde f$ denote either the function $\tilde f_1$ or the function $\tilde f_2$ defined in \S\ref{sec:potential}.
The pair $(\widetilde\frakM(W),\tilde f)$ is a smooth $G_W\times\bbC^\times$-invariant LG-model.
By \S\ref{sec:Kcritalg} there is
a monoidal structure on the category 
$$\DCoh_{G_W\times\bbC^\times}(\widetilde\frakM(W)^2,\tilde f^{(2)})_{\widetilde\calZ(W)}$$
and an associative $R_{G_W\times\bbC^\times}$-algebra structure on the Grothendieck group
$$K_{G_W\times\bbC^\times}(\widetilde\frakM(W)^2,\tilde f^{(2)})_{\widetilde\calZ(W)}.$$
The maps $i$, $\pi$ in \S\ref{sec:Hecke} yield the following  commutative diagram of stacks 
\begin{align*}
\begin{split}
\xymatrix{
\widetilde\calZ(W)\ar[d]&\ar[l]_-i\widetilde\frakP(W)^\nil\ar[r]^-\pi\ar[d]&\Rep^\nil\ar[d]\\
\widetilde\frakM(W)^2&\ar[l]_-i\widetilde\frakP(W)\ar[r]^-\pi&\Rep}
\end{split}
\end{align*}
The right square is Cartesian.
We equip the stack $\Rep$ with the trivial $G_W$-action.
This yields the functors
\begin{align*}
&Ri_*:\DCoh_{G_W\times\bbC^\times}(\widetilde\frakP(W),i^*\tilde f^{(2)})_{\widetilde\frakP(W)^\nil}\to
\DCoh_{G_W\times\bbC^\times}(\widetilde\frakM(W)^2,\tilde f^{(2)})_{\widetilde\calZ(W)},\\
&L\pi^*:\DCoh_{G_W\times\bbC^\times}(\Rep,h)_{\Rep^\nil}\to
\DCoh_{G_W\times\bbC^\times}(\widetilde\frakP(W),\pi^*h)_{\widetilde\frakP(W)^\nil}.
\end{align*}
By Lemma \ref{lem:fh}, composing $Ri_*$ and $L\pi^*$ we get a functor
\begin{align*}
\omega^+:\DCoh_{G_W\times\bbC^\times}(\Rep,h)_{\Rep^\nil}\to 
\DCoh_{G_W\times\bbC^\times}(\widetilde\frakM(W)^2,\tilde f^{(2)})_{\widetilde\calZ(W)}
\end{align*}
Taking the Grothendieck groups we get a $R_{G_W\times\bbC^\times}$-linear map
\begin{align*}
\omega^+:K_{\bbC^\times}(\Rep,h)_{\Rep^\nil}\otimes R_{G_W}\to 
K_{G_W\times\bbC^\times}(\widetilde\frakM(W)^2,\tilde f^{(2)})_{\widetilde\calZ(W)}
\end{align*}

\begin{proposition}\label{prop:double}
$\omega^+$
is an $R_{G_W\times\bbC^\times}$-algebra homomorphism.
\end{proposition}

\begin{proof}
Let us prove that the functor $\omega^+$ has a monoidal structure.
Recall the notation in \S\ref{sec:Hecke}. We consider the stack
$$\widetilde\frakP'(W)=\{(x,y,z)\in\widetilde\frakM(W)^3\,;\,x\subset y\subset z\}.$$
We have the following commutative diagram
\begin{align}
\begin{split}
\xymatrix{
\widetilde\frakM(W)^2\times \widetilde\frakM(W)^2&&\ar[ll]_-{\pi_{12}\times\pi_{23}}\widetilde\frakM(W)^3
\ar[r]^-{\pi_{13}}&\widetilde\frakM(W)^2\\
\widetilde\frakP(W)\times \widetilde\frakP(W)\ar[d]_-{\pi\times\pi}\ar[u]^-{i\times i}&&
\ar[ll]\widetilde\frakP'(W)\ar[u]\ar[r]\ar[d]&
\widetilde\frakP(W)\ar[d]^-\pi\ar[u]_-i\\
\Rep\times\Rep&&\ar[ll]_-q\Rep'\ar[r]^-p&\Rep
}
\end{split}
\end{align}
The left upper square is Cartesian. 
The right lower one either because the set of stable representations of the quiver $\widetilde Q_f$
is preserved by subobjects.
By base change, we get an isomorphism of functors
\begin{align*}
\omega^+\circ\star&=Ri_*\circ L\pi^*\circ Rp_*\circ Lq^*\\
&=R(\pi_{13})_*\circ L(\pi_{12}\times\pi_{23})^*\circ R(i\times i)_*\circ L(\pi\times\pi)^*\\
&=\star\circ R(i\times i)_*\circ L(\pi\times\pi)^*
\end{align*}
More precisely, for the right lower square
we use the flat base change, and for the
left upper square the fact that
$\widetilde\frakP(W)\times \widetilde\frakP(W)$ and $\widetilde\frakM(W)^3$
intersect transversally in $\widetilde\frakM(W)^2\times \widetilde\frakM(W)^2$.
Here the convolution functors $\star$ are as in \eqref{conv2} and \eqref{Hall1}.

\end{proof}

Taking the opposite algebras and the opposite Hecke correspondences
\eqref{opHecke}, we define in a similar way
an $R_{G_W\times\bbC^\times}$-algebra homomorphism
\begin{align}\label{omega-}
\omega^-:K_{\bbC^\times}(\Rep,h)^\op_{\Rep^\nil}\otimes R_{G_W}\to
K_{G_W\times\bbC^\times}(\widetilde\frakM(W)^2,\tilde f^{(2)})_{\widetilde\calZ(W)}.
\end{align}
Let $\omega^\pm$ denote also the composition of the map $\omega^\pm$ and the obvious inclusion
\begin{align}\label{obvincl}
K_{\bbC^\times}(\Rep,h)_{\Rep^\nil}\subset K_{\bbC^\times}(\Rep,h)_{\Rep^\nil}\otimes R_{G_W}.
\end{align}

Taking the triple quiver varieties with simple framing 
$\widehat\frakM(W)$ in \eqref{hatM} and the corresponding Hecke
correspondence $\widehat\frakP(W)$ in \eqref{hatHecke}, we define in the same way
the $R_{G_W\times\bbC^\times}$-algebra homomorphisms
$$\omega^+:K_{\bbC^\times}(\Rep,h)_{\Rep^\nil}\otimes R_{G_W}\to
K_{G_W\times\bbC^\times}(\widehat\frakM(W)^2,\hat f^{(2)})_{\widehat\calZ(W)}.$$
and
$$\omega^-:K_{\bbC^\times}(\Rep,h)^\op_{\Rep^\nil}\otimes R_{G_W}\to
K_{G_W\times\bbC^\times}(\widehat\frakM(W)^2,\hat f^{(2)})_{\widehat\calZ(W)}.$$

Let $|W\rangle$ be the fundamental class of $\widetilde\frakM(0,W)$ in 
$K_{G_W\times\bbC^\times}(\widetilde\frakM(W),\tilde f)_{\widetilde\frakL(W)}$.
We call this element the vacuum vector of 
$K_{G_W\times\bbC^\times}(\widetilde\frakM(W),\tilde f)_{\widetilde\frakL(W)}$. 
Composing the $R$-algebra homomorphism $\omega^-$ and the 
action on $|W\rangle$  yields an $R$-linear map
\begin{align}\label{ovw}
K_{\bbC^\times}(\Rep,h)_{\Rep^\nil}\to 
K_{G_W\times\bbC^\times}(\widetilde\frakM(W),\tilde f)_{\widetilde\frakL(W)}.
\end{align}

\begin{proposition} \label{prop:injective}
If $W=V$, then the map \eqref{ovw} is an embedding
$$\K_{\bbC^\times}(\Rep_v,h)_{\Rep^\nil_v}\to 
K_{G_W\times\bbC^\times}(\widetilde\frakM(v,W),\tilde f)_{\widetilde\frakL(v,W)}.$$
\end{proposition}

\begin{proof}
We abbreviate
\begin{align*}
\widehat\X^\nil(V,W)&=
\{(\alpha,\varepsilon, a)\in\widehat\X(V,W)\,;\,(\alpha,\varepsilon)\in\widetilde\X^\nil(V)\}\\
\widehat\X^\nil(V,W)_s&=\widehat\X(V,W)_s\cap\widehat\X^\nil(V,W).
\end{align*}
By Lemma \ref{lem:nil} we have
$$\widetilde\frakL(v,W)=\widehat\frakL(v,W)=\widehat\X^\nil(V,W)_s/G_V.$$
Hence, we have the following commutative diagram 
$$\xymatrix{
\Rep_v&\ar[l]_-q\widehat\frakM(v,W)\ar[r]^-p&\widetilde\frakM(v,W)\\
\Rep_v^\nil\ar@{^{(}->}[u]&\ar[l]\widehat\frakL(v,W)\ar@{^{(}->}[u]\ar@{=}[r]&\widetilde\frakL(v,W)\ar@{^{(}->}[u]
}$$
The vertical maps and $p$ are the obvious inclusions.
The map $q$ is the forgetting of the framing.
The right square is Cartesian.
The map in the proposition is the composition of \eqref{obvincl} and the chain of maps
\begin{align*}
K_{G_W\times\bbC^\times}(\Rep_v,h)_{\Rep_v^\nil}
\xrightarrow{q^*}
K_{G_W\times\bbC^\times}(\widehat\frakM(v,W),\hat f)_{\widehat\frakL(v,W)}
\xrightarrow{p_*}
K_{G_W\times\bbC^\times}(\widetilde\frakM(v,W),\tilde f)_{\widetilde\frakL(v,W)}
\end{align*}
The map $q^*$ is well-defined because $q^*(h)=\hat f$.
The map $p_*$ is injective because $\widetilde\frakL(v,W)=\widehat\frakL(v,W)$.
To prove the proposition it is enough to check that the map $q^*$ is also injective.
To do this we need more notation.
Recall that $W=V$.
We define
\begin{align*}
\widehat\X(V,W)_\circ&=\{x=(\alpha,\varepsilon,a)\in\widehat\X(V,W)\,;\,a\in G_V\}
,\quad
\widehat\X^\nil(V,W)_\circ=\widehat\X(V,W)_\circ\cap\widehat\X^\nil(V,W)
\end{align*}
and
\begin{align*}
\widehat\frakM(v,W)_\circ&=\widehat\X(V,W)_\circ/G_V,\quad
\widehat\frakL(v,W)_\circ=\widehat\X^\nil(V,W)_\circ/G_V.
\end{align*}
We consider the following commutative diagram
$$
\xymatrix{
\widehat\X(V,W)_\circ\ar[r]^-{\alpha_2}\ar[d]_-{q_2}&\widehat\X(V,W)_s\ar[d]^-{\alpha_1}\\
\widetilde\X(V)&\widehat\X(V,W)\ar[l]_-{q_1}
}
$$
The maps $\alpha_1$, $\alpha_2$ are the obvious open embeddings.
The maps $q_1$, $q_2$ are the forgetting of the framings.
Taking the Grothendieck groups, it yields the following commutative diagram
\begin{align*}
\xymatrix{
K_{\bbC^\times}(\Rep_v,h)_{\Rep_v^\nil}\ar[r]^-{q^*}\ar[rd]_-{q_2^*}&
K_{G_W\times\bbC^\times}(\widehat\frakM(v,W),\hat f)_{\widehat\frakL(v,W)}\ar[d]^-{\alpha_2^*}\\
&K_{G_W\times\bbC^\times}(\widehat\frakM(v,W)_\circ,\hat f)_{\widehat\frakL(v,W)_\circ}
}
\end{align*}
The map $q_2$ yields a $G_W$-torsor
$\widehat\frakM(v,W)_\circ\to\Rep_v$.
Hence the map $q_2^*$ is invertible. 
We deduce that the map $q^*$ is also injective.
\end{proof}

Let $K_{\bbC^\times}(\Rep,h)_{\Rep^\nil}/\tor$ be the image of $K_{\bbC^\times}(\Rep,h)_{\Rep^\nil}$
in $K_{\bbC^\times}(\Rep,h)_{\Rep^\nil}\otimes_RF$.
For each $i\in I$, $n\in\bbN$, the element $x_{i,n}$ of 
$K_{\bbC^\times}(\Rep,h)_{\Rep^\nil}$ in \eqref{x+3}
yield an element in $K_{\bbC^\times}(\Rep,h)_{\Rep^\nil}/\tor$.

\begin{corollary}\label{cor:relations57}
\hfill
\begin{enumerate}[label=$\mathrm{(\alph*)}$,leftmargin=8mm]
\item 
The elements $x_{i,n}$ of the $R$-algebra
$K_{\bbC^\times}(\Rep,h)_{\Rep^\nil}/\tor$ with $i\in I$, $n\in\bbN$,
satisfy the relations $\mathrm{(A.5)}$ and $\mathrm{(A.7)}$, up to some explicit twist.
\item
We have an $R$-algebra homomorphism
$\U_R(L\frakg)^\pm\to K_{\bbC^\times}(\Rep,h)_{\Rep^\nil}/\tor$.
\end{enumerate}
\end{corollary}

\begin{proof}
Setting $\gamma=0$ in \S\ref{sec:rem} we get
an $R$-algebra homomorphism
$$\U_R(L\frakg)\to 
K_{G_W\times\bbC^\times}\big(\widetilde\frakM(W)^2,(\tilde f_1)^{(2)}\big)_{\widetilde\calZ(W)}/\tor,$$
which takes the element $x_{i,n}^-$ in $\U_R(L\frakg)$ to a twisted version of the element
$A^-_{i,n}=\omega^-(x_{i,n})$ defined in  \eqref{xpm}.
Since the relations (A.5) and (A.7) hold in $\U_R(L\frakg)$, they also hold in the algebra
$$K_{G_W\times\bbC^\times}\big(\widetilde\frakM(W)^2,(\tilde f_1)^{(2)}\big)_{\widetilde\calZ(W)}/\tor.$$
Thus the elements $x_{i,n}$ in $K_{\bbC^\times}(\Rep,h)_{\Rep^\nil}/\tor$ 
satisfy the relations (A.5) and (A.7), up to some twist, by Proposition \ref{prop:injective} with 
$\tilde f=\tilde f_1$. Part (a) is proved. Part (b) follows from (a).
\end{proof}

\begin{remark}
In the symmetric case, the corresponding morphism is 
considered in \cite{VV22} where it is proved to be isomorphism.
We do not know if the map $\U_R(L\frakg)^\pm\to K_{\bbC^\times}(\Rep,h)_{\Rep^\nil}/\tor$
is injective.
\end{remark}

\medskip

\appendix

\section{Representations of quantum loop groups}\label{sec:QG}
\subsection{Representations of shifted quantum loop groups}\label{sec:sqg}
This appendix is a reminder on shifted quantum  loop groups.
We follow  \cite{FT19} and \cite{H22}.
Set
$[m]_q=(q^m-q^{-m})/(q-q^{-1})$
and
$[m]_q!=[m]_q[m-1]_q\cdots[1]_q$
for each integer $m>0$.
Let $Q$ be a Dynkin quiver.
Fix $w^+,w^-\in\bbZ I$.
Let $\bfc=(c_{ij})_{i,j\in I}$ be a Cartan matrix,
$\O\subset I\times I$ an orientation of $\bfc$,
and $(d_i)_{i\in I}$ a symmetrizer for $\bfc$.
We define
\begin{align}\label{gij}
q_i=q^{d_i},\quad
g_{ij}(u)=\frac{q_i^{-c_{ij}}u-1}{u-q_i^{-c_{ij}}},\quad
i,j\in I.\end{align}
Consider the formal series 
\begin{align*}
\delta(u)=\sum_{n\in\bbZ}u^n
,\quad
x^\pm_i(u)=\sum_{n\in\bbZ}x^\pm_{i,n}\,u^{-n}
 ,\quad
\psi^+_i(u)=\sum_{n\geqslant-w^+_i}\psi^+_{i,n}\,u^{-n}
 ,\quad
\psi^-_i(u)=\sum_{n\geqslant -w_i^-}\psi^-_{i,-n}\,u^{n}.
\end{align*}
Let $\U_F^{w^+,w^-}(L\frakg)$ be the $(w^+,w^-)$-shifted quantum loop group over $F$ 
with quantum parameter $q$. It is the $F$-algebra generated by 
$$x^\pm_{i,m}
 ,\quad
\psi^\pm_{i,\pm n}
 ,\quad
(\psi^\pm_{i,\mp w_i^\pm})^{-1}
,\quad
i\in I
 ,\quad
m,n\in\bbZ
 ,\quad
n\geqslant-w^\pm_i$$
with the following defining relations  
where $a=+$ or $-$ and $i,j\in I$
\hfill
\begin{enumerate}[label=$\mathrm{(\alph*)}$,leftmargin=11mm,itemsep=1mm]
\item[(A.2)] $\psi^\pm_{i,\mp w_i^\pm}$ is invertible with inverse $(\psi^\pm_{i,\mp w_i^\pm})^{-1}$,
\item[(A.3)] $\psi^a_i(u)\,\psi^\pm_j(v)=\psi^\pm_j(v)\,\psi^a_i(u)$,
\item[(A.4)] $\psi^a_i(u)\,x^\pm_j(v)=x^\pm_j(v)\,\psi^a_i(u)\,g_{ij}(u/v)^{\pm 1}$,
\item[(A.5)] $x^\pm_i(u)\,x^\pm_j(v)=x^\pm_j(v)\,x^\pm_i(u)\,g_{ij}(u/v)^{\pm 1}$,
\item[(A.6)] $(q_i-q_i^{-1})[x^+_i(u)\,,\,x^-_j(v)]=\delta_{ij}\,\delta(u/v)\,(\psi^+_i(u)-\psi^-_j(u))$,
\item[(A.7)] the quantum Serre relations with
$s=1-c_{ij}$ and $i\neq j$
$$\sum_{\sigma\in \frakS_s}\sum_{r=0}^{s}(-1)^r\Big[\begin{matrix}s\\r\end{matrix}\Big]_{q_i}
x^\pm_i(u_{\sigma(1)})\cdots x^\pm_i(u_{\sigma(r)})x^\pm_j(v)x^\pm_i(u_{\sigma(r+1)})\cdots
x^\pm_i(u_{\sigma(s)})=0
$$

\end{enumerate}
\setcounter{equation}{7}
Here the rational function $g_{ij}(u/v)$ is expanded as a power series of $v^{\mp 1}$.
We have a triangular decomposition
$$\U_F^{w^+,w^-}(L\frakg)=\U_F^{w^+,w^-}(L\frakg)^+\otimes\U_F^{w^+,w^-}(L\frakg)^0\otimes
\U_F^{w^+,w^-}(L\frakg)^-$$
where $\U_F^{w_+,w_-}(L\frakg)^\pm$ 
is the subalgebra generated by the $x_{i,n}^\pm$'s
and $\U_F^{w_+,w_-}(L\frakg)^0$ is the subalgebra generated by 
the $\psi_{i,\pm n}^\pm$'s.

Set
$(x^\pm_{i,n})^{[m]}=(x^\pm_{i,n})^{m}/[m]_{q_i}!$
and let the element $h_{i,\pm m}$  be such that
\begin{align*}
\psi_i^\pm(u)=\psi_{i,\mp w^\pm_i}^\pm u^{\pm w^\pm_i}
\exp\Big(\pm(q_i-q_i^{-1})\sum_{m>0}h_{i,\pm m}u^{\mp m}\Big)
 ,\quad
i\in I.
\end{align*}
We define $\U_R^{w^+,w^-}(L\frakg)\subset\U_F^{w^+,w^-}(L\frakg)$ 
to be the $R$-subalgebra generated by 
\begin{align}\label{SLGR}
\psi_{i,\mp w^\pm_i}^\pm 
 ,\quad
(\psi^\pm_{i,\mp w_i^\pm})^{-1}
 ,\quad
h_{i,\pm m}/[m]_{q_i}
 ,\quad
(x^\pm_{i,n})^{[m]}
,\quad
i\in I,
n\in\bbZ,
m\in\bbN^\times.
\end{align}
Note that the relation (A.4) is equivalent to the following relations
\begin{enumerate}[label=$\mathrm{(\alph*)}$,leftmargin=13mm,itemsep=1mm]
\item[(A.4a)] $x_{j,n}^a\,\psi_{i,\mp w^\pm_i}^\pm=q_i^{\pm ac_{ij}}\,\psi_{i,\mp w^\pm_i}^\pm\,x_{j,n}^a$,
\item[(A.4b)] $[h_{i,m},x^\pm_{j,n}]=\pm[m\,c_{ij}]_{q_i}\,x^\pm_{j,n+m}\,/\,m$ for $m\neq 0$.
\end{enumerate}

We fix $\zeta\in\bbC^\times$ which is not a root of unity.
We define $\U_\zeta^{w^+,w^-}(L\frakg)=\U_R^{w^+,w^-}(L\frakg)|_\zeta$,
where $(-)|_\zeta$ is the specialization along the map $R\to\bbC$,
$q\mapsto\zeta$.
We will concentrate on the module categories of the $\bbC$-algebra $\U_\zeta^{w^+,w^-}(L\frakg)$. 
The module categories of the $F$-algebra $\U_F^{w^+,w^-}(L\frakg)$ are similar.
Up to some isomorphism, the algebra $\U_\zeta^{w^+,w^-}(L\frakg)$ only depends on the sum 
$w=w^++w^-$ in $\bbZ I$.
Hence, we may assume that $w^+=0$ and we abbreviate
$\U_\zeta^{w}(L\frakg)=\U_\zeta^{0,w}(L\frakg)$.
We define $\U_F^{w}(L\frakg)$ and $\U_R^{w}(L\frakg)$ similarly.
The category $\bfO_w$ of $\U_\zeta^{w}(L\frakg)$-modules is defined as in 
\cite[def.~4.8]{H22}. 
A tuple $\Psi=(\Psi_i)_{i\in I}$ of rational functions over $\bbC$ 
such that $\Psi_i(u)$ is regular at 0 and of degree $w_i$
is called a $w$-dominant $\ell$-weight. Let
\begin{align*}
\Psi^+_i(u)=\sum_{n\in\bbN}\Psi_{i,n}^+\,u^{-n}
 ,\quad
\Psi^-_i(u)=\sum_{n\geqslant-w_i}\Psi_{i,-n}^-\,u^{n}
\end{align*}
be the expansions of the rational function $\Psi_i(u)$
in non-negative powers of $u^{\mp 1}$.
A representation $V$ in the category $\bfO_w$ 
is of highest $\ell$-weight $\Psi(u)$ if it is generated by a vector $v$ such that
$$x_{i,n}^+\cdot v=0
 ,\quad
\psi^\pm_{i,n}\cdot v=\Psi^\pm_{i,n}v
 ,\quad
i\in I, n\in\bbZ.$$
By \cite[thm.~4.11]{H22}
the simple objects in the category $\bfO_w$ are labelled by the $w$-dominant
$\ell$-weights.
Let $L(\Psi)$ be the unique simple object in $\bfO_w$ of highest $\ell$-weight $\Psi$.
For any module $V\in\bfO_w$ and for any tuple $\Psi=(\Psi_i(u))_{i\in I}$ of rational functions, 
the $\ell$-weight space of $V$ of $\ell$-weight $\Psi$ is
$$V_\Psi=\{v\in V\,;\,(\psi^\pm_{i,n}-\Psi^\pm_{i,n})^\infty\cdot v=0\,,\,i\in I\,,\,n\in\bbN\}.$$
The representation $V$ is a direct sum of its $\ell$-weight spaces.
The $q$-character of $V$ is the (possibly infinite) sum
$$q\!\ch(V)=\sum_\Psi\dim(V_\Psi)\,\Psi.$$
If the module $V$ admits an highest $\ell$-weight,
we may also consider the normalized $q$-character $q\widetilde\ch(V)$,
which is equal to the $q$-character $q\!\ch(V)$ divided by its highest weight monomial.
The map $q\!\ch$ is injective on the Grothendieck group $K_0(\bfO_w)$.
We abbreviate $I^\sharp=I\times\bbC^\times$.
For any $I^\sharp$-tuple $w^\sharp=(w_{i,a})$ in $\bbN I^\sharp$, we consider the tuple of rational functions
$\Psi_{\pm w^\sharp}=(\Psi_{\pm w^\sharp, i})_{i\in I}$ such that
$$\Psi_{\pm w^\sharp, i}(u)=\prod_{a\in\bbC^\times}(1-a/u)^{\pm w_{i,a}}.$$
Note that $w_{i,a}$ is zero except for finitely many $a$'s.
We write $L^\pm(w^\sharp)=L(\Psi_{\pm w^\sharp})$.
We abbreviate $\Psi_{\pm i,a}=\Psi_{\pm \delta_{i,a}}$ and
$L^\pm_{i,a}=L^\pm(\delta_{i,a})$.
We call  $L^\pm_{i,a}$ the positive/negative
prefundamental representation. 
A positive prefundamental representation is one-dimensional,
a negative one is infinite dimensional.
We also abbreviate $L^\pm_{i,k}=L^\pm_{i,\zeta^k}$ for each integer $k$.
To avoid a cumbersome notation, we may use the symbol $w$ for the tuple $w^\sharp\in\bbN I^\sharp$
and we may write $L(w)$ for the corresponding simple module,
hoping it will not create any confusion.

\subsection{Representations of quantum loop groups} \label{sec:qg}
The quantum loop group $\U_F(L\frakg)$ is the quotient
of $\U_F^{0}(L\frakg)$ by the relations 
\begin{align}\label{lastrelation}
\psi^+_{i,0}\,\psi^-_{i,0}=1,\quad i\in I.
\end{align}
We define $\U_R(L\frakg)\subset\U_F(L\frakg)$ 
to be the $R$-subalgebra generated by 
\begin{align}\label{LGR}
\psi^\pm_{i,0}
 ,\quad
h_{i,\pm m}/[m]_{q_i}
 ,\quad
(x^\pm_{i,n})^{[m]}
,\quad
i\in I,
n\in\bbZ,
m\in\bbN^\times
\end{align}
and the $\bbC$-algebra $\U_\zeta(L\frakg)$ to be the specialization at $q=\zeta$.
We have
$\U_F(L\frakg)=\U_R(L\frakg)\otimes_{R}F$.
We have a triangular decomposition
$$\U_F(L\frakg)=\U_F(L\frakg)^+\otimes\U_F(L\frakg)^0\otimes
\U_F(L\frakg)^-$$
and its analogues for the algebras
$\U_R(L\frakg)$ and $\U_\zeta(L\frakg)$ proved in \cite[prop.~6.1]{CP17}.
The $R$-algebra $\U_R(L\frakg)^\pm$ is generated by the quantum divided powers
$(x^\pm_{i,n})^{[m]}$ with
$i\in I$,
$n\in\bbZ$,
$m\in\bbN^\times$. 
Let
\begin{gather*}
\Big[\begin{matrix}\psi_{i,0}^+\,;\,n\\m\end{matrix}\Big]_{q_i}=
\prod_{r=1}^m\frac{q_i^{n-r+1}\,\psi_{i,0}^+-q_i^{-n+r-1}\,\psi_{i,0}^-}{q_i^r-q_i^{-r}}
 ,\quad
i\in I, n\in\bbZ,m\in\bbN^\times
\end{gather*}
The 
$R$-algebra $\U_R(L\frakg)^0$ is generated by the elements 
\begin{align*}
\psi^\pm_{i,0}
 ,\quad
h_{i,\pm m}/[m]_{q_i}
 ,\quad
\Big[\begin{matrix}\psi_{i,0}^+;n\\m\end{matrix}\Big]_{q_i}
 ,\quad
i\in I, n\in\bbZ, m\in\bbN^\times
\end{align*}
such that
\begin{align}\label{psi}
\psi_i^\pm(u)=\sum_{n\in\bbN}\psi^\pm_{i,\pm n}\,u^{\mp n}
=\psi_{i,0}^\pm\exp\Big(\pm(q_i-q_i^{-1})\sum_{m>0}h_{i,\pm m}u^{\mp m}\Big).
\end{align}
A simple module $L(\Psi)$ in the category $\bfO_0$ is finite dimensional if and only if there is a tuple of polynomials
$P=(P_i)_{i\in I}$ with $P_i(0)=1$, called the Drinfeld polynomial, such that
$$\Psi_i(u)=\zeta^{\deg P_i}\,P_i(1/\zeta_i u)\,P_i(\zeta_i/u)^{-1}.$$
For any tuple $w=(w_{i,a})$ in $\bbN I^\sharp$, we consider the tuple of polynomials
$P_{w}=(P_i(u))_{i\in I}$ given by
$P_{i}(u)=\prod_{a\in\bbC^\times}(1-au)^{w_{i,a}}.$
Let $\Psi_{w}$ be the corresponding $\ell$-weight and $L(w)=L(\Psi_{w})$ the 
corresponding finite dimensional module.
The simple module 
$$KR_{i,a,l}=L(w_{i,a,l}),\quad
w_{i,a,l}=\delta_{i,a\zeta_i^{1-l}}+\delta_{i,a\zeta_i^{3-l}}+\cdots+\delta_{i,a\zeta_i^{l-1}}$$
is called a Kirillov-Reshetikhin module.
We may identify the $q$-character $q\!\ch(V)$ of a finite dimensional module $V\in\bfO_0$ 
with the sum of monomials $e^v$ such that
\begin{align*}
q\!\ch(V)=\sum_{v\in\bbZ I^\sharp}\dim(V_{\Psi_v})\,e^v
\end{align*}
where the $\ell$-weight $\Psi_v$ is given by
$\Psi_v=\Psi_{v_+}\cdot\Psi_{v_-}^{-1}$
with $v=v_+-v_-$ and
$v_+,v_-\in\bbN I^\sharp.$
The monomial $e^v$ is called $\ell$-dominant if $v\in\bbN I^\sharp$.
The module $V$ is called special if its $q$-character contains a unique $\ell$-dominant monomial,
see \cite[Def.~10.1]{N04} for ADE types and \cite[Def.~2.10]{H06} in general.
A special module is simple by \cite[\S10]{N04}.
The following notation is standard
\begin{align}\label{AY}
\begin{split}Y_{i,a}&=e^{\delta_{i,a}},\\
A_{i,a}&=e^{\bfc\delta_{i,a}}=Y_{i,a\zeta_i}Y_{i,a\zeta_i^{-1}}
\prod_{c_{ij}=-1}Y_{j,a}^{-1}
\prod_{c_{ij}=-2}Y_{j,a\zeta^{-1}}^{-1}Y_{j,a\zeta}^{-1}
\prod_{c_{ij}=-3}Y_{j,a\zeta^{-2}}^{-1}Y_{j,a}^{-1}Y_{j,a\zeta^2}^{-1}.
\end{split}\end{align}
We view $I^\bullet$ as a subset of $I^\sharp$ such that
$(i,k)\mapsto(i,\zeta^k)$.
Hence we may write
$$Y_{i,k}=Y_{i,\zeta^k},\quad
A_{i,k}=A_{i,\zeta^k},\quad
KR_{i,k,l}=KR_{i,\zeta^k,l},\quad
w_{i,k,l}=w_{i,\zeta^k,l}.$$
For each $v\in\bbZ I^\bullet$ we set 
$$|v|=\max\{k\in\bbZ\,;\,\exists i\in I\,,\,v_{i,k}\neq 0\}.$$
The monomial $e^v$ is called right-negative if we have
\begin{align}\label{RN1}v_{i,|v|}\leqslant 0
,\quad
\forall i\in\bbN.
\end{align}
By \cite{FM01}, \cite[lem.~2.4]{H06}, we have
\begin{align}\label{RN2}
\begin{split}
m\ \text{right-negative} 
&\Rightarrow \text{all\ monomial\ } m'\in
m\,\bbZ[A_{j,r}^{-1}\,;\,(j,r)\in I^\bullet]\ \text{are\ right-negative}\\
&\Rightarrow \text{all\ monomial\ } m'\in
m\,\bbZ[A_{j,r}^{-1}\,;\,(j,r)\in I^\bullet]\
\text{are\ not\ $\ell$-dominant}.
\end{split}
\end{align}

\section{Triple quiver varieties in type $A_1$}
In this appendix we consider the case $Q=A_1$.
Then the torus $T$ is $\bbC^\times\times\bbC^\times$ and
the framed triple quiver $\widetilde Q_f$ is
as in the figure below.
\begin{figure}[htbp]
\centering
\setlength{\unitlength}{1mm}
\begin{picture}(90,27)
    \put(41,11){\circle{5}}
    \put(40,10){$i$}
    \put(39,22){\framebox(4,4){$\scriptstyle i'$}}
    \put(42,13.4){\thicklines\vector(0,1){8.5}}
     \put(40,22){\thicklines\vector(0,-1){8.6}}
   \put(41,5.8){\thicklines\arc[-247,67]{3}}
   \put(45,5){$\scriptstyle\varepsilon$}
    \put(36,17){$\scriptstyle a^*$}
 \put(43,17){$\scriptstyle a$}
\end{picture}
\caption{}
\end{figure}
The group $G_W\times T$ acts on the varieties $\widetilde\frakM(W)$
and $\widehat\frakM(W)$ as in \S\ref{sec:triple quiver}, \ref{sec:triple quiver simple framing}.
Recall that $\widetilde\calZ(W)$
and $\widehat\calZ(W)$ are the corresponding Steinberg varieties.
The group $G_W\times T$ acts on $\widetilde\calZ(W)$
and $\widehat\calZ(W)$ diagonally.
We say that a quasi-projective variety $X$ with an action of an affine group $G$ 
satisfies the property ($T$) if the following hold
\begin{itemize}[leftmargin=3mm]
\item[-]
$K^G(X)$ is a free $R_G$-module, 
\item[-] the forgetful map 
$K^G(X)\otimes_{R_G}R_H\to K^H(X)$
is invertible for all closed subgroup $H\subset G$.
\end{itemize}

\begin{lemma}\label{lem:(T)} 
\hfill
\begin{enumerate}[label=$\mathrm{(\alph*)}$,leftmargin=8mm]
\item
The $G_W\times T$-varieties
$\widehat\frakM(W)$ and $\widehat\calZ(W)$ satisfy the property $(T)$.
\item
The $G_W\times T$-varieties
$\widetilde\frakM(W)$ and $\widetilde\calZ(W)$ satisfy the property $(T)$.
\end{enumerate}
\end{lemma}

\begin{proof}
The variety $\widehat\frakM(v,W)$ parametrizes
the conjugacy classes of pairs consisting of a $(v,v)$-matrix $\varepsilon$ and a $w$-tuple of generators
of $\bbC^v$ for the $\varepsilon$-action. 
Thus
\begin{align}\label{quot}
\widehat\frakM(v,W)\cong\Quot(W\otimes\calO_{\bbA^1}, v)
\end{align}
is isomorphic to the Quot scheme parametrizing length $v$-quotients of the locally free sheaf 
$W\otimes\calO_{\bbA^1}$ over $\bbA^1$, and
$$\widehat\frakM_0(v,W)\cong\bbC^v,$$
the map $\pi:\widehat\frakM(v,W)\to\widehat\frakM_0(v,W)$ being given by the support.
In order to use the description of $\widehat\frakM(v,W)$ in terms of Quot schemes above, it is convenient to 
view it as the moduli space of pairs $(a^*,\varepsilon)$ rather than of pairs $(a,\varepsilon)$.
This does not affect the rest of the arguments.
The group $G_W$ acts on $W$ in the obvious way, and 
$\bbC^\times$ dilates both the framing and the matrix $\varepsilon$.
The variety $\widehat\frakM(v,W)$ is smooth. 
Fix a splitting $W=\bigoplus_{r=1}^wW_r$ of $W$ as a sum of lines.
Let $T_W\subset G_W$ be the corresponding maximal torus.
Let $\lambda:\bbC^\times\to T_W$ be the cocharacter $z\mapsto(z,z^2,\dots,z^w)$.
The $T_W$-fixed point locus is the disjoint union of the varieties
$$\widehat\frakM(\bfv,W)=\Quot(W\otimes\calO_{\bbA^1}, \bfv)=
\prod_{r=1}^w\Quot(W_r\otimes\calO, v_r)=\prod_{r=1}^w\bbC^{[v_r]}=\bbC^v$$
where $\bfv=(v_1,v_2,\cdots,v_w)$ runs into the set of tuples in $\bbN^w$ with sum $v$,
and $\bbC^{[v_r]}$ is the $v_r$-fold symmetric product of $\bbC$.
See, e.g., \cite[prop.~3.1]{MR22}.
The closed embedding 
$$\Quot(W\otimes\calO_{\bbA^1}, \bfv)\subset\Quot(W\otimes\calO_{\bbA^1}, v)$$
is given by the direct sum of the $\calO_{\bbA^1}$-modules of lengths $v_1,\dots,v_w$.
The Byalinicki-Birula theorem applied to the cocharacter $\lambda$ 
yields a $T_W\times T$-equivariant stratification
\begin{align}\label{cells}
\Quot(W\otimes\calO_{\bbA^1}, v)=\bigsqcup_\bfv\Quot(W\otimes\calO_{\bbA^1}, \bfv)^+
\end{align}
where, each cell  is an affine fiber bundle
$$\Quot(W\otimes\calO_{\bbA^1}, \bfv)^+\to
\Quot(W\otimes\calO_{\bbA^1}, \bfv)=\widehat\frakM(\bfv,W)$$ 
of relative dimension $\sum_{r=1}^w(r-1)v_r$.
See \cite[prop.~3.4]{MR22}.
Set 
$$\widehat\frakM(\bfv,W)^+=\Quot(W\otimes\calO_{\bbA^1}, \bfv)^+.$$
This yields a $T_W\times T$-equivariant stratification 
of $\widehat\frakM(v,W)$ by affine cells $\widehat\frakM(\bfv,W)^+$ of dimension $\sum_{r=1}^wrv_r$.
The first claim of Part (a) follows using \cite[thm.~6.1.22]{CG}.

The proof of the second claim in Part (a) is similar. Recall that 
\begin{align}\label{ZMM}\widehat\calZ(W)=\widehat\frakM(W)\times_{\widehat\frakM_0(W)}\widehat\frakM(W)
\end{align}
and that $\widehat\frakL(v,W)$ is the 0-fiber of the map $\pi$.
The isomorphism \eqref{quot} identifies the variety
$\widehat\frakL(v,W)$ with the punctual Quot scheme 
$$\widehat\frakL(v,W)\cong\Quot(W\otimes\calO_{\bbA^1}, v)_0$$
consisting of the sheaves on $\bbA^1$ supported at $0$.
Intersecting the cell decomposition \eqref{cells} with $\Quot(W\otimes\calO_{\bbA^1}, v)_0$
yields the affine cells
\begin{align*}
\widehat\frakL(\bfv,W)^+=\Quot(W\otimes\calO_{\bbA^1}, \bfv)_0^+
=\Quot(W\otimes\calO_{\bbA^1}, v)_0\cap \Quot(W\otimes\calO_{\bbA^1}, \bfv)^+
\end{align*}
and the cell decomposition 
\begin{align*}
\widehat\frakL(v,W)
=\bigsqcup_\bfv\widehat\frakL(\bfv,W)^+\end{align*}
More precisely,
for each tuple $\bfv$ as above we fix a flag of vector spaces
$V_{\geqslant w}\subset V_{\geqslant w-1}\subset\dots\subset V_{\geqslant 1}=V$
of dimensions $v_w,w_{w}+v_{w-1},\dots, v$.
Let $P_\bfv\subset G_V$ be the parabolic subgroup which fixes this flag.
The Lie algebra of $P_\bfv$ is $\frakp_\bfv$, and the Levi factor of $\frakp_\bfv$ is $\frakg_\bfv=\bigoplus_{r=1}^w\frakg_{v_r}$.
Let $\E_\bfv\subset\Hom(W,V)$ be the subspace of linear maps $a^*$ such that $a^*(W_r)\subset V_{\geqslant r}$ for each
$r=1,2,\dots,w$.
We have the following description of the cells in 
$\widehat\frakM(v,W)$ and $\widehat\frakL(v,W)$:
\begin{align*}
\widehat\frakM(v,W)&\cong\{(a^*,\varepsilon)\in\Hom(W,V)\times\frakg_V\,;\,(a^*,\varepsilon)\ \text{is\ stable}\}/G_V,\\
\widehat\frakL(v,W)&\cong\{(a^*,\varepsilon)\in\Hom(W,V)\times\frakg_V^\nil\,;\,(a^*,\varepsilon)\ \text{is\ stable}\}/G_V,\\
\widehat\frakM(\bfv,W)^+&\cong\{(a^*,\varepsilon)\in\E_\bfv\times\frakp_\bfv\,;\,(a^*,\varepsilon)\ \text{is\ stable}\}/{P_\bfv},\\
\widehat\frakL(\bfv,W)^+&\cong\{(a^*,\varepsilon)\in\E_\bfv\times\frakp_\bfv\cap\frakg_V^\nil\,;\,(a^*,\varepsilon)\ \text{is\ stable}\}/{P_\bfv}.
\end{align*}
There are obvious embeddings $\widehat\frakL(\bfv,W)^+\subset\widehat\frakM(\bfv,W)^+$.
Since the diagonal blocks
of the matrix $\varepsilon$ are regular matrices in $\frakg_{v_1},\frakg_{v_2},\dots,\frakg_{v_w}$
for each pair $(a^*,\varepsilon)$ in $\widehat\frakM(\bfv,W)^+$,
there is an obvious splitting 
$\widehat\frakM(\bfv,W)^+\cong\widehat\frakL(\bfv,W)^+\times\bbC^v$ yieldding a Cartesian square
\begin{align*}
\xymatrix{
\widehat\frakM(\bfv,W)^+\ar[d]_-{pr_1}\ar[r]^-{pr_2}&\widehat\frakM_0(W)\ar[d]\\
\widehat\frakL(\bfv,W)^+\ar[r]&\{0\}
}
\end{align*}
where $pr_1$, $pr_2$ are the obvious projections and $\widehat\frakM_0(W)$ is identified with $\bbC^v$.
Using \eqref{ZMM}, we deduce that $\widehat\calZ(W)$ has a $T_W\times T$-equivariant 
decomposition 
$$\widehat\calZ(W)=\bigsqcup_{\bfv_1,\bfv_2}
\widehat\frakM(\bfv_1,W)^+\times_{\widehat\frakM_0(W)}\widehat\frakM(\bfv_2,W)^+$$
as a disjoint union of the affine cells
$$\widehat\frakM(\bfv_1,W)^+\times_{\widehat\frakM_0(W)}\widehat\frakM(\bfv_2,W)^+\cong
\widehat\frakL(\bfv_1,W)^+\times\widehat\frakL(\bfv_2,W)^+\times\bbC^v.$$
This proves Part (a). Part (b) is similar. Note that forgetting the arrow $a^*$ yields a vector bundle
$\widetilde\frakM(W)\to\widehat\frakM(W)$.
\end{proof}

The cocharacter $\xi$ of $T$ in \S\ref{sec:triple quiver} yields an embedding
$$T_W\times\bbC^\times\subset T_W\times T.$$
For each tuple $\lambda\in\bbN^w$ let $|\lambda|=\sum_{s=1}^w\lambda_s$
be the weight of $\lambda$.

\begin{lemma}\label{lem:fixedpoints2}
The sets of $T_W\times T$-fixed points in $\widetilde\frakM(W)$
and of $T_W\times \bbC^\times$-fixed points in $\widehat\frakM(W)$ are finite.
We have
\begin{align}\label{fixpoints}\widetilde\frakM(v,W)^{T_W\times T}
=\widehat\frakM(v,W)^{T_W\times\bbC^\times}=
\{\underline x_\lambda\,;\,\lambda\in\bbN^w\,,\,|\lambda|=v\}.
\end{align} 
\end{lemma}

\begin{proof}
The fixed points in
$\widehat\frakM(W)^{T_W\times\bbC^\times}$
and
$\widetilde\frakM(W)^{T_W\times T}$
are both identified with the points of a framed graded triple quiver variety as in Lemma \ref{lem:fixedpoints}.
More precisely, let $w=\dim W$.
Then, each fixed point is identified
with a direct sum of $w$ representations of a quiver as in the figure below.
The dimension vector of each summand is $(1,1,\dots,1)$.
The lengths (=the number of $\varepsilon$'s) of the summands define a tuple
$\lambda=(\lambda_1,\lambda_2,\dots,\lambda_w)$ in $\bbN^w$ such that
the weight $|\lambda|$ is equal to the integer $v$ such that
$x\in\widetilde\frakM(v,W)$ or $\widehat\frakM(v,W)$. 
This proves the lemma.
\begin{figure}[htbp]
    \centering
\setlength{\unitlength}{1mm}
\begin{picture}(90,24)
    \multiput(18.5,3.5)(15,0){4}{\circle{7}}
    \multiput(30,3.5)(15,0){3}{\thicklines\vector(-1,0){8}}
    \put(20.7,5.9){\thicklines\vector(.5,1){4.7}}
      \put(20.4,9.8){$\scriptstyle a$}
    \put(22.5,15.4){\framebox(6,6){$k$}}
    \put(15.5,2.7){$\scriptstyle k+1$}
    \put(30.5,2.7){$\scriptstyle k-1$}
    \put(45.5,2.7){$\scriptstyle k-3$}
     \put(60.5,2.7){$\scriptstyle k-5$}
      \put(26,4.4){$\scriptstyle\varepsilon$}
       \put(41,4.4){$\scriptstyle\varepsilon$}
       \put(56,4.4){$\scriptstyle\varepsilon$}
\end{picture}
\caption{}
\end{figure}
\end{proof}


\section{Cohomological critical convolution algebras}\label{sec:cc}

The goal of this appendix is to introduce the critical cohomological convolution algebras.
They should be viewed as some doubles of the Kontsevich-Soibelmann cohomological Hall algebras
introduced in \cite{KS11}.

\subsection{Vanishing cycles and LG-models}
Let $G$ be an affine group acting on a smooth manifold $X$.
Let $\D^\b_G(X)$ be the $G$-equivariant derived category of
constructible complexes with complex coefficients on $X$.
Given a function $f:X\to\bbC$ with zero locus $Y=f^{-1}(0)$, 
we have the vanishing cycle and nearby cycle functors 
$\phi_f,\psi_f:\D^\b_G(X)\to \D^\b_G(Y).$ 
Let $i:Y\to X$ be the obvious embedding. 
Set
$\phi^p_{\!f}=i_*\phi_f[-1]$ and
$\psi^p_{\!f}=i_*\psi_f[-1].$
The functors $\phi^p_{\!f}$, $\psi^p_{\!f}$ commute with the Verdier duality $\bbD$.
They take perverse sheaves to perverse sheaves.
We have a distinguished triangle
\begin{align}\label{triangle}
\xymatrix{\psi^p_{\!f}\calE\ar[r]^{\can}&\phi^p_{\!f}\calE\ar[r]&i_*i^*\calE\ar[r]^{+1}& }
\end{align}
Let $(X,f)$ be a smooth $G$-invariant LG-model.
Let $i:Y\to X$ be the embedding of the zero locus of $f$,
and $j:Z\to X$ the embedding of a closed $G$-invariant subset of $Y$.
For any constructible complex $\calE\in\D^\b_G(X)$ we set
$H^\bullet_{Z}(X,\calE)=H^\bullet_G(Z,j^!\calE).$
Let 
$\calC_X=\bbC_X[\dim X]$
and
$$H^\bullet_G(X,f)_Z=
H^\bullet_{Z}(X,\phi^p_{\!f}\calC_X).$$
Let $\phi:(X_2,f_2)\to (X_1,f_1)$ be a morphism of smooth $G$-invariant LG-models.
Let $Y_1=(f_1)^{-1}(0)$ and $Y_2=(f_2)^{-1}(0)$.
Let $Z_1,$ $Z_2$ be closed $G$-invariant subsets of $Y_1,$ $Y_2$.
By \cite[\S 2.17]{D17} we have the following functoriality maps.
If $\phi^{-1}(Z_1)\subset Z_2$ then we have a pull-back map
$\phi^*:H^\bullet_G(X_1,f_1)_{Z_1}\to H^\bullet_G(X_2,f_2)_{Z_2}$ which
is an isomorphism if $\phi$ is an affine fibration.
If $\phi(Z_2)\subset Z_1$ and $\phi|_{Z_2}$ is proper
then we have a push-forward map
$\phi_*:H^\bullet_G(X_2,f_2)_{Z_2}\to H^\bullet_G(X_1,f_1)_{Z_1}.$

\subsection{Cohomological critical convolution algebras}

Let $(X_a,f_a)$ be a smooth $G$-invariant LG-model for $a=1,2,3$.
We define $X_{ab}$, $Y_{ab}$, $Z_{ab}$, $f_{ab}$, $\pi_{ab}$  as in \S\ref{sec:Kcritalg}.
The Thom-Sebastiani isomorphism yields a map
$$\boxtimes:H^\bullet_G(X_{12},f_{12})_{Z_{12}}\otimes H^\bullet_G(X_{23},f_{23})_{Z_{23}}\to
H^\bullet_G(X_{12}\times X_{23},f_{12}\oplus f_{23})_{\pi_{12}^{-1}(Z_{12})\cap\pi_{23}^{-1}(Z_{23})}.$$
We define the convolution product 
\begin{align}\label{conv4}
\star:H^\bullet_G(X_{12},f_{12})_{Z_{12}}\otimes H^\bullet_G(X_{23},f_{23})_{Z_{23}}\to
H^\bullet_G(X_{13},f_{13})_{Z_{13}}
\end{align}
to be the linear map such that
$\alpha\otimes\beta\mapsto(\pi_{13})_*(\pi_{12}\times\pi_{23})^*(\alpha\boxtimes\beta).$
Next, we consider the following particular setting where
$\pi:X\to X_0$ is a proper morphism of $G$-schemes with 
$X$ smooth quasi-projective and $X_0$ affine,
$f_0:X_0\to\bbC$ is a invariant function, and $f=f_0\circ\pi$ is regular.
Let $Y$, $Y_0$, $Z$, $L$ and $f^{\boxplus 2}$ be as in \S\ref{sec:Kcritalg}.
We set $X_a=X$ and $f_a=f$ for each $a=1,2,3$. 
We equip the $H^\bullet_G$-module
$\Ext^\bullet_{\D^\b_G(X_0)}(\phi^p_{\!f_0}\pi_*\calC_X, \phi^p_{\!f_0}\pi_*\calC_X)$
with the Yoneda product.

\begin{proposition}\label{prop:critalg3}
\hfill
\begin{enumerate}[label=$\mathrm{(\alph*)}$,leftmargin=8mm]
\item 
There is an isomorphism
$H^\bullet_G(X^2,f^{(2)})_Z=
\Ext^\bullet_{\D^\b_G(X_0)}(\phi^p_{\!f_0}\pi_*\calC_X, \phi^p_{\!f_0}\pi_*\calC_X)$
which intertwines the convolution product and the Yoneda product.
\item
The convolution product equips $H^\bullet_G(X^2,f^{(2)})_Z$ with 
an $H^\bullet_G$-algebra structure.
\item 
The $H^\bullet_G$-algebra $H^\bullet_G(X^2,f^{(2)})_Z$ acts on the $H^\bullet_G$-modules 
$H^\bullet_G(X,f)_L$ and $H^\bullet_G(X,f)$.
\end{enumerate}
\end{proposition}

\begin{proof}
Parts (b), (c) follow from (a).
The isomorphism in Part (a) is 
\begin{align*}
H^\bullet_G(X^2,f^{(2)})_Z
&=H^\bullet(Z,j^!\phi^p_{\!f^{(2)}}\calC_{X^2})\\
&=H^\bullet(Z,j^!(\phi^p_{\!f}\calC_X\boxtimes\phi^p_{\!f}\calC_X))\\
&=H^\bullet(Z,j^!(\bbD\phi^p_{\!f}\calC_X\boxtimes\phi^p_{\!f}\calC_X))\\
&=\Ext^\bullet_{\D^\b_G(X_0)}(\pi_*\phi^p_{\!f}\calC_X, \pi_*\phi^p_{\!f}\calC_X)\\
&=\Ext^\bullet_{\D^\b_G(X_0)}(\phi^p_{\!f_0}\pi_*\calC_X, \phi^p_{\!f_0}\pi_*\calC_X)
\end{align*}
where the second isomorphism follows from the Thom-Sebastiani theorem and the inclusion
$\crit(f)\subset f^{-1}(0)$, the third one follows from the self-duality of the complex $\phi^p_{\!f}\calC_X$,
the fourth equality is as in \cite[(8.6.4)]{CG}, and the last one is the commutation of 
proper direct image and vanishing cycles.
The compatibility under the isomorphism in (b) of the convolution product in $H^\bullet_G(X^2,f^{(2)})_Z$
and the Yoneda composition in 
$\Ext^\bullet_{\D^\b_G(X_0)}(\phi^p_{\!f_0}\pi_*\calC_X, \phi^p_{\!f_0}\pi_*\calC_X)$
follows from \cite[\S 8.6.27]{CG}, modulo observing that the convolution product \cite[(8.6.27)]{CG}
is the same as the convolution product \eqref{conv4}.
\end{proof}

\begin{remark}\hfill
\begin{enumerate}[label=$\mathrm{(\alph*)}$,leftmargin=8mm]
\item 
If $f_{ab}=0$, then there is an $H^\bullet_G$-module isomorphism
$$H^\bullet_G(X_{ab},f_{ab})_{Z_{ab}}
=H^\bullet_G(Z_{ab},\bbD_{Z_{ab}})[-\dim X_{ab}]
=H^G_{-\bullet}(Z_{ab},\bbC)[-\dim X_{ab}]$$
where $\bbD_{Z_{ab}}$ is the dualizing complex.
Under this isomorphism the convolution product \eqref{conv4} is the same as the
convolution product in equivariant Borel-Moore homology used in \cite[\S 2.7]{CG}.
In particular, if $f=0$ then there is an algebra isomorphism
$H^\bullet_G(X^2,f^{(2)})_Z=H_\bullet^G(Z,\bbC)$, up to a grading renormalization.
The algebra isomorphism in Proposition \ref{prop:critalg3}(b) is the algebra isomorphism
$H_\bullet^G(Z,\bbC)=\Ext^\bullet_{\D^\b_G(X_0)}(\pi_*\calC_X, \pi_*\calC_X)$
in \cite [thm.~8.6.7]{CG}.
\item
The functoriality of $\phi^p_{\!f_0}$ yields 
an algebra homomorphism 
\begin{align*}\Upsilon:H_\bullet^G(Z,\bbC)\to H^\bullet_G(X^2,f^{(2)})_Z\end{align*}
which is an analog of the algebra homomorphism in Proposition \ref{prop:critalg1}(d).
Here $H_\bullet^G(Z,\bbC)$ is given the convolution product induced by the closed embedding of $Z$ in  the 
smooth variety $X^2$ as in \cite{CG}.
\end{enumerate}
\end{remark}

Now, let us consider the case of graded quiver varieties.

\begin{corollary}
\hfill
\begin{enumerate}[label=$\mathrm{(\alph*)}$,leftmargin=8mm]
\item 
$H^\bullet(\widetilde\frakM^\bullet(W)^2,(\tilde f_\gamma^\bullet)^{(2)})_{\widetilde\calZ^\bullet(W)}$ 
is an algebra which acts on 
$H^\bullet(\widetilde\frakM^\bullet(W),\tilde f_\gamma^\bullet)$ and
$H^\bullet(\widetilde\frakM^\bullet(W),\tilde f_\gamma^\bullet)_{\widetilde\frakL^\bullet(W)}$.
\item
There is an algebra homomorphism
$\U_\zeta(L\frakg)\to 
H^\bullet\big(\widetilde\frakM^\bullet(W)^2,(\tilde f^\bullet_\gamma)^{(2)}\big)_{\widetilde\calZ^\bullet(W)}$,
and representations of $\U_\zeta(L\frakg)$ on 
$H^\bullet(\widetilde\frakM^\bullet(W),\tilde f^\bullet_\gamma)_{\widetilde\frakL^\bullet(W)}$ and
$H^\bullet(\widetilde\frakM^\bullet(W),\tilde f^\bullet_\gamma)$.
\end{enumerate}
\end{corollary}

\begin{proof}
The proof is similar to the proof of Corollary \ref{cor:notshifted2}.
The details will be given elsewhere.
If the Cartan matrix is symmetric, a proof using Nakajima's work is given in 
\cite[thm.~4.3]{VV23}.
\end{proof}

\section{The algebraic and topological critical K-theory}\label{sec:Ktheory}

In this appendix we discuss some topological analogues of the Grothendieck groups following 
\cite{B16} and \cite{HP20}.
To do this, for $\flat=\alg$ or $\top$ We use the functor $\bfK^\flat$
from the category of all dg-categories over $\bbC$ 
to the category of spectra introduced in \cite{S06} and \cite{B16}.
Let $\overline\calC$ denote the idempotent completion of an additive category $\calC$.
Let $X^\an$ be the underlying complex analytic space of a scheme $X$.
Given a closed subset $Y$ of $X$ We say that $Y^\an$ is homotopic to $X^\an$, or that
$Y$ is homotopic to $X$, if the inclusion $Y^\an\subset X^\an$ admits a deformation retraction 
$X^\an\to Y^\an$.
The following properties hold:
\begin{itemize}[leftmargin=3mm]
\item[-]
$\bfK^\alg(\calC)$ is the algebraic $K$-theory spectrum of the 
category $H^0(\overline\calC)$,
\item[-]
there is natural topologization map $\top:\bfK^\alg\to \bfK^\top$,
\item[-]
$\bfK^\flat$ takes localization sequences of dg-categories
to exact triangles.
\end{itemize}
Let $(X,\chi,f)$ be a $G$-equivariant LG-model.
Let $Y\subset X$ be the zero locus of $f$.
Let $i$ be the closed embedding $Y\to X$, and
$Z\subset Y$ a closed $G$-invariant subset.
The triangulated categories
$\Db\Coh_G(Y)_Z$, $\Perf_G(Y)_Z$ and $\DCoh_G(X,f)_Z$
admit dg-enhancements, and all derived functors above admit also dg-enhancements.
We write
\begin{align*}
\bfK^G_\flat(X)_Z=\bfK^\flat(\Db\Coh_G(X)_Z)
 ,\quad
\bfK_G^\flat(X)=\bfK^\flat(\Perf_G(X))
,\quad
\bfK_\flat^G(X)=\bfK_\flat^G(X)_X.
\end{align*}
The following properties hold:
\begin{itemize}[leftmargin=3mm]
\item[-]
$\bfK_\flat^G$ is covariantly functorial for proper morphisms of
$G$-schemes, and contravariantly functorial for
 finite $G$-flat dimensional morphisms,
 \item[-]
$\bfK_\flat^G$ satisfies the flat base change and the projection formula,
\item[-]
$\bfK_\flat^G$ satisfies equivariant d\'evissage, i.e., there is a weak equivalence
$\bfK_\flat^G(Z)\to\bfK_\flat^G(X)_Z$,
\item[-]
$\bfK_\top^G(X)$ is the 
$G$-equivariant Borel-Moore K-homology spectrum of $X^\an$, and
$\bfK^\top_G(X)$ is its 
$G$-equivariant K-theory spectrum, up to weak equivalences.
\end{itemize}

The Grothendieck groups $K^G(Z)$ and $K_G(Z)$
satisfy
$$K_G(Z)=\pi_0\bfK_G^\alg(Z)\otimes\bbC
,\quad
K^G(Z)=\pi_0\bfK^G_\alg(Z)\otimes\bbC.$$
The $G$-equivariant Borel-Moore K-homology of $X$ 
and its $G$-equivariant K-theory are 
\begin{align}\label{Ktop1}
K_G^\top(Z)=\pi_0\bfK_G^\top(Z)\otimes\bbC
 ,\quad
K^G_\top(Z)=\pi_0\bfK^G_\top(Z)\otimes\bbC
\end{align}

We define
\begin{align}\label{Ktop2}
\begin{split}
K_G(X,f)_Z&=K_0(\DCoh_G(X,f)_Z),\\
K_G^\top(X,f)_Z&=\pi_0\bfK^\top(\DCoh_G(X,f)_Z)\otimes\bbC,\\
K_G^\alg(X,f)_Z&=\pi_0\bfK^\alg(\DCoh_G(X,f)_Z)\otimes\bbC.
\end{split}
\end{align}
Compare \eqref{Kcrit}.
By \cite[cor.~2.3]{T97} there is an inclusion
$K_G(X,f)_Z\subset K_G^\alg(X,f)_Z$,
because 
$K_G(X,f)_Z$ is the Grothendieck group of $\DCoh_G(X,f)_Z$ while
$K_G^\alg(X,f)_Z$ is the Grothendieck group of its idempotent completion, and
$\DCoh_G(X,f)_Z$ is dense in $\overline{\DCoh_G(X,f)_Z}$.
Now, let $X,Y,Z,L$ be as in \S\ref{sec:Kcritalg}.
The functor \eqref{conv1} yields 
an associative $R_G$-algebra structure on 
$K^G_\flat(Z)$ and a representation on $K^G_\flat(L)$ and $K^G_\flat(X)$.
The functor \eqref{conv2} yields the following.

\begin{proposition}\label{prop:critalg2}
\hfill
\begin{enumerate}[label=$\mathrm{(\alph*)}$,leftmargin=8mm]
\item
$K_G^\flat(X^2,f^{(2)})_Z$ is an $R_G$-algebra which acts on
$K_G^\flat(X,f)_L$ and $K_G^\flat(X,f)$.

\item
The functor $\Upsilon$ yields an algebra homomorphism 
$K^G_\flat(Z)\to K_G^\flat(X^2,f^{(2)})_Z$.

\item
The functor $\Upsilon$ yields maps
$K^G_\flat(Y)\to K_G^\flat(X,f)$
and
$K^G_\flat(L)\to K_G^\flat(X,f)_L$
which intertwine the actions of the algebras
$K^G_\flat(Z)$ and $K_G^\flat(X^2,f^{(2)})_Z$
under the algebra homomorphism  in ${\operatorname{(b)}}$.
\qed
\end{enumerate}
\end{proposition}

\section{Notation list}

\ref{sec:Dcrit}: $(X,\chi,f)$, $\Coh_G(X,f)$, $\Coh_G(X,f)_Z$, $\Coh_G(X,f)_\uZ$, $\DCoh_G(X,f)$, $\DCoh_G(X,f)_Z$,
\vskip2mm
 $\DCoh_G(X,f)_\uZ$,
\\

\ref{sec:Upsilon}: $\Upsilon$, $\DCoh_G^\sg(Y)_Z$, $\DCoh_G^\sg(Y)_\uZ$,
\\

\ref{sec:critK}: $K_G(X,f)_Z$, $K_G(X,f)_\uZ$,
\\

\ref{sec:Kcritalg}: $f^{(2)}$,
\\

\ref{sec:quiver basic1}: $Q=(Q_0,Q_1)$, $\overline Q$, $\widetilde Q$, $Q_f$, $\overline Q_f$, $\widetilde Q_f$, $\widehat Q_f$, $Q^\bullet$,
$x=(\alpha,a,a^*,\varepsilon)$,
$\overline\X$, $\widetilde\X$, $\widehat\X$, $\overline\X^\bullet$, $\widetilde\X^\bullet$, $\widetilde\X^\nil$, $\widehat\X^\nil$, 
\\

\ref{sec:triple quiver}: $R_T$, $G_V$, $\frakg_V$, $\widetilde\frakM$, $\widetilde\frakM_0$, $\widetilde\frakL$, $\widetilde\calZ$,
\\

\ref{sec:or}: $\bfc$, $\O$, 
\\

\ref{sec:GQVD}: $\widetilde\frakM^\bullet$, $\widetilde\frakM_0^\bullet$, $\widetilde\frakL^\bullet$,
\\

\ref{sec:hat}: $\widehat\X$, $\widehat\frakM$, $\widehat\frakL$, $\widehat\calZ$, 
\\

\ref{sec:Hecke}: $\widetilde\frakP$, $\widehat\frakP$, $\frakR$, 
\\

\ref{sec:potential}: $\bfw_1$, $\bfw_2$, $\bfw_1^\bullet$, $\bfw_2^\bullet$, $\tilde f_1$, $\tilde f_2$, $\hat f_1$, $\hat f_2$, 
$\tilde f_1^\bullet$, $\tilde f_2^\bullet$, $\hat f_1^\bullet$, $\hat f_2^\bullet$,$h$, $\pi$, $i$,
\\

\ref{sec:univbdl}: $\calV$, $\calW$, $\calV_i$, $\calW_i$, $\calV_{\circ i}$, $\calV_{-i}$, $\calV_i^+$, $\calV_i^-$, $\calL_i$, 
$v_{\circ i}$, $v_{+i}$, $v_{-i}$,
\\

\ref{sec:QGr}: $\widetilde\Pi$, $\widetilde\Pi_l$, $\widetilde\Pi^\bullet$, $\omega$,
$\tau$, $H$, $H_l$, $\bfD$, $\bfD^\bullet$, $\bfD^\nil$, $\bfD^{\bullet,\nil}$, 
$K_{i,k,l}$, $I_i$, $I_{i,k}$, $\widetilde\Gr(M)$, $\widetilde\Gr^\bullet(M)$, 
\\

\ref{sec:N00}: $\U_F(L\frakg)$, $\U_R(L\frakg)$, $\U_\zeta(L\frakg)$, $\U_F^{-w}(L\frakg)$, $\U_R^{-w}(L\frakg)$, $\U_\zeta^{-w}(L\frakg)$, 
\\

\ref{sec:w1}: $f_\gamma$, $f_\gamma^\bullet$, $\tilde f_\gamma$,
\\

\ref{sec:421}: $x_{i,n}^\pm$, $\psi_{i,n}^\pm$, $\calH_{i,\pm m}$, $\Upsilon$, $\Delta$, $A_{i,n}^\pm$,
\\

\ref{sec:A6=}: $\frakM(v,W)_\heartsuit$, $\frakM(v,W)_\spadesuit$, $\frakM(W)_\diamondsuit$, $\calZ(W)_\diamondsuit$, $G_\diamondsuit$,
$\calP_i$, 
\\

\ref{sec:423}: $I_{v_1,v_2,v_3}$, $I_{v_1,v_4,v_3}$, 
\\

\ref{sec:425}: $I_{v_1,v_2,v_3}$, $I_{v_1,v_4,v_3}$, 
\\

\ref{sec:HL1}: $W_{i,k,l}$, $w_{i,k,l}$, $\gamma_{i,k,l}$, $\gamma$, $K_\gamma$,  $KR_\gamma$, $\diamond$,
\\

\ref{sec:sqg}: $[m]_q$, $q_i$, $g_{ij}(u)$, $\delta(u)$, $\bfO_w$, $V_\Psi$, $q\!\ch$, $L_{i,a}^\pm$, $\Psi_{\pm i, a}$,
\\

\ref{sec:qg}: $h_{i,m}$, $L(\Psi)$, $P_w$, $\Psi_w$, $KR_{i,k,l}$, $|v|$, $Y_{i,k}$, $A_{i,k}$,
\\

\end{document}